\documentclass[12pt]{amsart}
\usepackage{amssymb,amsmath,tabularx,mathrsfs,mathbbol,yfonts,upgreek}
\usepackage{amsthm,verbatim,comment}
\usepackage{amsmath}
\usepackage[bookmarks=true]{hyperref}
\usepackage{pstricks,pst-node,pst-plot}
\usepackage{geometry}
\usepackage{stmaryrd}
\usepackage{paralist}
\usepackage[all]{xy}
\usepackage{array}
\usepackage{setspace}
\usepackage{ytableau}
\usepackage{lipsum}
\numberwithin{equation}{section}

\DeclareSymbolFontAlphabet{\mathbb}{AMSb}
\DeclareSymbolFontAlphabet{\mathbbl}{bbold}

\newtheorem{thm}{Theorem}[section]
\newtheorem{thmx}{Theorem}

\newtheorem{lem}[thm]{Lemma}
\newtheorem{prop}[thm]{Proposition}
\newtheorem{cor}[thm]{Corollary}

\newtheorem*{ques*}{Question}
\theoremstyle{definition}

\newtheorem{nota}[thm]{Notation}
\newtheorem{eg}[thm]{Example}
\newtheorem{rem}[thm]{Remark}
\theoremstyle{remarks}

\newtheorem*{rem*}{Remarks}

\newtheoremstyle{case}{}{}{}{}{}{:}{ }{}
\theoremstyle{case}

\newcommand{\Tr}{\mathrm{Tr}}

\newcommand{\sym}[1]{\mathfrak{S}_{#1}}

\renewcommand{\O}{\mathcal{O}}

\newcommand{\B}{\mathcal{B}}

\newcommand{\F}{\mathbb{F}}

\newcommand\blfootnote[1]{%
  \begingroup
  \renewcommand\thefootnote{}\footnote{#1}%
  \addtocounter{footnote}{-1}%
  \endgroup
}
\title{On the symmetric and exterior powers of Young permutation modules}
\begin{document}
\author{Yu Jiang}
\address[Y. Jiang]{Division of Mathematical Sciences, Nanyang Technological University, SPMS-MAS-05-34, 21 Nanyang Link, Singapore 637371.}
\email[Y. Jiang]{jian0089@e.ntu.edu.sg}


\begin{abstract}
Let $n$ be a positive integer and $\lambda$ be a partition of $n$. Let $M^\lambda$ be the Young permutation module labelled by $\lambda$. In this paper, we study symmetric and exterior powers of $M^\lambda$ in positive characteristic case. We determine the symmetric and exterior powers of $M^\lambda$ that are projective. All the indecomposable exterior powers of $M^\lambda$ are also classified. We then prove some results for indecomposable direct summands that have the largest complexity in direct sum decompositions of some symmetric and exterior powers of $M^\lambda$. We end by parameterizing all the Scott modules that are isomorphic to direct summands of the symmetric or exterior square of $M^\lambda$ and determining their corresponding multiplicities explicitly.
\end{abstract}
\maketitle
\noindent
\section{Introduction}
Let $G$ be a finite group and fix a field $\F$ of positive characteristic $p$. The complexity of a finite-dimensional $\F G$-module $M$, defined by Alperin and Evens in \cite{JAlperin} and denoted by $c_G(M)$, describes the growth rate of a minimal projective resolution of $M$. However, the invariant is usually very difficult to compute and is still poorly understood even in the case of group algebras of symmetric groups.

As usual, the symmetric and exterior powers of $M$ are natural $\F G$-modules by the diagonal action of $G$. In \cite[Lemma 1 (i), (ii)]{BensonLim}, the authors obtained an optimal upper bound for the complexities of symmetric and exterior powers of $R$, where $R$ is a direct sum of some copies of the regular module $\F G$. Given a positive integer $a$, if $a>1$, let $S^aM$ and $\Lambda^aM$ be the $a$th symmetric and exterior power of $M$ respectively. Let $\nu_p(a)$ be the largest non-negative integer $b$ such that $p^b\mid a$. Using the Brauer quotients of $p$-permutation $\F G$-modules, We generalize their result as follows.
\begin{thmx}\label{T;A}
Let $G$ be a finite group with $p$-rank $r$ and $P$ be a non-zero projective $\F G$-module. Let $a$ be a positive integer and $a>1$. Then $c_G(S^aP)=\min\{\nu_p(a),r\}$. Moreover, if $a\leq\dim_\F P$, then $c_G(\Lambda^aP)=\min\{\nu_p(a),r\}$.
\end{thmx}
For the modules of a symmetric group, notice that the class of Young permutation modules contains the regular module. Following the notation of \cite{GJ1}, let $\sym{n}$ be the symmetric group on $n$ letters and write $\lambda\vdash n$ to imply that $\lambda$ is a partition of $n$. Let $\lambda\vdash n$ and $M^\lambda$ be the $\F \sym{n}$-Young permutation module labelled by $\lambda$. Theorem \ref{T;A} motivates us to study $S^aM^\lambda$ and $\Lambda^aM^\lambda$. We determine the symmetric and exterior powers of $M^\lambda$ that are projective (see Propositions \ref{P;symmetricproj}, \ref{P;exteriorproj}). As another main result of this paper, all the indecomposable exterior powers of $M^\lambda$ are also classified (see Theorem \ref{T;indecomposable} and Remark \ref{R;Remark} (i)).

Given $\F G$-modules $M$ and $N$, write $N\mid M$ to indicate that $N$ is isomorphic to a direct summand of $M$. Let $Y^\lambda$ be the $\F\sym{n}$-Young module labelled by $\lambda$. A well-known fact tells us that there exists some $\mu\vdash n$ such that $Y^\mu\mid M^\lambda$ and
$c_{\sym{n}}(Y^\mu)=c_{\sym{n}}(M^\lambda)$. Moreover, all the partitions of $n$ that satisfy the two conditions can be explicitly determined by $\lambda$. We show two similar results for some symmetric and exterior powers of $M^\lambda$. To state our results precisely, assume further that $\lambda=(\lambda_1,\ldots,\lambda_\ell)\vdash n$. For all $1\leq i\leq \ell$, write $\lambda_i=pq_i+r_i$, where $q_i$, $r_i$ are non-negative integers and $0\leq r_i<p$. Let $q_{\lambda}=(pq_1,\ldots,pq_\ell)$, $r_{\lambda}=(r_1,\ldots,r_\ell)$ and use $\overline{\alpha}$ to denote the rearrangement of a composition $\alpha$. Let $q_{\lambda}\cup(p)$ be the rearrangement of the composition $(pq_1,\ldots,pq_\ell, p)$. We prove the following:
\begin{thmx}\label{T;B}
Let $a$, $n$ be positive integers and $a>1$. Let $\lambda\vdash n$.
\begin{enumerate}
\item [\em (i)] There exists some $\mu\vdash n$ such that $Y^\mu\mid S^aM^\lambda$ and $c_{\sym{n}}(Y^\mu)=c_{\sym{n}}(S^aM^\lambda)$.
\item [\em (ii)]Assume that $\lambda\neq (n)$. Then there exists some $\mu\vdash n$ such that $Y^\mu\mid \Lambda^2M^\lambda$ and  $c_{\sym{n}}(Y^\mu)=c_{\sym{n}}(\Lambda^2M^\lambda)$.
\end{enumerate}
Moreover, the partition $\mu$ in both \em (i) and $\mathrm{(ii)}$ can be explicitly determined by $\lambda$.
\end{thmx}
\begin{thmx}\label{T;C}
Let $n$ be a positive integer and $\lambda\vdash n$. Let $\overline{r_{\lambda}}=(p-1,1)$ and $\mu=q_\lambda\cup(p)$. Let $M$ be an indecomposable $\F\sym{n}$-module.
\begin{enumerate}
\item [\em (i)] If $M\mid S^pM^\lambda$ and $c_{\sym{n}}(M)=c_{\sym{n}}(S^pM^\lambda)$, then $M\cong Y^\nu$, where $Y^\nu\mid M^\mu$. Moreover, all the partitions labelling these Young modules can be explicitly determined by $\mu$.
\item [\em (ii)] If $p=2$, $M\mid \Lambda^2M^\lambda$ and $c_{\sym{n}}(M)=c_{\sym{n}}(\Lambda^2M^\lambda)$, then $M\cong Y^\nu$, where $Y^\nu\mid M^\mu$. Moreover, all the partitions labelling these Young modules can be explicitly determined by $\mu$.
\end{enumerate}
\end{thmx}
Our final result is as follows. We refer the reader to Theorem \ref{T;multiplicity} for more details.
\begin{thmx}\label{T;D}
Let $n$ be a positive integer and $\lambda\vdash n$. The vertices of all the $\F\sym{n}$-Scott modules that are isomorphic to direct summands of $S^2M^\lambda$ or $\Lambda^2M^\lambda$ are determined. Moreover, the corresponding multiplicity of such a Scott module in $S^2M^\lambda$ or $\Lambda^2M^\lambda$ is also explicitly determined.
\end{thmx}
This paper is organized as follows. In Section $2$, we gather the notation and give a brief summary of the required knowledge. Section $3$ includes the proof of Theorem \ref{T;A}. We also determine the symmetric and exterior powers of $M^\lambda$ that are projective. In Section $4$, we classify all the indecomposable exterior powers of $M^\lambda$. Sections $5$ and $6$ contain the proof of Theorem \ref{T;B}. Theorem \ref{T;C} is shown in Section $7$. By a theorem proved in Section $8$, in Section $9$, we prove Theorem \ref{T;D}.

\section{Notation and Preliminaries}
Fix a field $\F$ of positive characteristic $p$ throughout the whole paper. For a given finite group $G$, we write $H\leq G$ (resp. $H<G$) to indicate that $H$ is a subgroup (resp. proper subgroup) of $G$. If $H\leq G$, let $N_G(H)$ be the normalizer of $H$ in $G$. For a non-empty set $S$, let $\langle S\rangle_\F$ be the $\F$-linear space generated by $S$. By convention, $\langle\varnothing\rangle_\F$ is the zero space. All the $\F G$-modules considered in this paper are finitely generated left $\F G$-modules. Induction and restriction of $\F G$-modules are presented by $\uparrow$ and $\downarrow$ respectively. Use $\F_{G}$ to denote the trivial $\F G$-module throughout the whole paper and omit the subscript if there is no confusion. We fix the notation and present some required results in the following subsections.

\subsection{Representation theory of finite groups}
We assume that the reader is familiar with modular representation theory of finite groups. For a general background on this topic, one may refer to \cite{JAlperin1} or \cite{HNYT}. Let $\mathbb{N}$ be the set of natural numbers.

Let $G$ be a finite group and $M$, $N$ be $\F G$-modules. Write $N\mid M$ if $N$ is isomorphic to a direct summand of $M$, i.e., $M\cong L\oplus N$ for some $\F G$-module $L$. If $N$ is indecomposable, for a decomposition of $M$ into a direct sum of indecomposable modules, the number of indecomposable direct summands that are isomorphic to $N$ is well-defined by the Krull--Schmidt Theorem and is denoted by $[M:N]$. Let $P\leq G$ and $N$ be indecomposable. Following \cite{JGreen}, say $P$ a vertex of $N$ if $P$ is a minimal subgroup of $G$ (with respect to the inclusion of subgroups of $G$) subject to the condition that $N\mid (N{\downarrow_P}){\uparrow^G}$. All the vertices of $N$ can form a $G$-conjugacy class of a $p$-subgroup of $G$. Let $P$ be a vertex of $N$. There exists some indecomposable $\F P$-module $S$ such that $N\mid S{\uparrow^G}$. Call $S$ a $P$-source of $N$. All the $P$-sources of $N$ are $N_G(P)$-conjugate to each other. If $N$ has a trivial $P$-source, then $N$ is called a trivial source $\F G$-module.
Let $1<a\in \mathbb{N}$ and $S^aM$ be the $a$th symmetric power of $M$. If $a\leq\dim_\F M$, let $\Lambda^aM$ be the $a$th exterior power of $M$. By convention, write $S^0M=\Lambda^0M=\F$ and $S^1M=\Lambda^1M=M$. For the definitions of these objects, one may refer to \cite[1.14]{Benson}. In particular, they are natural $\F G$-modules by the diagonal action of $G$. Let $M^{{\otimes}^a}$ be the $a$th inner tensor product of $M$ over $\F$. If $a<p$, it is well-known that $S^aM\mid M^{{\otimes}^a}$ and $\Lambda^aM\mid M^{{\otimes}^a}$. Let $L$ be an $\F G$-module. We also have $S^c(M\oplus L)\cong\bigoplus_{i=0}^c(S^iM\otimes S^{c-i}L)$ and $\Lambda^c(M\oplus L)\cong\bigoplus_{i=0}^c(\Lambda^iM\otimes \Lambda^{c-i}L)$ as $\F G$-modules, where $c\in \mathbb{N}$ and the inner tensor products of modules are over $\F$.
\subsection{p-permutation modules and Scott modules}
Let $G$ be a finite group and $M$ be an $\F G$-module. The module $M$ is said to be a $p$-permutation module if, for every $p$-subgroup $P$ of $G$, there exists an $\F$-basis $\B _P$ of $M$ (depending on P) that is permuted by $P$. It is well-known that the indecomposable $p$-permutation $\F G$-modules are exactly the trivial source $\F G$-modules. Also note that the class of $p$-permutation $\F G$-modules is closed by taking direct sums and direct summands. Therefore, the projective $\F G$-modules are $p$-permutation $\F G$-modules. Let $P$ be a $p$-subgroup of $G$ and $M$ be a $p$-permutation $\F G$-module. Let $\B _P(M)$ be a fixed $\F$-basis of $M$ permuted by $P$. Let $1<a\in \mathbb{N}$, $d=\dim_\F M$ and $\B_P(M)=\{v_1,\ldots,v_d\}$. Set $\B_{P,a}^s(M)=\{u:u=v_{i_1}\odot\cdots\odot v_{i_a}\in S^aM,\ 1\leq i_1\leq\cdots\leq i_a\leq d\}$ and note that $\B_{P,a}^s(M)$ is an $\F$-basis of $S^aM$ permuted by $P$. By the definition of a $p$-permutation $\F G$-module, $S^aM$ is a $p$-permutation $\F G$-module. Similarly, we have
\begin{lem}\label{L;Basis}
Let $P$ be a $p$-subgroup of a finite group $G$ and $1<a\in \mathbb{N}$. Let $M$ be a $p$-permutation $\F G$-module. Let $a\leq d=\dim_\F M$ and $\B _P(M)=\{v_1,\ldots,v_d\}$. Then $\{\alpha_uu: u=v_{i_1}\wedge\cdots\wedge v_{i_a}\in \Lambda^aM,\ 1\leq i_1<\cdots<i_a\leq d\}$ is an $\F$-basis of $\Lambda^aM$ permuted by $P$, where the scalar $\alpha_u$ depends on $u$ and is either $1$ or $-1$. In particular, $\Lambda^aM$ is a $p$-permutation $\F G$-module.
\end{lem}
\begin{proof}
Note that $\B=\{u:u=v_{i_1}\wedge\cdots\wedge v_{i_a}\in\Lambda^aM,\ 1\leq i_1<\cdots<i_a\leq d\}$ is an $\F$-basis of $\Lambda^aM$. For any $g\in G$ and $u\in\B$, there exists some $v\in \B$ such that $gu=v$ or $gu=-v$. The first statement thus follows by \cite[Lemma 2.6]{GLDM}. For any $p$-subgroup $Q$ of $G$ and the fixed $\B_Q(M)$, we get an $\F$-basis of $\Lambda^aM$ permuted by $Q$. The second statement follows by the definition of a $p$-permutation $\F G$-module.
\end{proof}
For the given $\B_P(M)$, if $a\leq d$, let $\B_{P,a}^e(M)$ be the $\F$-basis of $\Lambda^aM$ in Lemma \ref{L;Basis}. Let $H\leq G$. For the transitive permutation $\F G$-module $(\F_H){\uparrow^G}$, it is well-known that $(\F_H){\uparrow^G}$ has a unique trivial $\F G$-submodule and a unique trivial $\F G$-quotient module. Moreover, there exists an indecomposable $\F G$-direct summand $S$ of $(\F_H){\uparrow^G}$ with a trivial $\F G$-submodule and a trivial $\F G$-quotient module. The module $S$, unique up to isomorphism, is called the Scott module of $(\F_H){\uparrow^G}$ and is denoted by $Sc_G(H)$. It is well-known that the vertices of $Sc_G(H)$ are $G$-conjugate to a Sylow $p$-subgroup of $H$. Moreover, $Sc_G(H)$ is a trivial source $\F G$-module. Let $K\leq G$ and $Q$ be a Sylow $p$-subgroup of $H$. We remark that $Sc_G(H)\cong Sc_G(K)$ if and only if a Sylow $p$-subgroup of $H$ is $G$-conjugate to a Sylow $p$-subgroup of $K$. Therefore, we usually use $Sc_G(Q)$ to denote $Sc_G(H)$. For more properties of $Sc_G(Q)$, one may refer to \cite{Burry}.

\subsection{Brauer quotients and complexities of modules}
We now describe the Brauer quotients of $p$-permutation modules developed by Brou\'{e} in \cite{Broue}. Let $G$ be a finite group and $P\leq G$. Let $M$ be an $\F G$-module. For a given non-empty set $S\subseteq M$, we set $S^P=\{v\in S: \forall\ q\in P,\ qv=v\}$. If $S$ is an $\F G$-submodule of $M$, then $S^P$ is an $\F [N_G(P)/P]$-module. For any $p$-subgroup $Q$ of $P$, the relative trace map from $M^Q$ to $M^P$, denoted by $\mathrm{Tr}_Q^P$, is defined to be $$ \Tr_Q^P(v)=\sum_{g\in \{P/Q\}}gv,$$
where $\{P/Q\}$ is a complete set of representatives of left cosets of $Q$ in $P$. Note that the $\F$-linear map $\Tr_Q^P$ is independent of the choices of $\{P/Q\}$. We also define
$$\Tr^P(M)=\sum Tr_Q^P(M^Q),$$
where the sum runs over all the members of $\{Q<P:Q\ \text{is}\ \text{a}\ \text{$p$-group}\}$.
It is clear to see that $\Tr^P(M)$ is an $\F[N_G(P)/P]$-submodule of $M^P$. The Brauer quotient of $M$ with respect to $P$, written as $M(P)$, is the $\F[N_G(P)/P]$-module $$M^P/\Tr^P(M).$$ Notice that $M(P)=0$ unless $P$ is a $p$-subgroup of $G$. If $P$ is a $p$-subgroup of $G$ and $M$ is indecomposable, it is well-known that $M(P)\neq 0$ only if $P$ is $G$-conjugate to a $p$-subgroup of a vertex of $M$. We collect some well-known results as follows.
\begin{lem}\label{L;Brauerquotient}
Let $P$ be a $p$-subgroup of a finite group $G$. Let $M$ be a trivial source $\F G$-module with a vertex $V$. Let $N$ be a $p$-permutation $\F G$-module.
\begin{enumerate}
\item [\em (i)] We have $M(P)\neq 0$ if and only if $P$ is $G$-conjugate to a $p$-subgroup of $V$.
\item [\em (ii)] We have $N(P)\cong\langle \B_P(N)^P\rangle_\F$ as $\F[N_G(P)/P]$-modules.
\item [\em (iii)]We have $M\mid N$ if and only if $M(V)\mid N(V)$ and $[N:M]=[N(V):M(V)]$.
\end{enumerate}
\end{lem}
The $p$-rank of $G$ is the largest non-negative integer $m$ subject to the condition that $G$ has an elementary abelian $p$-subgroup of order $p^m$. Let $\text{rank}(G)$ be the $p$-rank of $G$.
The groups that are mutually isomorphic have the same $p$-rank. The complexity of $M$, denoted by $c_G(M)$, is the smallest non-negative integer $c$ such that
$$\lim_{\ell\rightarrow\infty}\frac{\dim_\F P_\ell}{\ell^c}=0,$$ where $\cdots\rightarrow P_1\rightarrow P_0\rightarrow M$ is a minimal projective resolution of $M$. A well-known fact tells us that $c_G(\F)=\mathrm{rank}(G)$. Moreover, $c_G(M)=0$ if and only if $M$ is projective. Let $H$ be a finite group and $N$ be an $\F H$-module. Let $M\boxtimes N$ be the outer tensor product of $M$ and $N$ over $\F$. It is a natural $\F[G\times H]$-module. We summarize some well-known properties of $c_G(M)$ as follows.
\begin{lem}\label{L;complexity}
Let $G$ and $H$ be finite groups and $M$ be an $\F G$-module.
\begin{enumerate}
\item [\em (i)] Let $N$, $L$ be $\F G$-modules. If $M\cong N\oplus L$, then $c_G(M)=\max\{c_G(N),c_G(L)\}$.
\item [\em (ii)]If $M$ is an indecomposable $\F G$-module with a vertex $P$ and a $P$-source $S$, then $c_G(M)=c_P(S)$.
\item[\em (iii)]If $N$ is an $\F H$-module, then $c_{G\times H}(M\boxtimes N)=c_{G}(M)+c_H(N)$.
\item [\em (iv)]If $H\leq G$ and $N$ is an $\F H$-module, then $c_{G}(N{\uparrow^G})=c_H(N)$.
\item [\em (v)]Let $\varepsilon^G$ be the set of all elementary abelian $p$-subgroups of $G$. If $M$ is a $p$-permutation module, then $c_G(M)=\max\{\mathrm{rank}(E): E\in\varepsilon^G,\ M(E)\neq 0\}$.
\end{enumerate}
\end{lem}

\subsection{Combinatorics}
Let $\mathbb{N}_0=\mathbb{N}\cup\{0\}$, $n\in \mathbb{N}$ and $\mathbf{n}=\{1,\ldots,n\}$. Let $\sym{S}$ be the symmetric group acting on a finite set $S$. Let $\sym{n}=\sym{\mathbf{n}}$ and $\sym{\varnothing}=1$. A composition of $n$ is a finite sequence of non-negative integers $(\lambda_1,\ldots,\lambda_{\ell})$ such that $\sum_{i=1}^\ell\lambda_i=n$. If this sequence is non-increasing and $\lambda_\ell\neq 0$, call this composition a partition of $n$. As the empty set, the unique composition of $0$ is denoted by $\varnothing$. It is also the unique partition of $0$. Let $a$, $m\in \mathbb{N}_0$. Write $\lambda\models m$ (resp. $\lambda\vdash m$) to indicate that $\lambda$ is a composition (resp. partition) of $m$. In this paper, we shall use the exponential expression of sequences of integers. For example, $(1,2,2,1)=(1,2^2,1)$. Let $\unlhd$ be the dominance order of partitions. Let $\lambda=(\lambda_1,\ldots,\lambda_{\ell})\models n$ and call each entry of $\lambda$ a part of $\lambda$. Let $|\lambda|=\sum_{i=1}^\ell\lambda_i=n$ and $\overline{\lambda}$ be the partition obtained from $\lambda$ by omitting the zero parts of $\lambda$ and rearranging the non-zero parts of $\lambda$. Let $\ell(\lambda)$ be the number of parts of $\overline{\lambda}$. By convention, $\ell(\varnothing)=|\varnothing|=0$ and $\overline{\varnothing}=\varnothing$. Let $\mu=(\mu_1,\ldots,\mu_{\ell'})\models m$. Set $a\mu=(a\mu_1,\ldots,a\mu_{\ell'})\models am$,
$\mu\cup\lambda=\overline{(\mu_1,\ldots,\mu_{\ell'},\lambda_1,\ldots,\lambda_{\ell})}\vdash m+n$ and $\varnothing\cup \mu=\mu\cup\varnothing=\overline{\mu}$. Put $\varnothing+\lambda=\lambda+\varnothing=\lambda$ and $\varnothing+\varnothing=\varnothing$. Let $m>0$. If $\ell'\leq \ell$,
\begin{align*}
\lambda+\mu=\mu+\lambda=(\lambda_1+\mu_1,\ldots,\lambda_{\ell'}+\mu_{\ell'},\lambda_{\ell'+1},\ldots,\lambda_\ell)\models m+n.
\end{align*}
Assume further that $\mu\vdash m$. It is called a $p$-restricted partition of $m$ if $\mu_i-\mu_{i+1}<p$ for all $1\leq i<\ell'$ and $\mu_{\ell'}<p$. By \cite[Lemma 7.5]{GJ}, the $p$-adic expansion of $\mu$ is the unique sum $\sum_{i=0}^xp^i\mu(i)$ for some $x\in\mathbb{N}_0$, where $\mu(i)$ is either $\varnothing$ (a sequence of zeros) or a $p$-restricted partition of $|\mu(i)|$ for all $0\leq i\leq x$, $\mu(x)\neq\varnothing$ and $\mu=\sum_{i=0}^xp^i\mu(i)$.

The Young diagram of $\lambda$, denoted by $[\lambda]$, is a left aligned array of boxes, where the $i\mathrm{th}$ row of the array has exactly $\lambda_i$ boxes for all $1\leq i\leq\ell$. In particular, the $i\mathrm{th}$ row of $[\lambda]$ is empty if $\lambda_i=0$. We regard the $1\mathrm{th}$ row (resp. $\ell\mathrm{th}$ row) of $[\lambda]$ as the top row (resp. bottom row) of $[\lambda]$. We will not distinguish between $\lambda$ and $[\lambda]$. Via a bijection, all the boxes of $\lambda$ are replaced by the numbers $1,\ldots,n$. A result of this process is called a $\lambda$-tableau. Following this definition,
we view each row of a $\lambda$-tableau as a set of numbers. Let $\mathrm{T}(\lambda)$ be the set of all the $\lambda$-tableaux. Given $s$, $t\in \mathrm{T}(\lambda)$, write $s\approx t$ if the $i\mathrm{th}$ rows of $s$ and $t$ are equal as sets for all $1\leq i\leq\ell$. Let $\{t\}$ be the equivalence class containing $t$ with respect to $\approx$ and call it a $\lambda$-tabloid. Regard the rows of $t$ as the rows of $\{t\}$. Let $\mathcal{T}(\lambda)$ be the set of all the  $\lambda$-tabloids. To display a $\lambda$-tabloid $\{t\}$, draw lines between the rows of $t$ and rewrite the numbers of each row of $t$ in increasing order from left to right. For example, if $\lambda=(2,0,3)\models 5$ and
\[t=\ {\begin{matrix}
4 & 1   \\
         \\
3 & 5 &2 \\
\end{matrix}}
\ \ ,\ \ \ \ \
\{t\}=\ {\begin{matrix}
\cline{1-2}
\cline{1-2}
1 & 4   \\ \cline{1-2}
         \\ \cline{1-3}
2 & 3 &5 \\ \hline
\end{matrix}}
\ \ .\]

\subsection{Modules of symmetric groups}
We now briefly present some material of representation theory of symmetric groups needed in the paper. One can refer to \cite{GJ1} or \cite{GJ3} for a background on this topic. Given $H\leq\sym{n}$, let $\mathbf{n}/H$ be the set of orbits of $\mathbf{n}$ under the natural action of $H$. Let $\lambda=(\lambda_1,\ldots,\lambda_\ell)\vdash n$. The Young subgroup of $\sym{n}$ with respect to $\lambda$, denoted by $\sym{\lambda}$, is defined to be $
\sym{\lambda_1}\times\cdots\times\sym{\lambda_\ell},$ where the first factor $\sym{\lambda_1}$ acts on the set $\{1,\ldots, \lambda_1\}$, the second factor acts on the set $\{\lambda_1+1,\ldots, \lambda_1+\lambda_2\}$ and so on. The Young permutation module with respect to $\lambda$, denoted by $M^{\lambda}$, is the $\F$-linear space generated by all $\lambda$-tabloids, where $\sym{n}$ permutes these $\lambda$-tabloids. It is also isomorphic to $(\F_{\sym{\lambda}}){\uparrow^{\sym{n}}}$. Since $M^\lambda$ is a permutation module, observe that $M^\lambda$ is a $p$-permutation module. Let $d=\dim_\F M^\lambda$ and  $1<a\in \mathbb{N}$. For any given order of the members of $\mathcal{T}(\lambda)$, say $\mathcal{T}(\lambda)=\{\{t_1\},\ldots,\{t_d\}\}$, let
\begin{align}
&\mathcal{T}(\lambda)_a^s=\{\{t_{i_1}\}\odot\cdots\odot\{t_{i_a}\}\in S^aM^\lambda: 1\leq i_1\leq\cdots\leq i_a\leq d\},\\
&\mathcal{T}(\lambda)_a^e=\{\{t_{i_1}\}\wedge\cdots\wedge\{t_{i_a}\}\in \Lambda^aM^\lambda: 1\leq i_1<\cdots<i_a\leq d\}.
\end{align}
Note that $\mathcal{T}(\lambda)_a^s$ is an $\F$-basis of $S^aM^\lambda$. If $a\leq d$, $\mathcal{T}(\lambda)_a^e$ is also an $\F$-basis of $\Lambda^aM^\lambda$.
As $\mathcal{T}(\lambda)$ forms an $\F$-basis of $M^\lambda$ that can be permuted by any $p$-subgroup of $\sym{n}$, for any $p$-subgroup $P$ of $\sym{n}$, we fix $\B_P(M^\lambda)=\mathcal{T}(\lambda)$. We need the following result.
\begin{lem}\label{PModules}\cite[Lemma 2.7]{JLW}
Let $\lambda\vdash n$ and $P$ be a $p$-subgroup of $\sym{n}$. Then $\{t\}\in \mathcal{T}(\lambda)^P$ if and only if the set of numbers in each row of $\{t\}$ is a union set of some members of $\mathbf{n}/P$. In particular, both $|\mathcal{T}(\lambda)^P|$ and $\dim_\F M^\lambda(P)$ are the number of unordered ways to insert the orbits in $\mathbf{n}/P$ into the rows of $\lambda$.
\end{lem}

The representatives of all isomorphism classes of indecomposable direct summands of Young permutation modules are called Young modules and are parameterized by James in \cite[Theorem 3.1]{GJ2}. Namely, let $\lambda$, $\mu\vdash n$. For a given decomposition of $M^\lambda$ into a direct sum of indecomposable modules, there exists a unique indecomposable direct summand of $M^\lambda$ that contains the Specht module labelled by $\lambda$. The module, unique up to isomorphism, is denoted by $Y^\lambda$. It is well-known that $Y^\lambda$ is a self-dual trivial source module. James proved that $[M^\lambda:Y^\lambda]=1$ and $[M^{\lambda}:Y^{\mu}]\neq 0$ only if $\lambda\unlhd \mu$. We thus get
\begin{equation*}
M^\lambda\cong Y^\lambda\oplus\bigoplus_{\lambda\lhd \mu}[M^\lambda:Y^\mu]Y^\mu.
\end{equation*}
Recall that the vertices of Young modules are the Sylow $p$-subgroups of Young subgroups. In \cite{Donkin2}, Donkin introduced the signed Young modules for the case $p>2$ and parameterized them by the pairs of partitions. Recall that they are trivial source modules. We shall use $Y(\alpha|p\beta)$ to denote the $\F\sym{n}$-signed Young module labelled by the partitions $\alpha,\beta$, where $|\alpha|+p|\beta|=n$.
We end this section by two known results.
\begin{thm}\cite[(3.6)]{Donkin}\label{T;Donkin}
Let $\lambda$, $\mu\vdash n$ and $\lambda$ have the $p$-adic expansion $\sum_{i=0}^mp^i\lambda(i)$ for some $m\in \mathbb{N}_0$. Then $Y^\lambda\mid M^\mu$ if and only if there exists a sum $\sum_{i=0}^mp^i\nu[i]$ such that $\nu[i]\models|\lambda(i)|$ and $\overline{\nu[i]}\unlhd \lambda(i)$ for all $0\leq i\leq m$ and $\mu=\sum_{i=0}^mp^i\nu[i]$.
\end{thm}

\begin{lem}\label{HemmerNakano}\cite[Proposition 3.2.2, Theorem 3.3.2]{DHDN}
Let $\lambda\vdash n$ and $\lambda$ have $p$-adic expansion $\sum_{i=0}^mp^i\lambda(i)$ for some $m\in \mathbb{N}_0$. Then we have $\ c_{\sym{n}}(M^\lambda)=\mathrm{rank}(\sym{\lambda})$ and $c_{\sym{n}}(Y^\lambda)=\frac{n-|\lambda(0)|}{p}$.
\end{lem}
\section{Proof of Theorem \ref{T;A}}
The aim of this section is to finish the proof of Theorem \ref{T;A}. Moreover, we also determine the symmetric and exterior powers of a Young permutation module that are projective. Some lemmas are required as a preparation.

\begin{lem}\label{L;symmetric}
Let $G$ be a finite group and $P$ be a $p$-subgroup of $G$. Let $1<a\in \mathbb{N}$ and $M$ be an $\F G$-module. Then $\langle S\rangle_\F\subseteq \mathrm{Tr}^P(S^aM)$, where
$$S=\{u\in S^aM:u=w\odot v_1\odot\cdots\odot v_{a-1},\ w\in \mathrm{Tr}^P(M),\ v_1,\ldots,v_{a-1}\in M^P\}.$$
\end{lem}
\begin{proof}
We assume that $P\neq 1$. For any $Q<P$, $x\in M^Q$ and $y_1,\ldots,y_{a-1}\in M^P$, it is enough to show that $\mathrm{Tr}^P_Q(x)\odot y_1\odot \cdots\odot y_{a-1}\in \mathrm{Tr}^P(S^aM)$. We have
\begin{align*}
\mathrm{Tr}^P_Q(x)\odot y_1\odot \cdots\odot y_{a-1}=&(\sum_{g\in \{P/Q\}}\!gx)\odot y_1\odot \cdots\odot y_{a-1}\\=&\sum_{g\in \{P/Q\}}\!g(x\odot y_1\odot \cdots\odot y_{a-1})\\=& \mathrm{Tr}^P_Q(x\odot y_1\odot \cdots\odot y_{a-1})\in \mathrm{Tr}^P(S^aM),
\end{align*}
as desired. As the case $P=1$ is trivial, the lemma follows.
\end{proof}
By a similar computation, we deduce the following result.
\begin{lem}\label{L;exterior}
Let $G$ be a finite group and $P$ be a $p$-subgroup of $G$. Let $1<a\in \mathbb{N}$ and $M$ be an $\F G$-module. If $a\leq\dim_\F M$, then $\langle S\rangle_\F\subseteq \mathrm{Tr}^P(\Lambda^aM)$, where
\begin{align*}
S=\{u\in \Lambda^aM: u=w\wedge v_1\wedge\cdots\wedge v_{a-1},\ w\in \mathrm{Tr}^P(M),\ v_1,\ldots,v_{a-1}\in M^P\}.
\end{align*}
\end{lem}

Let $H$ be a subgroup of a finite group $G$ and $M$ be an $\F G$-module. Let $S$ be a non-empty multiset whose members are vectors of $M$ and $S'$ be the underlying set of $S$. Write $H\Rrightarrow S$ if $S'\not\subseteq M^H$ and $S$ is closed under the action of $H$ on $M$. Also write $H\Rightarrow S$ if $H\Rrightarrow S$ and $S$ is a set. In particular, $H\Rrightarrow S$ implies that $S$ may have repeated members.

\begin{lem}\label{L;symmetricBrauer}
Let $G$ be a finite group and $P$ be a $p$-subgroup of $G$. Let $M$ be a $p$-permutation $\F G$-module and  $\mathcal{B}_P(M)=\{v_1,\ldots, v_d\}$, where $d=\dim_\F M$. Let $1<a\in \mathbb{N}$. Then $(S^aM)(P)\cong S^a(M(P))\oplus S_a(M,P)$ as $\F[N_G(P)/P]$-modules, where
$S_a(M,P)$ is the $\F$-linear space generated by
$$\{u\in \mathcal{B}_{P,a}^s(M): u=v_{i_1}\odot\cdots\odot v_{i_a},\ 1\leq i_1\leq\cdots\leq i_a\leq d, \ P\Rrightarrow\{v_{i_1},\ldots,v_{i_a}\}\}.$$
\end{lem}
\begin{proof}
The case $P=1$ is trivial. We thus assume that $P\neq 1$. By the definition, note that $(S^aM)(P)=\langle\{u+\mathrm{Tr}^P(S^aM): u\in (\mathcal{B}_{P,a}^s(M))^P\}\rangle_\F$. It is also clear that $u=v_{i_1}\odot\cdots\odot v_{i_a}\in (\mathcal{B}_{P,a}^s(M))^P$ if and only if
$\{v_{i_1},\ldots,v_{i_a}\}$ is closed under the action of $P$ on $M$. Since $N_G(P)$ permutes the orbits of $\mathcal{B}_P(M)$ under the action of $P$, $(S^aM)(P)=M_1\oplus M_2$ as $\F[N_G(P)/P]$-modules, where
\begin{align*}
&M_1=\langle\{u+\mathrm{Tr}^P(S^aM): u=v_{i_1}\odot\cdots\odot v_{i_a}\in \mathcal{B}_{P,a}^s(M),\  v_{i_1},\ldots,v_{i_a}\in(\mathcal{B}_P(M))^P\}\rangle_{\F},\\
&M_2=\langle\{u+\mathrm{Tr}^P(S^aM): u=v_{i_1}\odot\cdots\odot v_{i_a}\in \mathcal{B}_{P,a}^s(M),\  P\Rrightarrow\{v_{i_1},\ldots,v_{i_a}\}\}\rangle_{\F}.
\end{align*}
Observe that $M_2\cong S_a(M,P)$ as $\F[N_G(P)/P]$-modules. It suffices to prove that $M_1\cong S^a(M(P))$ as $\F[N_G(P)/P]$-modules. The natural map $\pi$ from $(S^aM)^P$ to $(S^aM)(P)$ induces an $\F$-linear map $\phi$ from $S^a(M(P))$ to $M_1$ by sending each vector  $(v_{i_1}+\mathrm{Tr}^P(M))\odot\cdots\odot(v_{i_a}+\mathrm{Tr}^P(M))$ to $v_{i_1}\odot\cdots\odot v_{i_a}+\mathrm{Tr}^P(S^aM)$, where we have $1\leq i_1\leq\cdots\leq i_a\leq d$ and $v_{i_1},\ldots,v_{i_a}\in(\mathcal{B}_P(M))^P$. By Lemma \ref{L;symmetric}, we know that $\phi$ is well-defined. Also notice that $\phi$ is a surjective $\F[N_G(P)/P]$-homomorphism. As $\dim_\F M_1=\dim_\F S^a(M(P))$, we conclude that $\phi$ is an $\F[N_G(P)/P]$-isomorphism. So $M_1\cong S^a(M(P))$ as $\F[N_G(P)/P]$-modules and the lemma follows.
\end{proof}

By Lemma \ref{L;exterior} and mimicking the proof of Lemma \ref{L;symmetricBrauer}, one deduces
\begin{lem}\label{L;exteriorBrauer}
Let $G$ be a finite group and $P$ be a $p$-subgroup of $G$. Let $1<a\in \mathbb{N}$. Let $M$ be a $p$-permutation $\F G$-module and $\mathcal{B}_P(M)=\{v_1,\ldots, v_d\}$, where $d=\dim_\F M$. If $a\leq d$, then $(\Lambda^aM)(P)\cong \Lambda^a(M(P))\oplus \Lambda_a(M,P)$ as $\F[N_G(P)/P]$-modules, where
$\Lambda_a(M,P)$ is the $\F$-linear space generated by
$$\{\alpha_uu\in \mathcal{B}_{P,a}^e(M): u=v_{i_1}\wedge\cdots\wedge v_{i_a},\ 1\leq i_1<\cdots<i_a\leq d, \ P\Rightarrow \{v_{i_1},\ldots,v_{i_a}\}\}.$$
\end{lem}

The correctness of the following lemma is obvious.
\begin{lem}\label{L;Brauer}
Let $G$ be a finite group and $P$ be a $p$-subgroup of $G$. Let $M$ be a $p$-permutation $\F G$-module, $d=\dim_\F M$ and $\mathcal{B}_P(M)=\{v_1,\ldots,v_d\}$.
\begin{enumerate}
\item [\em(i)] Let $1<a\in \mathbb{N}$. Then $S_a(M,P)\neq 0$ if and only if $P\Rrightarrow\{v_{i_1},\ldots,v_{i_a}\}$ for some $1\leq i_1\leq\cdots\leq i_a\leq d$.
\item [\em(ii)] Let $1<a\in \mathbb{N}$. If $a\leq d$, then $\Lambda_a(M,P)\neq 0$ if and only if $P\Rightarrow\{v_{i_1},\ldots,v_{i_a}\}$ for some $1\leq i_1<\cdots<i_a\leq d$.
\end{enumerate}
\end{lem}

Let $G$ be a finite group and $M$ be a $p$-permutation $\F G$-module. By Lemma \ref{L;complexity} (v), recall that $c_G(M)=\max\{\mathrm{rank}(E): E\in\varepsilon^G,\ M(E)\neq0\}$. To determine $c_G(M)$, our strategy is to find an elementary abelian $p$-subgroup $E$ of $G$ with the largest $p$-rank such that $M(E)\neq 0$. Let $\nu_p(a)$ be the largest non-negative integer $b$ such that $p^b\mid a$. For convenience, we restate Theorem \ref{T;A} as follows.

\begin{thm}\label{T;Proj}
Let $G$ be a finite group and $P$ be a non-zero projective $\F G$-module. Let $1<a\in \mathbb{N}$ and $r=\mathrm{rank}(G)$. Then $c_G(S^a P)=\min\{\nu_p(a),r\}$. Moreover, if $a\leq\dim_\F P$, then $c_{G}(\Lambda^aP)=\min\{\nu_p(a),r\}$.
\end{thm}
\begin{proof}
Set $u=\min\{\nu_p(a),r\}$ and use $\Delta^a P$ to denote either $\Lambda^a P$ or $S^a P$. We require that $a\leq\dim_\F P$ if $\Delta^a P=\Lambda^aP$. We first verify that $(\Delta^aP)(E)\neq 0$ for any elementary abelian $p$-subgroup $E$ of $G$ with $p$-rank $u$. It then suffices to prove that $\text{rank}(F)\leq u$ for any elementary abelian $p$-subgroup $F$ of $G$ such that $(\Delta^aP)(F)\neq0$. Since $u\leq r$, there indeed exists an elementary abelian $p$-subgroup of $G$ with $p$-rank $u$. For any elementary abelian $p$-subgroup $E$ of $G$ with $p$-rank $u$, if $u=0$, then $(\Delta^aP)(E)=\Delta^a P\neq0$. If $u>0$, let $\mathcal{B}_E(P)=\{v_1,\ldots,v_d\}$, where $d=\dim_\F P$. As $P$ is projective, we claim that $\mathcal{B}_E(P)$ has no fixed points under the action of any non-trivial subgroup of $E$. Otherwise, by Lemma \ref{L;Brauerquotient} (ii), $P(K)\neq 0$ for some non-trivial $K\leq E$, which contradicts with Lemma \ref{L;Brauerquotient} (i). The claim is shown. So $\mathcal{B}_E(P)$ is a union set of some orbits of size $p^u$ under the action of $E$. If $\Delta^a P=\Lambda^a P$, as $p^u\mid a$ and $a\leq d$, one can choose $a/p^u$ $E$-orbits from $\mathcal{B}_E(P)$ and define $S$ to be the union set of these orbits. We thus have $E\Rightarrow S$, which implies that $\Lambda_a(P,E)\neq 0$ by Lemma \ref{L;Brauer} (ii). So $(\Delta^aP)(E)\neq 0$ by Lemma \ref{L;exteriorBrauer}. If $\Delta^aP=S^aP$, as $p^u\mid a$, one can choose a single $E$-orbit $\O$ from $\mathcal{B}_E(P)$ and put $S$ to be the multiset satisfying the conditions that $S'=\O$ and each member of $S'$ has multiplicity $a/p^u$ in $S$. So $S$ has $a$ members and $P\Rrightarrow S$, which implies that $S_a(P,E)\neq0$ by Lemma \ref{L;Brauer} (i). We also have  $(\Delta^aP)(E)\neq 0$ by Lemma \ref{L;symmetricBrauer}.

Let $F$ be an elementary abelian $p$-subgroup of $G$ such that $(\Delta^aP)(F)\neq 0$. So $|F|\mid p^r$ as $\text{rank}(G)=r$. If $|F|=1$, we are done. If $|F|>1$, for the given $\B_F(P)$, as $P$ is projective, $(\Delta^aP)(F)\neq 0$ implies that $\Lambda_a(P,F)\neq 0$ and $S_a(P,F)\neq 0$ by Lemmas \ref{L;Brauerquotient} (i), \ref{L;symmetricBrauer} and \ref{L;exteriorBrauer}. If $\Delta^a P=\Lambda^a P$, since $\Lambda_a(P,F)\neq 0$, by Lemma \ref{L;Brauer} (ii), there exists some $S\subseteq \mathcal{B}_F(P)$ such that $|S|=a$ and $F\Rightarrow S$. By letting $F$ and $S$ play the roles of $E$ and $\B_E(P)$ respectively in the claim, note that $S$ is a union set of some orbits of size $|F|$ under the action of $F$. So $|F|\mid a$. We thus have $\mathrm{rank}(F)\leq \nu_p(a)$ and $\text{rank}(F)\leq \min\{\nu_p(a),r\}$. The case $\Delta^a P=S^aP$ is similar. The proof is now complete.
\end{proof}
For a more general case, an inequality is given as follows.
\begin{lem}\label{L;compare}
Let $G$ be a finite group and $M$ be a $p$-permutation $\F G$-module. Let $1<a\in \mathbb{N}$ and $d=\dim_\F M$. If $a\leq d$, then $c_G(\Lambda^a M)\leq c_G(S^a M)$.
\end{lem}
\begin{proof}
For any elementary abelian $p$-subgroup $E$ of $G$, we show that $(S^aM)(E)\neq 0$ if $(\Lambda^aM)(E)\neq0$,
which will complete the proof. The case $E=1$ is trivial. We thus assume $E\neq 1$. For the given $\B_E(M)$, by Lemma \ref{L;exteriorBrauer}, the condition $(\Lambda^aM)(E)\neq0$ implies $\Lambda^a(M(E))\neq0$ or $\Lambda_a(M,E)\neq0$. If $\Lambda^a(M(E))\neq 0$, then $M(E)\neq 0$ and $(S^aM)(E)\neq0$ by Lemma \ref{L;symmetricBrauer}. If $\Lambda_a(M,E)\neq 0$, by Lemma \ref{L;Brauer} (ii), there exists some $S\subseteq \mathcal{B}_E(M)$ such that $|S|=a$ and $E\Rightarrow S$. So $E\Rrightarrow S$ and $S_a(M,E)\neq0$ by Lemma \ref{L;Brauer} (i). We have $(S^aM)(E)\neq 0$ by Lemma \ref{L;symmetricBrauer}. The lemma follows.
\end{proof}

Our next goal is to determine the symmetric and exterior powers of a Young permutation module that are projective. Given $\lambda\vdash n$, we have fixed $\mathcal{B}_P(M^\lambda)=\mathcal{T}(\lambda)$ for any $p$-subgroup $P$ of $\sym{n}$. Also recall that $\text{rank}(\sym{n})=\lfloor \frac{n}{p}\rfloor$.

\begin{lem}\label{L;symmetriccomplexity}
Let $\lambda\vdash n$, $1<a\in \mathbb{N}$ and $r=\mathrm{rank}(\sym{\lambda})$. Then
\[c_{\sym{n}}(S^aM^\lambda)=\begin{cases}
r, &\text{if}\ \nu_p(a)=0,\\
r+\nu_p(a), &\text{if}\ \nu_p(a)>0,\ \nu_p(a)+r\leq \lfloor \frac{n}{p} \rfloor,\\
\lfloor \frac{n}{p} \rfloor, &\text{if}\ \nu_p(a)>0,\ \nu_p(a)+r>\lfloor \frac{n}{p} \rfloor.
\end{cases}
\]
\end{lem}
\begin{proof}
It suffices to find an elementary abelian $p$-subgroup $E$ of $\sym{n}$ with the largest $p$-rank such that $(S^aM^\lambda)(E)\neq 0$. Let $u=\nu_p(a)$ and $b=\lfloor \frac{n}{p}\rfloor$. Define a $p$-cycle $s_i=((i-1)p+1,\ldots,ip)$ for all $1\leq i\leq b$. We distinguish three cases.
\begin{enumerate}[\text{Case} 1:]
\item $u=0$.
\end{enumerate}
Let $E$ be an elementary abelian $p$-subgroup of $\sym{n}$ such that $\mathrm{rank}(E)=r$ and $M^\lambda(E)\neq0$. We thus have $S^a(M^\lambda(E))\neq 0$ and $(S^aM^\lambda)(E)\neq0$ by Lemma \ref{L;symmetricBrauer}. For any elementary abelian $p$-subgroup $F$ of $\sym{n}$ satisfying $\text{rank}(F)>r$, Note that $M^\lambda(F)=0$. Otherwise, by Lemmas \ref{L;complexity} (v) and \ref{HemmerNakano}, $r<\text{rank}(F)\leq c_{\sym{n}}(M^\lambda)=r$, which is a contradiction. So $S^a(M^\lambda(F))=0$ and $|\mathcal{T}(\lambda)^F|=0$ by Lemma \ref{L;Brauerquotient} (ii). Therefore, if there exists some multiset $S$ such that $S$ has $a$ members, $S'\subseteq \mathcal{T}(\lambda)$ and $F\Rrightarrow S$, we get $p\mid a$, which contradicts with the assumption $u=0$. By Lemma \ref{L;Brauer} (i), $S_a(M^\lambda,F)=0$. So $(S^aM^\lambda)(F)=0$ by Lemma \ref{L;symmetricBrauer}. Therefore, $c_{\sym{n}}(S^aM^\lambda)=r$.
\begin{enumerate}[\text{Case} 2:]
\item $u>0$ and $u+r\leq b$.
\end{enumerate}
As $u+r\leq b$, let $E=\langle \bigcup_{i=1}^{u+r}\{s_i\}\rangle$ and $F=\langle\bigcup_{i=1}^r\{s_i\}\rangle$. By Lemma \ref{PModules}, there exists some $\{t\}\in \mathcal{T}(\lambda)^F$. Note that the orbit $\O$ of $\mathcal{T}(\lambda)$ containing $\{t\}$ under the action of $E$ has size $p^{u}$. Otherwise, there exists some $F<H\leq E$ such that $\{t\}\in \mathcal{T}(\lambda)^H$. By Lemma \ref{L;Brauerquotient} (ii), $M^\lambda(H)\neq 0$. This implies that
$r<\mathrm{rank}(H)\leq c_{\sym{n}}(M^\lambda)=r$ by Lemmas \ref{L;complexity} (v) and
\ref{HemmerNakano}. This is a contradiction. As $p^u\mid a$, let $S$ be the multiset satisfying the conditions $S'=\O$ and each member of $S'$ has multiplicity $a/p^u$ in $S$. So $S$ has $a$ members and $E\Rrightarrow S$, which implies that $S_a(M^\lambda,E)\neq 0$ by Lemma \ref{L;Brauer} (i). So $(S^aM^\lambda)(E)\neq 0$ by Lemma \ref{L;symmetricBrauer}. For any elementary abelian $p$-subgroup $L$ of $\sym{n}$ satisfying $\text{rank}(L)>u+r$, note that $S^a(M^\lambda(L))=0$ as $M^\lambda(L)=0$. Also notice that each orbit of $\mathcal{T}(\lambda)$ under the action of $L$ has size at least $p^{u+1}$.
If $S_a(M^\lambda,L)\neq 0$, by Lemma \ref{L;Brauer} (i), there exists some multiset $S$ such that $S$ has $a$ members, $S'\subseteq\mathcal{T}(\lambda)$ and $L\Rrightarrow S$. We thus have $p^{u+1}\mid a$. This contradicts with the fact $u=\nu_p(a)$. So $(S^aM^\lambda)(L)=0$ by Lemma \ref{L;symmetricBrauer}. Therefore, $c_{\sym{n}}(S^aM^\lambda)=u+r$.
\begin{enumerate}[\text{Case} 3:]
\item $u>0$ and $u+r>b$.
\end{enumerate}
Let $\tilde{E}=\langle\bigcup_{i=1}^b\{s_i\}\rangle$ and $\tilde{F}=\langle\bigcup_{i=1}^r\{s_i\}\rangle$. Let $\tilde{E}$ and $\tilde{F}$ play the roles of $E$ and $F$ in Case 2 respectively. As $b-r<u$, the proof of Case $2$ shows that we have $S^a(M^\lambda(\tilde{E}))\neq0$ or $S_a(M^\lambda,\tilde{E})\neq 0$. Therefore, $(S^aM^\lambda)(\tilde{E})\neq 0$ by Lemma \ref{L;symmetricBrauer}. The fact $\mathrm{rank}(\sym{n})=b$ thus forces that $c_{\sym{n}}(S^aM^\lambda)=b$. The proof is now complete.
\end{proof}

\begin{prop}\label{P;symmetricproj}
Let $\lambda=(\lambda_1,\ldots,\lambda_\ell)\vdash n$ and $a\in \mathbb{N}_0$. Then $S^aM^\lambda$ is projective if and only if $n<p$ or $p\nmid a$ and $\lambda_i<p$ for all $1\leq i\leq \ell$.
\end{prop}
We need another notation to study all the projective exterior powers of a Young permutation module. Let $b=\lfloor\frac{n}{p}\rfloor$. Set $t_i=\prod_{j=1}^i((j-1)p+1,\ldots,jp)\in\sym{n}$ and put $F_i=\langle t_i\rangle$ for all $1\leq i\leq b$. Given $\lambda\vdash n$, for all $1\leq i\leq b$, set $m_{\lambda,i}=\dim_\F M^\lambda(F_i)$ and $\overline{m_{\lambda,i}}=\dim_\F M^\lambda-m_{\lambda,i}$. By Lemma \ref{PModules}, it is easy to compute $m_{\lambda,i}$ for all $1\leq i\leq b$.
For our purpose, note that $\overline{m_{\lambda,i}}\geq 0$ for all $1\leq i\leq b$.
\begin{prop}\label{Projectivity}\label{P;exteriorproj}
Let $\lambda\vdash n$, $1<a\in \mathbb{N}$ and $b=\lfloor\frac{n}{p}\rfloor$. If $a\leq\dim_\F M^\lambda$, then $\Lambda^a M^\lambda$ is projective if and only if, for all $1\leq i\leq b$, if $a=x+py$ for some $x, y\in \mathbb{N}_0$, then we have either $m_{\lambda,i}<x$ or $\overline{m_{\lambda,i}}<py$.
\end{prop}
\begin{proof}
If $\Lambda^aM^\lambda$ is projective, by Lemma \ref{L;Brauerquotient} (i), $(\Lambda^aM^\lambda)(F_i)=0$ for all $1\leq i\leq b$. For any $1\leq i\leq b$, by Lemma \ref{L;exteriorBrauer}, $0=(\Lambda^aM^\lambda)(F_i)\cong\Lambda^a(M^\lambda(F_i))\oplus\Lambda_a(M^\lambda,F_i)$. In particular, $\Lambda^a(M^\lambda(F_i))=0$ and $\Lambda_a(M,F_i)=0$. Note that $\Lambda^a(M^\lambda(F_i))=0$ gives us $m_{\lambda,i}<a$. Therefore, it suffices to check the decompositions $a=x+py$, where $x,y\in \mathbb{N}_0$ and $x<a$. For some $1\leq i\leq b$, suppose that $m_{\lambda,i}\geq x$ and $\overline{m_{\lambda,i}}\geq py$ for such a decomposition $a=x+py$. By Lemma \ref{L;Brauerquotient} (ii), $|\mathcal{T}(\lambda)^{F_i}|\geq x$. Moreover, as $F_i$ is a cyclic group of order $p$ and $\overline{m_{\lambda,i}}\geq py$, when acted by $F_i$, the number of orbits of $\mathcal{T}(\lambda)$ having size $p$ is at least $y$. Let $\mathcal{X}\subseteq \mathcal{T}(\lambda)^{F_i}$ and $\mathcal{Y}\subseteq \mathcal{T}(\lambda)$, where $|\mathcal{X}|=x$ and $\mathcal{Y}$ is a union set of exactly $y$ orbits of $\mathcal{T}(\lambda)$ having size $p$ under the action of $F_i$. Since $x<a$ and $y>0$, we have $|\mathcal{X}\cup\mathcal{Y}|=a$ and $F_i\Rightarrow \mathcal{X}\cup\mathcal{Y}$, which implies that $\Lambda_a(M^\lambda, F_i)\neq 0$ by Lemma \ref{L;Brauer} (ii). So $(\Lambda^aM^\lambda)(F_i)\neq0$ by Lemma \ref{L;exteriorBrauer}. This is a contradiction. We get the desired assertion.

Conversely, if the assertion holds, suppose that $\Lambda^aM^\lambda$ is not projective. By Lemma \ref{L;complexity} (v), there exists a non-trivial $p$-subgroup $P$ of $\sym{n}$ such that $(\Lambda^a(M^\lambda))(P)\neq0$. We may assume further that $P=F_i$ for some $1\leq i\leq b$. By Lemma \ref{L;exteriorBrauer}, we have $\Lambda^a(M^\lambda(F_i))\neq 0$ or $\Lambda_a(M^\lambda, F_i)\neq 0$. For the decomposition $a=a+0$, we deduce that $\dim_\F M^\lambda(F_i)=m_{\lambda,i}<a$ by the assertion. So $\Lambda^a(M^\lambda(F_i))=0$ and $\Lambda_a(M^\lambda,F_i)\neq 0$. By Lemma \ref{L;Brauer} (ii), we  deduce that there exists some $S\subseteq\mathcal{T}(\lambda)$ such that $|S|=a$ and $F_i\Rightarrow S$. Let $x=|S^{F_i}|$ and $y$ be the number of orbits of $S$ having size $p$ under the action of $F_i$. We have $a=x+py$, $m_{\lambda,i}\geq x$ and $\overline{m_{\lambda,i}}\geq py$. This contradicts with the assertion. So $\Lambda^a M^\lambda$ is projective and we are done.
\end{proof}

We end this section with an example. Let $p=3$, $n=5$, $a=8$, $\lambda=(3,2)\vdash 5$. All allowable decompositions of $8$ in Proposition \ref{P;exteriorproj} are exactly $8+0$, $5+3$, $2+6$. By Lemma \ref{PModules}, $m_{(3,2),1}=1<2$, $5$, $8$. By Proposition \ref{P;exteriorproj}, $\Lambda^8M^{(3,2)}$ is projective. By this example and Lemma \ref{L;symmetriccomplexity}, the inequality in Lemma \ref{L;compare} can be strict.
\section{Indecomposable exterior powers of Young permutation modules}
In this section, we classify all the indecomposable exterior powers of $\F\sym{n}$-Young permutation modules. For our purpose, for any $m\in \mathbb{N}$, let $sgn(m)$ denote the $\F\sym{m}$-sign module and omit the parameter if there is no confusion. Also recall that $Y(\alpha|p\beta)$ is the $\F\sym{n}$-signed Young module labelled by partitions $\alpha$, $\beta$, where $|\alpha|+p|\beta|=n$.
\begin{lem}\label{L;(n-1,1)}
Let $a\in \mathbb{N}$ and $1<a<n$. Then $\Lambda^aM^{(n-1,1)}$ is indecomposable if and only if precisely one of the following situations holds:
\begin{enumerate}
\item[{\em (i)}] $p>2$, $p\mid n$, $a=pb+r$, where $b,\ r\in\mathbb{N}_0$ and $0\leq r<p$, in this case, $\Lambda^aM^{(n-1,1)}\cong Y((n-a,1^r)|(pb))$;
\item[{\em (ii)}] $p=2$, $2\mid n$, in this case, $\Lambda^aM^{(n-1,1)}\cong Y^{\overline{(n-a,a)}}$, where $2^m\leq n<2^{m+1}$ for some $m\in \mathbb{N}$ and $2^{i-1}\leq \min \{n-a,a\}<2^i$, $\min\{n-a,a\}\equiv\frac{n-2^m}{2}\pmod {2^{i-1}}$ for some $1\leq i\leq m$.
\end{enumerate}
\end{lem}
\begin{proof}
When $p>2$, according to \cite[Theorem 1.4 (i), Proposition 7.1]{GLDM}, $\Lambda^aM^{(n-1,1)}$ is indecomposable if and only if (i) holds. When $p=2$,  $\Lambda^aM^{(n-1,1)}\cong M^{\overline{(n-a,a)}}$. By \cite[Theorem 2 (b)]{Gill}, $\Lambda^aM^{(n-1,1)}$ is indecomposable if and only if (ii) holds. We are done.
\end{proof}

One more notation is required here. Given $\lambda=(\lambda_1,\ldots,\lambda_\ell)\vdash n$, let $c(\lambda)$ be the largest subscript $c$ of parts of $\lambda$ such that $\lambda_{c}\geq 2$. Set $c(\lambda)=0$ if $\lambda=(1^n)$. Write
\[
f^\lambda=\begin{cases}
\dim_\F M^\lambda-\sum_{i=1}^{c(\lambda)}\frac{(n-2)!}{(\lambda_i-2)!\prod_{j\neq i}\lambda_j!}, & \text{if}\ c(\lambda)>0,\\
\dim_\F M^{(1^n)}, & \text{if}\ c(\lambda)=0.\\\end{cases}\]

\begin{lem}\label{L;powerd}
Let $\lambda\vdash n$ and $d=\dim_\F M^\lambda>1$. Then
\[\Lambda^d M^\lambda\cong\begin{cases}
\F,& \text{if}\ p=2\ \text{or}\ p>2\ \text{and}\ 4\mid f^\lambda,\\
sgn, & \text{otherwise}.\\\end{cases}\]
\end{lem}
\begin{proof}
Let $y=\frac{f^\lambda}{2}$. As $\dim_\F \Lambda^dM^\lambda=1$ and $\F$, $sgn$ are the unique one-dimensional $\F\sym{n}$-modules. it suffices to consider $(\Lambda^dM^\lambda){\downarrow_{\langle (1,2)\rangle}}$. Let
$S=\mathcal{T}(\lambda)^{\langle (1,2)\rangle}$ and note that $\langle (1,2)\rangle$ acts on $ \mathcal{T}(\lambda)\setminus S$. Set $x=|S|$ and observe that each orbit of $\mathcal{T}(\lambda)\setminus S$ under the action of $\langle (1,2)\rangle$ has size $2$. By Lemma \ref{PModules}, $x=d-f^\lambda$ and $\mathcal{T}(\lambda)\setminus S$ has $y$ orbits under the action of $\langle (1,2)\rangle$. Let $\mathcal{T}(\lambda)=\{\{t_1\},\ldots,\{t_d\}\}$. Define
$$ t=\{t_1\}\wedge \{t_2\}\wedge\cdots\wedge \{t_{d-x-1}\}\wedge \{t_{d-x}\}\wedge \{t_{d-x+1}\}\wedge\cdots\wedge \{t_d\}\in \Lambda^d M^\lambda,$$
where $\{t_i\}\in S $ for all $d-x+1\leq i\leq d$. Moreover, for all $1\leq j\leq d-x-1$ and $2\nmid j$, $\{\{t_j\},\{t_{j+1}\}\}$ is an orbit of $\mathcal{T}(\lambda)\setminus S$ under the action of $\langle (1,2)\rangle$. We thus have $(1,2)t=(-1)^yt$. The lemma thus follows.
\end{proof}
Let $a\in \mathbb{N}$ and $M$, $N$ be $\F\sym{n}$-modules. For further discussion, recall that all the inner tensor products of modules are over $\F$ and $\Lambda^a(M\oplus N)\cong\bigoplus_{i=0}^a (\Lambda^iM\otimes\Lambda^{a-i}N)$.
\begin{lem}\label{L;powerd-1}
Let $\lambda\vdash n$ and $d=\dim_\F M^\lambda>2$. If $\lambda\neq (n-1,1)$, then $\Lambda^{d-1}M^\lambda$ is indecomposable if and only if $p=2$, $2\mid n$, $\lambda$ is one of the $m$ partitions $(n-k_i,k_i)$ except for $(n-1,1)$, where $m\in \mathbb{N}$, $2^m\leq n<2^{m+1}$, $2^{i-1}\leq k_i<2^i$ and $k_i\equiv \frac{n-2^m}{2}\pmod {2^{i-1}}$ for any $1\leq i\leq m$. If $\Lambda^{d-1}M^\lambda$ is indecomposable, then $\Lambda^{d-1}M^\lambda\cong Y^\lambda$.
\end{lem}
\begin{proof}
When $p>2$, by \cite[Theorem 2 (a)]{Gill}, $M^\lambda$ is decomposable since $\lambda\neq(n-1,1)$ and $(n)$. So $M^\lambda\cong M\oplus N$ for some non-zero $\F\sym{n}$-modules $M$ and $N$. We set $x=\dim_\F M$, $y=\dim_\F N$ and note that $d=x+y$. We have $\Lambda^xM\otimes\Lambda^{y-1}N\mid \Lambda^{d-1}M^\lambda$ and $\Lambda^{x-1}M\otimes\Lambda^{y}N\mid \Lambda^{d-1}M^\lambda$. So $\Lambda^{d-1}M^\lambda$ is decomposable. When $p=2$, notice that $\Lambda^{d-1}M^\lambda\cong M^\lambda$. By \cite[Theorem 2 (b)]{Gill}, $\Lambda^{d-1}M^\lambda$ is indecomposable if and only if the given case occurs. Moreover, if the case occurs, $\Lambda^{d-1}M^\lambda\cong Y^\lambda$.
\end{proof}

\begin{lem}\label{L;powerd-2}
Let $\lambda\vdash n$, $1<a\in \mathbb{N}$ and $d=\dim_\F M^\lambda$. If $\ell(\lambda)\geq 3$ and $a\leq d-2$, then $\Lambda^{a}M^\lambda$ is decomposable.
\end{lem}
\begin{proof}
Since $\ell(\lambda)\geq 3$, $M^\lambda$ is decomposable by \cite[Theorem 2]{Gill}. So $M^\lambda\cong M\oplus N$ for some non-zero $\F\sym{n}$-modules $M$ and $N$. Let $x=\dim_\F M$, $y=\dim_\F N$ and note that $d=x+y$. If $x\leq a$, as $1<a\leq d-2$, note that $0\leq a-x\leq d-2-x=y-2$ and $1\leq a-x+1\leq d-2-(x-1)=y-1$. So $\Lambda^{x} M\otimes\Lambda^{a-x}N\mid\Lambda^a M^\lambda$ and $\Lambda^{x-1} M\otimes\Lambda^{a-x+1}N\mid\Lambda^aM^\lambda$. If $x>a$, $\Lambda^{a} M\mid\Lambda^a M^\lambda$ and $\Lambda^{a-1} M\otimes N\mid\Lambda^a M^\lambda$. We thus get that $\Lambda^a M^\lambda$ is decomposable.
\end{proof}

We are left with the case $\lambda=(n-r,r)\vdash n$ and $r>1$. In this case, we view $M^\lambda$ naturally as the $\F\sym{n}$-module generated by the subsets of $\mathbf{n}$ having size $r$. Write $P(r)=\{s\subseteq \mathbf{n}: |s|=r\}$ and define a total order $\leq_\ell$ for $P(r)$ as follows. Let $s=\{a_1,\ldots,a_r\}$,  $\tilde{s}=\{b_1,\ldots,b_r\}\in P(r)$, where $a_1<\cdots<a_r$ and $b_1<\cdots<b_r$. Write $s<_\ell \tilde{s}$ if $a_1<b_1$ or there exists some $i\in \mathbb{N}$ such that $1\leq i<r$, $a_j=b_j$ for all $1\leq j\leq i$ and $a_{i+1}<b_{i+1}$. We have $s=_\ell \tilde{s}$ if and only if $s=\tilde{s}$. We label $P(r)$ to be $\{s_1,\ldots, s_d\}$, where $d=\dim_\F M^{(n-r,r)}$ and $s_i<_\ell s_j$ for all $1\leq i<j\leq d$. If $1<a\in \mathbb{N}$ and $a\leq d$, $P(r,a)=\{s_{i_1}\wedge\cdots\wedge s_{i_a}\in \Lambda^aM^{(n-r,r)}:1\leq i_1<\cdots<i_a\leq d\}$ is an $\F$-basis of $\Lambda^aM^{(n-r,r)}$. Given $v=s_{i_1}\wedge\cdots\wedge s_{i_a}\in P(r,a)$, for all $1\leq j\leq n$, define $n_j^v$ to be the number of $j$ in the sets $s_{i_1},\ldots,s_{i_a}$. Put $m_k^v=|\{1\leq i\leq n:n_i^v=k\}|$ for all $k\geq0$. Note that $ra=\sum_{i=1}^nn_i^v$ and $n=\sum_{i\geq0}m_i^v$. Given $u, v\in P(r,a)$, there exists $g\in\sym{n}$ such that $gu=v$ or $gu=-v$ only if $m_k^{u}=m_k^v$ for all $k\geq0$.

We use an example to illustrate all these definitions. Let $a=3$, $\lambda=(3,2)\vdash 5$ and $v=\{1,2\}\wedge\{1,3\}\wedge\{1,4\}\in P(2,3)$. So $\{1,2\}<_\ell \{1,3\}$. We have $n_1^v=3,$ $n_{2}^v=n_{3}^v=n_{4}^v=1$, $n_{5}^v=0$, $m_0^v=m_3^v=1$, $m_1^v=3$ and $m_i^v=0$ for all $i\in \mathbb{N}\setminus\{1,3\}$.

\begin{lem}\label{L;twopart}
Let $a, r\in \mathbb{N}$, $2<2r\leq n$ and $1<a\leq d=\dim_\F M^{(n-r,r)}$. If $a\leq d-2$, then $\Lambda^a M^{(n-r,r)}$ is decomposable.
\end{lem}
\begin{proof}
We shall pick $x$, $y\in P(r,a)$ and show that $m_i^x\neq m_{i}^{y}$ for some $i\geq 0$. This fact implies the correctness of the lemma. Let $u=s_1\wedge\cdots\wedge s_a\in P(r,a)$ and $v=s_2\wedge\cdots\wedge s_a\wedge s_d$. As $a\leq d-2$, $s_a<_\ell s_{d-1}$ and $v\in P(r,a)$. Following the total order $\leq_\ell$, $s_1=\{1,\ldots, r\}$, $s_{d-1}=\{n-r, n-r+2,\ldots, n\}$ and $s_d=\{n-r+1,\ldots, n\}$. Let $S_b=\{\{b,i_1,\ldots,i_{r-1}\}:\{b,i_1,\ldots,i_{r-1}\}\in P(r)\}$ and put $t=|S_b|$. Note that $t={n-1\choose r-1}=\frac{rd}{n}$, $n_1^u=\min\{a,t\}$ and $n_b^u\leq t$ for all $b\in \mathbf{n}$. We distinguish two cases.
\begin{enumerate}[\text{Case} 1:]
\item $a\geq t.$
\end{enumerate}
In this case, $n_1^u=t$. As $a\leq d-2$, neither $s_{d-1}$ nor $s_d$ occurs in $u$ as a component. So, for all $n-r+2\leq i\leq n$, we have $n_i^u\leq t-2$. If $n_{n-r+1}^u\leq t-2$, then we have $m_t^u>m_t^v$, which implies that $\Lambda^a M^{(n-r,r)}$ is decomposable. If $n_{n-r+1}^u=t-1$ and $n-r>2$, then $\{n-r,n-r+1,\ldots, n-1\}$ has to occur in $u$ as a component. By the definitions of $u$ and $\leq_\ell$,
as $r>1$, notice that $n_2^u=t$. So we still have $m_t^u>m_t^v$ and $\Lambda^a M^{(n-r,r)}$ is decomposable. If $n-r=2$ and $n_3^u=t-1$, as $2<2r\leq n$, $n_3^u=\frac{rd}{n}-1=2$, $a=4$ and $u=\{1,2\}\wedge\{1,3\}\wedge\{1,4\}\wedge\{2,3\}$. Set $z=\{1,2\}\wedge\{1,3\}\wedge\{2,4\}\wedge\{3,4\}$ and note that $m_2^u=2<4=m_2^z$. This means that $\Lambda^4M^{(2,2)}$ is decomposable.
\begin{enumerate}[\text{Case} 2:]
\item $a<t.$
\end{enumerate}
In this case, $n_1^u=a$. As $1<a$, both $\{1,\ldots,r\}$ and $\{1,\ldots,r-1,r+1\}$ occur in $u$ as components. So, for all $n-r+2\leq i\leq n$, $n_i^u\leq a-2$. If $2r<n$, then $r+1<n-r+1$ and $n_{n-r+1}^u\leq a-2$ as well. So $m_a^u>m_a^v$ and $\Lambda^a M^{(n-r,r)}$ is decomposable. If $2r=n$ and $a>2$, note that $\{1,\ldots,r\}$, $\{1,\ldots,r-1,r+1\}$ and $\{1,\ldots, r-1,r+2\}$ occur in $u$ as components. In particular, we have $n_{r+1}^u\leq a-2$. Therefore, following the case $2r<n$, we have $m_a^u>m_a^v$ and $\Lambda^a M^{(n-r,r)}$ is decomposable. If $2r=n$ and $a=2$, then $u=\{1,\ldots, r\}\wedge\{1,\ldots,r-1, r+1\}$. Set $z=\{1,\ldots,r\}\wedge\{r+1,\ldots,2r\}$. We have $m_2^u=r-1>0=m_2^z$, which implies that $\Lambda^a M^{(n-r,r)}$ is decomposable. The lemma follows by combining the two listed cases.
\end{proof}

We now summarize what we have got by the following theorem. The proof of it is from Lemmas \ref{L;(n-1,1)}, \ref{L;powerd}, \ref{L;powerd-1}, \ref{L;powerd-2} and \ref{L;twopart}.

\begin{thm}\label{T;indecomposable}
Let $\lambda\vdash n$, $1<a\in \mathbb{N}$ and $a\leq d=\dim_\F M^\lambda$. Then $\Lambda^aM^\lambda$ is indecomposable if and only if exactly one of the following cases occurs:
\begin{enumerate}
\item [\em(i)] $a=d$. In this case, \[\Lambda^d M^\lambda\cong\begin{cases}
\F,& \text{if}\ p=2\ \text{or}\ p>2\ \text{and}\ 4\mid f^\lambda,\\
sgn, & \text{otherwise};\\\end{cases}\]
\item [\em(ii)] $p=2$, $2\mid n$, $a=d-1$, $\lambda$ is one of the $m$ partitions $(n-k_i,k_i)$ except for $(n-1,1)$, where $m\in \mathbb{N}$, $2^m\leq n<2^{m+1}$, $2^{i-1}\leq k_i<2^i$ and $k_i\equiv \frac{n-2^m}{2}\pmod {2^{i-1}}$ for any $1\leq i\leq m$. In this case, $\Lambda^{d-1}M^\lambda\cong Y^\lambda$;
\item[{\em (iii)}] $p>2$, $p\mid n$, $a<n$, $\lambda=(n-1,1)$, $a=pb+r$, where $b,\ r\in\mathbb{N}_0$ and $0\leq r<p$. In this case, $\Lambda^aM^{(n-1,1)}\cong Y((n-a,1^r)|(pb))$;
\item[{\em (iv)}] $p=2$, $2\mid n$, $a<n$, $\lambda=(n-1,1)$. In this case, $\Lambda^aM^{(n-1,1)}\cong Y^{\overline{(n-a,a)}}$, where $2^m\leq n<2^{m+1}$ for some $m\in \mathbb{N}$ and $2^{i-1}\leq \min \{n-a,a\}<2^i$, $\min\{n-a, a\}\equiv\frac{n-2^m}{2}\pmod {2^{i-1}}$ for some $1\leq i\leq m$.
\end{enumerate}
\end{thm}
\begin{rem}\label{R;Remark}
We have the following remarks.
\begin{enumerate}[(i)]
\item In \cite{Gill}, Gill classified all the indecomposable $\F\sym{n}$-Young permutation modules (see \cite[Theorem 2]{Gill}). As the trivial $\F\sym{n}$-module is indecomposable, by his result and Theorem \ref{T;indecomposable}, all the indecomposable exterior powers of $\F\sym{n}$-Young permutation modules are now classified.
\item Let $1<a\in \mathbb{N}$ and $\lambda\vdash n$. Note that $S^aM^\lambda$ is indecomposable if and only if $\lambda=(n)$. Therefore, all the indecomposable symmetric powers of $\F\sym{n}$-Young permutation modules are also clear.
\end{enumerate}
\end{rem}

\section{The symmetric powers of Young permutation modules}
The first part of Theorem \ref{T;B} is proved in this section. Let $1<a\in \mathbb{N}$ and $\lambda\vdash n.$ We can find a Young module $Y^\mu$ such that $Y^\mu\mid S^aM^\lambda$ and
$c_{\sym{n}}(Y^\mu)=c_{\sym{n}}(S^aM^\lambda)$. We begin with fixing the required notation.

\begin{nota}\label{N;notation}
Let $\lambda=(\lambda_1,\ldots,\lambda_\ell)\vdash n$ and $1<a\in \mathbb{N}$.
\begin{enumerate}[(i)]
\item Let $m\in \mathbb{N}$, $\mu=(\mu_1,\ldots,\mu_{\ell'})\models m$ and $\{t\}\in \mathcal{T}(\mu)$. For all $1\leq i\leq \ell'$, if $\mu_i\neq 0$, let $R_i(\{t\})$ be the set of entries in the $i\mathrm{th}$ row of $\{t\}$. Otherwise, $R_i(\{t\})=\varnothing$. Put $d_{\lambda,a}=\min\{\nu_p(a),\lfloor \frac{n}{p}\rfloor-\mathrm{rank}(\sym{\lambda})\}$ and $e_{\lambda,a}=n-pd_{\lambda,a}$. Note that $e_{\lambda,a}\geq 0$. Let the numbers $1,\ldots,m$ replace the boxes of $\mu$ from left to right and from the top to the bottom successively and use $t^\mu$ to denote the obtained $\mu$-tableau.
\item For all $1\leq i\leq \ell$, write $\lambda_i=pq_i+r_i$, where $q_i$, $r_i\in \mathbb{N}_0$ and $0\leq r_i<p$. Set $q_\lambda=(pq_1,\ldots,pq_\ell)$ and
   $r_\lambda=(r_1,\ldots,r_\ell)$. Note that $|{q_\lambda}|=\mathrm{rank}(\sym{\lambda})p$ and $|r_\lambda|=n-\mathrm{rank}(\sym{\lambda})p\geq pd_{\lambda,a}$. Moreover, $q_\lambda$ is a non-increasing sequence. Let $r_{\lambda,a}$ be the composition obtained from $r_{\lambda}$ by removing $pd_{\lambda,a}$ boxes from right to left and from the bottom to the top successively. Let $\lambda^a=q_\lambda+r_{\lambda,a}\models e_{\lambda,a}$.
\item Define a $\lambda$-tabloid $\mathfrak{s}_\lambda$ as follows. If $0<\mathrm{rank}(\sym{\lambda})<\frac{n}{p}$, for all $1\leq i\leq \ell$, put
\begin{align*}
R_i(\mathfrak{s}_\lambda)=\begin{cases}R_i(\{t^{q_\lambda}\})\cup \{\mathrm{rank}(\sym{\lambda})p+x:x\in R_i(\{t^{r_{\lambda}}\})\}, & \text{if}\ r_i>0,\\
R_i(\{t^{q_\lambda}\}), & \text{if}\ r_i=0.\end{cases}
\end{align*}
Otherwise, set $\mathfrak{s}_\lambda=\{t^\lambda\}$. For the case $e_{\lambda,a}>0$,
define a $\lambda^a$-tabloid $\mathfrak{s}_{\lambda,a}$ as follows. If $0<e_{\lambda,a}<n$, for all $1\leq i\leq \ell$, set $R_i(\mathfrak{s}_{\lambda,a})=R_i(\mathfrak{s}_\lambda)\setminus \{e_{\lambda,a}+1,\ldots,n\}$. If $e_{\lambda,a}=n$, we have $d_{\lambda,a}=0$, $r_{\lambda,a}=r_\lambda$ and $\lambda^a=\lambda$. Set $\mathfrak{s}_{\lambda,a}=\mathfrak{s}_{\lambda}$.
\item Let   $H_{\lambda,a}=\{g\in\mathfrak{S}_{e_{\lambda,a}}\!:g\mathfrak{s}_{\lambda,a}=\mathfrak{s}_{\lambda,a}\}\leq \sym{e_{\lambda,a}}$ for the case $e_{\lambda,a}>0$. If $e_{\lambda,a}=0$, put $H_{\lambda,a}=1$. If $d_{\lambda,a}>0$, define a $p$-cycle $s_{\lambda,a,i}=(e_{\lambda,a}+(i-1)p+1,\ldots, e_{\lambda,a}+ip)$ for all $1\leq i\leq d_{\lambda,a}$. Let $P$ be a Sylow $p$-subgroup of $H_{\lambda,a}$ and put
$$P_{\lambda,a}=\begin{cases}
  P\times\langle\bigcup_{i=1}^{d_{\lambda,a}}\{s_{\lambda,a,i}\}\rangle, &\text{if}\ d_{\lambda,a}>0,\\
  P, & \text{if}\ d_{\lambda,a}=0.\\
  \end{cases}$$
    So $P_{\lambda,a}$ is unique up to $H_{\lambda,a}$-conjugation. Let $\mathfrak{S}_{\lambda,a}= H_{\lambda,a}\times (\sym{p})^{d_{\lambda,a}},$
    where, if $d_{\lambda,a}>0$, the first factor $\sym{p}$ acts on the set $\{e_{\lambda,a}+1,\ldots,e_{\lambda,a}+p\}$, the second factor $\sym{p}$ acts on the set $\{e_{\lambda,a}+p+1,\ldots,e_{\lambda,a}+2p\}$ and so on. So $P_{\lambda,a}$ is a Sylow $p$-subgroup of $\mathfrak{S}_{\lambda,a}$ and $\mathfrak{S}_{\lambda,a}$ is $\sym{n}$-conjugate to $\sym{\eta}$, where $\eta=\lambda^a\cup(p^{d_{\lambda,a}})$.
\item Let $d=\dim_\F M^\lambda$ and label $\mathcal{T}(\lambda)=\{\{t_1\},\ldots,\{t_d\}\}$, where $\{t_1\}=\{t^\lambda\}$. By this order and $(2.1)$, $\mathcal{T}(\lambda)_a^s$ is defined. Given $t\in \mathcal{T}(\lambda)_a^s$, let $\mathcal{O}^s(t)$ be the orbit of $\mathcal{T}(\lambda)_a^s$ containing $t$ under the action of $\sym{n}$ and set $V(t)=\langle \mathcal{O}^s(t)\rangle_\F$. Therefore, $V(t)\cong (\F_{K(t)}){\uparrow}^{\sym{n}}$, where
$K(t)=\{g\in\sym{n}: gt=t\}\leq \sym{n}$. For some $x\in \mathbb{N}$,
    \begin{equation}
    S^aM^\lambda=\bigoplus_{i=1}^x V(t_i)\cong\bigoplus_{i=1}^x (\F_{K(t_i)}){\uparrow}^{\sym{n}},
    \end{equation}
    where $x\leq d$ and $t_1,\ldots,t_x\in \mathcal{T}(\lambda)_a^s$. If $a\leq d$, by $(2.2)$, this order defines $\mathcal{T}(\lambda)_a^e$. Given $t\in \mathcal{T}(\lambda)_a^e$, let $\mathcal{O}^e(t)=\{s\in\mathcal{T}(\lambda)_a^e:\exists\ g\in\sym{n},\ gt=s\ \text{or}\ gt=-s\}$ and $W(t)=\langle \mathcal{O}^e(t)\rangle_\F$. Let $L(t)=\{g\in\sym{n}:gt=t\ \text{or}\ gt=-t\}\leq \sym{n}$. Note that $W(t)\cong (\F ^t){\uparrow^{\sym{n}}}$, where $\F ^t$ is a one-dimensional $\F L(t)$-module. If $t'\in \O^e(t)$, also note that $L(t')$ is $\sym{n}$-conjugate to $L(t)$ and $(\F^{t'}){\uparrow^{\sym{n}}}\cong (\F^{t}){\uparrow^{\sym{n}}}$. Moreover, $\F ^{t}=\F$ if $p=2$. For some $y\in \mathbb{N}$, $y\leq d$ and $t_1,\ldots,t_y\in \mathcal{T}(\lambda)_a^e$,
     \begin{equation}
    \Lambda^aM^\lambda=\bigoplus_{i=1}^y W(t_i)\cong\bigoplus_{i=1}^y (\F ^{t_i}){\uparrow}^{\sym{n}}.
    \end{equation}
\end{enumerate}
\end{nota}
\begin{eg}\label{E;example}
An example illustrates most of the definitions in Notation \ref{N;notation}. Let $p=3$, $n=11$, $a=6$ and $\lambda=(5,4,2)\vdash 11$. We have $d_{\lambda,a}=1$, $e_{\lambda,a}=8$, $q_\lambda=(3^2,0)\models 6$, $r_\lambda=(2,1,2)\models 5$, $r_{\lambda,a}=(2)\vdash2$ and $\lambda^a=(5,3,0)\models 8$. Moreover,
\[\{t^\lambda\}=\ {\begin{matrix}
\hline
1 & 2 & 3 & 4 & 5\\ \hline
6 & 7 & 8 & 9\\ \cline{1-4}
10 & 11 \\ \cline{1-2}
\end{matrix}},\
\mathfrak{s}_{\lambda}=\ {\begin{matrix}
\hline
1 & 2 & 3 & 7 & 8\\ \hline
4 & 5 & 6 & 9\\ \cline{1-4}
10 & 11 \\ \cline{1-2}
\end{matrix}},\
\mathfrak{s}_{\lambda,a}=\ {\begin{matrix}
\hline
1 & 2 & 3 & 7 & 8\\ \hline
4 & 5 & 6 \\ \cline{1-3}
\\ \cline{1-3}
\end{matrix}},
\]
$H_{\lambda,a}=\sym{\{1,2,3,7,8\}}\times \sym{\{4,5,6\}}$ and $\mathfrak{S}_{\lambda,a}=H_{\lambda,a}\times \sym{\{9,10,11\}}$. Also note that $P_{\lambda,a}$ is $H_{\lambda,a}$-conjugate to $\langle (1,2,3), (4,5,6),(9,10,11)\rangle$.
\end{eg}
\begin{lem}\label{L;defi}
Let $\lambda\vdash n$ and $1<a\in\mathbb{N}$. Let $\mathcal{O}_{\lambda,a}$ be the orbit of $\mathcal{T}(\lambda)$ containing $\mathfrak{s}_\lambda$ under the action of $P_{\lambda,a}$.
\begin{enumerate}
\item[{\em (i)}] For any $\{t\}\in \mathcal{O}_{\lambda,a}$, if $e_{\lambda,a}>0$, then $R_i(\mathfrak{s}_{\lambda,a})\subseteq R_i(\{t\})$ for all $1\leq i\leq \ell(\lambda)$.
\item [{\em (ii)}] We have $|\mathcal{O}_{\lambda,a}|=p^{d_{\lambda,a}}$. In particular, $|\mathcal{O}_{\lambda,a}|\mid a$.
\end{enumerate}
\end{lem}
\begin{proof}
By the definition of $P_{\lambda,a}$, $P_{\lambda,a}=P\times E$, where $P\leq H_{\lambda,a}$, $E\leq\sym{n}$ and $|E|=p^{d_{\lambda,a}}$. Moreover, if $e_{\lambda,a}>0$, for any $e\in E$ and
$x\in\{1,\ldots,e_{\lambda,a}\}$, $e(x)=x$.

For (i), as $e_{\lambda,a}>0$, by the definition of $\mathfrak{s}_{\lambda,a}$, $R_i(\mathfrak{s}_{\lambda,a})\subseteq R_i(\mathfrak{s}_\lambda)$ for all $1\leq i\leq \ell(\lambda)$. As $H_{\lambda,a}$ fixes $\mathfrak{s}_{\lambda,a}$ and all the entries of $\mathfrak{s}_{\lambda,a}$ are exactly the numbers $1,\ldots,e_{\lambda,a}$, if $g\mathfrak{s}_\lambda=\{t\}$ for some $g\in P_{\lambda,a}$, we have $g\mathfrak{s}_{\lambda,a}=\mathfrak{s}_{\lambda,a}$ and $R_i(g\mathfrak{s}_{\lambda,a})\subseteq R_i(g\mathfrak{s}_\lambda)=R_i(\{t\})$ for all $1\leq i\leq \ell(\lambda)$. The two facts imply the correctness of (i).

For (ii), if $e_{\lambda,a}=n$, then $d_{\lambda,a}=0$, $\mathfrak{s}_\lambda=\mathfrak{s}_{\lambda,a}$ and $P_{\lambda,a}=P\leq H_{\lambda,a}$. Therefore, $P_{\lambda,a}$ fixes $\mathfrak{s}_\lambda$ and $|\mathcal{O}_{\lambda,a}|=1$. If $0<e_{\lambda,a}<n$, by (i), $P$ fixes $\mathfrak{s}_\lambda$. By the definition of $\mathfrak{s}_\lambda$, for all $1\leq i\leq \ell(\lambda)$,
$|R_i(\mathfrak{s}_\lambda)\cap\{e_{\lambda,a}+1,\ldots,n\}|<p$. As $e(x)=x$ for any $e\in E$ and $x\in\{1,\ldots,e_{\lambda,a}\}$, by Lemma \ref{PModules}, this implies that $e\mathfrak{s}_\lambda\neq\mathfrak{s}_\lambda$ for all $1\neq e\in E$. So $|\O_{\lambda,a}|=p^{d_{\lambda,a}}$. If $e_{\lambda,a}=0$, by the definitions of $H_{\lambda,a}$ and $d_{\lambda,a}$, $P_{\lambda,a}=E$ and every part of $\lambda$ is strictly less than $p$. By Lemma \ref{PModules}, $e\mathfrak{s}_\lambda\neq\mathfrak{s}_\lambda$ for all $1\neq e\in E$. So $|\O_{\lambda,a}|=p^{d_{\lambda,a}}$. The first assertion is shown. As $d_{\lambda,a}\leq \nu_p(a)$ and $|\O_{\lambda,a}|=p^{d_{\lambda,a}}$, $|\mathcal{O}_{\lambda,a}|\mid a$, as desired.
\end{proof}

Let $\lambda\vdash n$ and $1<a\in \mathbb{N}$. Let $\mathcal{O}_{\lambda,a}$ be the orbit of $\mathcal{T}(\lambda)$ containing $\mathfrak{s}_\lambda$ under the action of $P_{\lambda,a}$. Let $o_{\lambda,a}=|\mathcal{O}_{\lambda,a}|$ and $\mathcal{O}_{\lambda,a}=\{\{s_1\},\ldots,\{s_{o_{\lambda,a}}\}\}$. By Lemma \ref{L;defi} (ii), $a=o_{\lambda,a}x_{\lambda,a}$ for some $x_{\lambda,a}\in\mathbb{N}$. For all $1\leq i\leq o_{\lambda,a}$, formally write $$\{s_i\}^{x_{\lambda,a}}=\underbrace{\{s_i\}\odot\cdots\odot \{s_i\}}_{x_{\lambda,a}\ \text{times}}$$ and let $t_{\lambda,a}=\{s_1\}^{x_{\lambda,a}}\odot\cdots\odot \{s_{o_{\lambda,a}}\}^{x_{\lambda,a}}\in \mathcal{T}(\lambda)_a^s$. Notice that $P_{\lambda,a}\leq K(t_{\lambda,a})$. Moreover, $g\{s_i\}\in \mathcal{O}_{\lambda,a}$ for any $g\in K(t_{\lambda,a})$ and $1\leq i\leq o_{\lambda,a}$.
\begin{lem}\label{L;subgroup0}
Let $\lambda\vdash n$ and $1<a\in \mathbb{N}$. If $d_{\lambda,a}=0$, then $K(t_{\lambda,a})=\mathfrak{S}_{\lambda,a}$.
\end{lem}
\begin{proof}
As $d_{\lambda,a}=0$, by the definitions of $\mathcal{O}_{\lambda,a}$, $t_{\lambda,a}$, $\mathfrak{s}_{\lambda,a}$, $\mathfrak{S}_{\lambda,a}$ and Lemma \ref{L;defi} (ii), $e_{\lambda,a}=n$, $\mathcal{O}_{\lambda,a}=\{\mathfrak{s}_\lambda\}$, $t_{\lambda,a}={\mathfrak{s}_\lambda}^a$, $\mathfrak{s}_\lambda=\mathfrak{s}_{\lambda,a}$ and
$K(t_{\lambda,a})=H_{\lambda,a}=\mathfrak{S}_{\lambda,a}$. This completes the proof.
\end{proof}

\begin{lem}\label{L;subgroup1}
Let $\lambda\vdash n$ and $1<a\in \mathbb{N}$. If $d_{\lambda,a}>0$, then $K(t_{\lambda,a})= H_{\lambda,a}\times K$, where $K\leq \sym{\{e_{\lambda,a}+1,\ldots,n\}}$.
\end{lem}
\begin{proof}
As $d_{\lambda,a}>0$, we have $0\leq e_{\lambda,a}<n$. If $e_{\lambda,a}=0$, by the definition of $H_{\lambda,a}$, $H_{\lambda,a}=1$. The desired equality holds trivially. We thus assume that $0<e_{\lambda,a}<n$.

For any $g\in K(t_{\lambda,a})$, we claim that there does not exist a pair $i$, $j\in \mathbf{n}$ such that $i\in\{1,\ldots,e_{\lambda,a}\}$, $j\in\{e_{\lambda,a}+1,\ldots,n\}$ and $g(j)=i$. Suppose that such $g$, $i$ and $j$ exist. Then, for any $\{t\}\in \mathcal{O}_{\lambda,a}$, $i$ and $j$ must lie in the same row of $\{t\}$. Otherwise, if $i\in R_u(\{t\})$, $j\in R_v(\{t\})$ and $u\neq v$, by Lemma \ref{L;defi} (i), $i\in R_u(\mathfrak{s}_{\lambda,a})\cap R_v(g\{t\})$. As $g\in K(t_{\lambda,a})$, $g\{t\}\in\mathcal{O}_{\lambda,a}$. By Lemma \ref{L;defi} (i) again, $i\in R_u(g\{t\})$. This is a contradiction. We thus assume that $i$, $j\in R_u(\{t\})$. As $d_{\lambda,a}>0$, by the definition of $d_{\lambda,a}$, $\lambda\neq (n)$. By the definition of $P_{\lambda,a}$, note that there exist some $m$, $w\in \mathbb{N}$ and $k\in\{e_{\lambda,a}+1,\ldots,n\}\setminus\{j\}$ such that $1\leq m<p$, $1\leq w\leq d_{\lambda,a}$ and  $s^m_{\lambda,a,w}(k)=j$. As $\{t\}\in \O_{\lambda,a}$, by the definition of $\O_{\lambda,a}$, $|R_u(\{t\})\cap\{e_{\lambda,a}+1,\ldots,n\}|<p$. Therefore, we may require that $k\in R_v(\{t\})$ and $u\neq v$. We have $i\in R_u(s^m_{\lambda,a,w}\{t\})$ and $j\in  R_v(s^m_{\lambda,a,w}\{t\})$. Since $s_{\lambda,a,w}\in P_{\lambda,a}$, $s^m_{\lambda,a,w}\{t\}\in \mathcal{O}_{\lambda,a}$. By above discussion, we get a contradiction as $i$, $j$ are not in the same row of $s^m_{\lambda,a,w}\{t\}$. The claim is shown.

By this claim, $K(t_{\lambda,a})\leq H\times K$, where $H$ is the projection of $K(t_{\lambda,a})$ with respect to $\sym{e_{\lambda,a}}$ and $K$ is the projection of $K(t_{\lambda,a})$ with respect to $\sym{\{e_{\lambda,a}+1,\ldots,n\}}$. Note that $H$ fixes $\mathfrak{s}_{\lambda,a}$. Otherwise, if there exist some $h\in H$, $i\in R_c(\mathfrak{s}_{\lambda,a})$ and $j\in R_d(\mathfrak{s}_{\lambda,a})$ such that $c\neq d$ and $h(i)=j$, as there exists some $x\in K(t_{\lambda,a})$ such that $x=hk$ and $k\in K$, we get that $x(i)=j$ and $j\in R_c(x\mathfrak{s}_{\lambda,a})\cap R_d(\mathfrak{s}_{\lambda,a})$. By Lemma \ref{L;defi} (i), $j\in R_c(x\mathfrak{s}_\lambda)\cap R_d(x\mathfrak{s}_\lambda)$, which is a contradiction. So $H\leq H_{\lambda,a}$. Since $H_{\lambda,a}\leq \sym{e_{\lambda,a}}$, by Lemma \ref{L;defi} (i) again, $H_{\lambda,a}\leq K(t_{\lambda,a})$. We have $H=H_{\lambda,a}$ as $H\leq H_{\lambda,a}\leq H\cap K(t_{\lambda,a})$. This equality implies that $K\leq K(t_{\lambda,a})$ and $K(t_{\lambda,a})=H_{\lambda,a}\times K$. The proof is now complete.
\end{proof}
\begin{lem}\label{L;subgroup2}
Let $\lambda\vdash n$ and $1<a\in \mathbb{N}$. If $d_{\lambda,a}>0$, then $K(t_{\lambda,a})\leq \mathfrak{S}_{\lambda,a}.$
\end{lem}
\begin{proof}
By Lemma \ref{L;subgroup1} and the definitions of $d_{\lambda,a}$ and $\mathfrak{S}_{\lambda,a}$, as $d_{\lambda,a}>0$, we have $0\leq e_{\lambda,a}<n$, $\lambda\neq (n)$, $K(t_{\lambda,a})=H_{\lambda,a}\times K$ and $K\leq \sym{\{e_{\lambda,a}+1,\ldots,n\}}$. It is enough to show that $K$ is contained in $$\underbrace{\sym{p}\times\cdots\times\sym{p}}_{d_{\lambda,a}\ \text{times}},$$
where the first factor $\sym{p}$ acts on the set $\{e_{\lambda,a}+1,\ldots,e_{\lambda,a}+p\}$, the second factor $\sym{p}$ acts on the set $\{e_{\lambda,a}+p+1,\ldots,e_{\lambda,a}+2p\}$ and so on. For all $1\leq u<v\leq d_{\lambda,a}$, we claim that there do not exist $k\in K$, $i,j\in \mathbf{n}$ such that $i\in\{e_{\lambda,a}+(u-1)p+1,\ldots, e_{\lambda,a}+up\}$, $j\in\{e_{\lambda,a}+(v-1)p+1,\ldots, e_{\lambda,a}+vp\}$ and $k(i)=j$. Suppose that such $u$, $v$, $i$, $j$ and $k$ exist. For any $\{t\}\in \mathcal{O}_{\lambda,a}$, by the definitions of $\mathfrak{s}_\lambda$ and $\mathcal{O}_{\lambda,a}$, there exist some $b$, $c$, $d$, $e\in \mathbb{N}$ (independent of the choices of $\{t\}$) such that $b<c\leq d<e$, the numbers $e_{\lambda,a}+(u-1)p+1,\ldots,e_{\lambda,a}+up$ lie in the $b\mathrm{th}$ row,$\ldots,$ $c\mathrm{th}$ row of $\{t\}$, the numbers $e_{\lambda,a}+(v-1)p+1,\ldots,e_{\lambda,a}+vp$ lie in the $d\mathrm{th}$ row, $\ldots$, $e\mathrm{th}$ row of $\{t\}$. Moreover,
$|R_h(\{t\})\cap\{e_{\lambda,a}+1,\ldots,n\}|<p$ for all $b\leq h\leq e$. As $k\in K(t_{\lambda,a})$ by Lemma \ref{L;subgroup1}, we have $k\{t\}\in \mathcal{O}_{\lambda,a}$. Therefore, $j\in R_x(k\{t\})$, where $b\leq x\leq c$ and $d\leq x\leq e$. This forces that $c=d$ and $i\in R_c(\{t\})$. Recall that $s_{\lambda,a,u}=(e_{\lambda,a}+(u-1)p+1,\ldots, e_{\lambda,a}+up)$ and $s_{\lambda,a,u}\in P_{\lambda,a}\leq K(t_{\lambda,a})$. Note that there exists some $m\in \mathbb{N}$ such that $1\leq m<p$, $i\in R_y(s_{\lambda,a,u}^m\{t\})$ and $b\leq y<c$. Note that $s_{\lambda,a,u}^m\{t\}\in \mathcal{O}_{\lambda,a}$ while $i\notin R_c(s_{\lambda,a,u}^m\{t\})$. This is a contradiction. The claim is shown and the desired containment follows.
\end{proof}


Let $\lambda\vdash n$ and $1<a\in \mathbb{N}$. We may not always have $K(t_{\lambda,a})=\mathfrak{S}_{\lambda,a}$. For a counterexample, let $p=a=3$, $n=6$ and $\lambda=(3, 1^3)\vdash 6$. Note that $d_{\lambda,a}=1$, $e_{\lambda,a}=3$, $\mathfrak{s}_\lambda=\{t^{\lambda}\}$ and $P_{\lambda,a}=\langle(1,2,3),(4,5,6)\rangle$. Therefore, $\mathcal{O}_{\lambda,a}$ contains exactly
\[\ {{\begin{matrix}
\hline
1 & 2 & 3 \\ \hline
4 & \\ \cline{1-1}
5 & \\ \cline{1-1}
6 & \\ \cline{1-1}
\end{matrix}}},\ \ \ \ \
{{\begin{matrix}
\hline
1 & 2 & 3 \\ \hline
5 & \\ \cline{1-1}
6 & \\ \cline{1-1}
4 & \\ \cline{1-1}
\end{matrix}}}
,\ \ \ \ \
{{\begin{matrix}
\hline
1 & 2 & 3 \\ \hline
6 & \\ \cline{1-1}
4 & \\ \cline{1-1}
5 & \\ \cline{1-1}
\end{matrix}}}.\]
Also notice that $\mathfrak{S}_{\lambda,a}=\sym{(3^2)}$. Observe that $K(t_{\lambda,a})<\mathfrak{S}_{\lambda,a}$.

We state a result to figure out all the Young modules that are also Scott modules.
\begin{thm}\label{T;Young-Scott}\cite[Propositions 12.1.1, 13.1.2]{CHN}
Let $\lambda\vdash n$ and $\lambda$ have the $p$-adic expansion $\sum_{i=0}^mp^i\lambda(i)$ for some $m\in \mathbb{N}_0$. Then $Y^\lambda$ is a Scott module if and only if $\lambda(i)=((p-1)^{m_i},n_i)$, $m_i\geq0$ and $0\leq n_i<p-1$ for all $0\leq i\leq m$.
\end{thm}
\begin{lem}\label{L;Young-Scott2}
Let $\lambda\vdash n$. Then there exists a unique sum $\sum_{i=0}^mp^i\lambda[i]$ such that $m\in \mathbb{N}_0$, $\lambda[i]$ is a composition and each part of $\lambda[i]$ is no more than $p-1$ for all $0\leq i\leq m$ and $\lambda[i]\neq \varnothing$, $\lambda[m]\neq\varnothing$ and $\lambda=\sum_{i=0}^mp^i\lambda[i]$.
\end{lem}
\begin{proof}
Let $\lambda=(\lambda_1,\ldots,\lambda_\ell)$. For all $1\leq b\leq \ell$, let $\lambda_b=\sum_{c\geq0}p^ca_{b,c}$, where $a_{b,c}\in \mathbb{N}_0$ and $a_{b,c}<p$ for all $c\geq 0$. For all $c\geq 0$, set $\lambda[c]=(a_{1,c},\ldots,a_{\ell,c})$. Note that $\lambda=\sum_{i=0}^mp^i\lambda[i]$ and $\lambda[m]\neq \varnothing$ for some $m\in\mathbb{N}_0$. Moreover, for all $0\leq i\leq m$ and $\lambda[i]\neq \varnothing$, each part of $\lambda[i]$ is no more than $p-1$. We get the existence of the desired sum. For all $1\leq b\leq \ell$ and $c\geq 0$, if $a_{b,c}\neq 0$, then $a_{b,c}$ is uniquely determined by $\lambda_b$. This fact gives us the uniqueness of the sum. The proof is now complete.
\end{proof}
Let $\lambda\vdash n$ and $\lambda=\sum_{i=0}^mp^i\lambda[i]$ satisfy all the conditions given in Lemma \ref{L;Young-Scott2}. For all $0\leq i\leq m$, let $|\lambda[i]|=a_i(p-1)+b_i$, where $a_i$, $b_i\in \mathbb{N}_0$ and $b_i<p-1$. Put $\lambda(i)=\overline{((p-1)^{a_i}, b_i)}$ for all $0\leq i\leq m$ and set $s(\lambda)=\sum_{i=0}^mp^i\lambda(i)$. By Lemma \ref{L;Young-Scott2}, $s(\lambda)$ is a well-defined partition. Moreover, $Y^{s(\lambda)}\mid M^\lambda$ and $Y^{s(\lambda)}$ is a Scott module by Theorems \ref{T;Donkin} and
\ref{T;Young-Scott}. We now finish the proof of the first part of Theorem \ref{T;B}.
\begin{prop}\label{P;symmetriccase}
Let $\lambda\vdash n$ and $1<a\in \mathbb{N}$. Let $\mu=s(\lambda^a\cup(p^{d_{\lambda,a}}))$. Then $Y^\mu\mid S^aM^\lambda$ and $c_{\sym{n}}(Y^\mu)=c_{\sym{n}}(S^aM^\lambda)$.
\end{prop}
\begin{proof}
Let $\eta=\lambda^a\cup(p^{d_{\lambda,a}})$. By (5.1), note that $(\F_{K(t_{\lambda,a})}){\uparrow^{\sym{n}}}\mid S^aM^\lambda$. As $P_{\lambda,a}$ is a Sylow $p$-subgroup of $\mathfrak{S}_{\lambda,a}$ and $P_{\lambda,a}\leq K(t_{\lambda,a})$, by Lemmas \ref{L;subgroup0} and \ref{L;subgroup2}, $P_{\lambda,a}$ is also a Sylow $p$-subgroup of $K(t_{\lambda,a})$. Therefore, $Sc_{\sym{n}}(P_{\lambda,a})\mid (\F_{K(t_{\lambda,a})}){\uparrow^{\sym{n}}}$. In particular, $Sc_{\sym{n}}(P_{\lambda,a})\mid S^aM^\lambda$. By the definition of $\mathfrak{S}_{\lambda,a}$, also notice that $\mathfrak{S}_{\lambda,a}$ is $\sym{n}$-conjugate to $\sym{\eta}$ and $Sc_{\sym{n}}(P_{\lambda,a})\mid M^\eta$. So $Sc_{\sym{n}}(P_{\lambda,a})\cong Y^\mu$ and $Y^\mu\mid S^aM^\lambda$. It suffices to show that both $Y^\mu$ and $S^aM^\lambda$ have the same complexity. Set $r=\mathrm{rank}(P_{\lambda,a})$. If $d_{\lambda,a}=0$, by the definitions of $P_{\lambda,a}$ and $H_{\lambda,a}$, $P_{\lambda,a}$ is $\sym{n}$-conjugate to a Sylow $p$-subgroup of $\sym{\lambda}$. In particular, $r=\mathrm{rank}(\sym{\lambda})$. If $d_{\lambda,a}>0$, by the definitions of $P_{\lambda,a}$ and $\lambda^a$, $r=\mathrm{rank}(\sym{\lambda})+d_{\lambda,a}$. Due to Lemmas \ref{L;complexity} (ii) and \ref{L;symmetriccomplexity}, both $Y^\mu$ and $S^aM^\lambda$ have the same complexity. The lemma follows.
\end{proof}
For an example, in Example \ref{E;example}, the detected partition $\mu$ in Proposition \ref{P;symmetriccase} is $(8,3)$. We conclude this section by providing two other examples.
\begin{eg}\label{E;counterexample}
We have the following examples.
\begin{enumerate}[(i)]
 \item Let $p=2$ and $D^{(3,1)}$ be the unique $2$-dimensional simple module of $\F\sym{4}$. Then $D^{(3,1)}\mid S^2M^{(2^2)}$ and $D^{(3,1)}\mid \Lambda^2M^{(2^2)}$. Furthermore, we also obtain that $c_{\sym{4}}(D^{(3,1)})=c_{\sym{4}}(S^2M^{(2^2)})=c_{\sym{4}}(\Lambda^2M^{(2^2)})=2$ by Lemmas \ref{L;complexity} (i), \ref{L;compare} and \ref{L;symmetriccomplexity}. However, $D^{(3,1)}$ is not a Young module.
  \item Let $p=2$. Then $Y^{(4,2)}\mid S^4M^{(5,1)}$ and $Y^{(4,2)}\mid \Lambda^4 M^{(5,1)}$. Moreover, by Lemmas \ref{L;complexity} (i), \ref{HemmerNakano}, \ref{L;compare} and \ref{L;symmetriccomplexity}, $c_{\sym{6}}(Y^{(4,2)})=c_{\sym{6}}(S^4M^{(5,1)})=c_{\sym{6}}(\Lambda^4M^{(5,1)})=3$. However, although the Young module that we have found in Proposition \ref{P;symmetriccase} is a Scott module, note that $Y^{(4,2)}$ is not a Scott module by Theorem \ref{T;Young-Scott}.
\end{enumerate}
\end{eg}
\section{The exterior squares of Young permutation modules}
In this section, we finish the remaining part of Theorem \ref{T;B}. For our purpose, recall the definitions in Notation \ref{N;notation}. Given $\lambda=(\lambda_1,\ldots,\lambda_\ell)\models n$ and $\{t\}\in \mathcal{T}(\lambda)$, put $R(\{t\})=\{g\in\sym{n}: g\{t\}=\{t\}\}$. Note that $R(\{t\})=\prod_{i=1}^\ell\sym{R_i(\{t\})}$. So $R(\{t\})$ is $\sym{n}$-conjugate to $\sym{\overline{\lambda}}$, where we recall that $\overline{\lambda}$ is the partition whose parts are exactly the non-zero parts of $\lambda$. If $\lambda \vdash n$, also notice that $\ell(r_\lambda)=|\{1\leq i\leq \ell:\ p\nmid\lambda_i\}|$.
\begin{lem}\label{L;subgroup3}
Let $\lambda\vdash n$ and $\lambda\neq (n)$. Let $\{u\}$, $\{v\}\in\mathcal{T}(\lambda)$ and $\{u\}\neq\{v\}$. Let $t$ denote $\{u\}\odot\{v\}$ or $\{u\}\wedge\{v\}$, where $\{u\}\wedge\{v\}\in \mathcal{T}(\lambda)_2^e$. Correspondingly, let $H(t)$ be $K(t)$ or $L(t)$. If there does not exist some $g\in H(t)$ such that $g\{u\}=\{v\}$ and $g\{v\}=\{u\}$, then $H(t)=R(\{u\})\cap R(\{v\})$.
\end{lem}
\begin{proof}
By the hypotheses, for any $h\in H(t)$, $h\{u\}=\{u\}$ and $h\{v\}=\{v\}$. Therefore, we have
$H(t)\leq R(\{u\})\cap R(\{v\})\leq H(t)$. The lemma follows.
\end{proof}

\begin{lem}\label{L;subgroup4}
Let $\lambda\vdash n$ and $\lambda\neq (n)$. Let $\{u\}$, $\{v\}\in\mathcal{T}(\lambda)$ and $\{u\}\neq\{v\}$. Let $t$ denote $\{u\}\odot\{v\}$ or $\{u\}\wedge\{v\}$, where $\{u\}\wedge\{v\}\in \mathcal{T}(\lambda)_2^e$. Correspondingly, let $H(t)$ be $K(t)$ or $L(t)$. If there exists some $g\in H(t)$ such that $g\{u\}=\{v\}$ and $g\{v\}=\{u\}$, then there exists some involution $e$ of $\sym{n}$ such that
$e\{u\}=\{v\}$, $e\{v\}=\{u\}$ and $H(t)=(R(\{u\})\cap R(\{v\}))\rtimes \langle e\rangle$ (the semidirect product of these groups). Moreover, if $g^2=1$, $e$ can be chosen to be $g$.
\end{lem}
\begin{proof}
Let $R=R(\{u\})\cap R(\{v\})$. For any $h\in H(t)$, if $h\{u\}=\{v\}$ and $h\{v\}=\{u\}$, we claim that $H(t)=R\cup hR$. Note that $R\leq H(t)$. For any $k\in H(t)$, if $k\{u\}=\{u\}$ and $k\{v\}=\{v\}$, then $k\in R$. If $k\{u\}=\{v\}$ and $k\{v\}=\{u\}$, then $h^{-1}k\{u\}=\{u\}$ and $h^{-1}k\{v\}=\{v\}$. Therefore, $k\in hR$. The claim is shown. By this claim and the hypotheses, $g\in N_{\sym{n}}(R)$. If $g^2=1$, as $g\not\in R$, we are done. If $g^2\neq 1$, let $S$ be the set of orbits of $\mathbf{n}$ under the action of $R$. Notice that each row of $\{u\}$ or $\{v\}$ is a union of some members of $S$. As $g\in N_{\sym{n}}(R)$, $g$ permutes the members of $S$ with the same size. For any $\mathcal{O}_1$, $\mathcal{O}_2\in S$, if $|\mathcal{O}_1|=|\mathcal{O}_2|$, $\mathcal{O}_1\subseteq R_i(\{u\})$, $\mathcal{O}_2\subseteq R_j(\{u\})$ and $g$ sends $\mathcal{O}_1$ to $\mathcal{O}_2$, $\mathcal{O}_2\subseteq R_i(\{v\})$ as $g\{u\}=\{v\}$. Moreover, note that $\mathcal{O}_1\subseteq R_j(\{v\})$. Otherwise, as $g\{v\}=\{u\}$, if $\mathcal{O}_1\not\subseteq R_j(\{v\})$, $\mathcal{O}_2\not\subseteq R_j(\{u\})$. This is a contradiction. Therefore, for all $\mathcal{O}_1$, $\mathcal{O}_2\in S$ satisfying that $\O_1\neq \O_2$ and $g$ sends $\mathcal{O}_1$ to $\mathcal{O}_2$, let $e$ be an involution of $\sym{n}$ swapping all these $\mathcal{O}_1$, $\mathcal{O}_2$. By the above discussion, $e\{u\}=\{v\}$, $e\{v\}=\{u\}$, $e\in N_{\sym{n}}(R)$, $e\notin R$ and $e^2=1$. The lemma thus follows.
\end{proof}
\begin{lem}\label{L;preparation}
Let $\lambda=(\lambda_1,\ldots,\lambda_\ell)\vdash n$, $\lambda\neq(n)$ and $r=\mathrm{rank}(\sym{\lambda})$. Let $a\in R_i(\{t^\lambda\})$ and $b\in R_j(\{t^\lambda\})$, where $1\leq i\neq j\leq \ell$. Let $t=(a,b)t^\lambda$ and $s=\{t^\lambda\}\wedge\{t\}$. Then $L(s)=(R(\{t^\lambda\})\cap R(\{t\}))\times \langle (a,b)\rangle,$ where $R(\{t^\lambda\})\cap R(\{t\})$ is $\sym{n}$-conjugate to a Young subgroup of $\sym{n}$ with $p$-rank
\[\begin{cases} r, &\text{if}\ p\nmid\lambda_i\ \text{and}\ p\nmid \lambda_j,\\
r-1, & \text{if}\ p\nmid\lambda_i\ \text{and}\ p\mid\lambda_j\ \text{or}\  p\mid\lambda_i\ \text{and}\ p\nmid\lambda_j,\\
r-2, & \text{if}\ p\mid\lambda_i\ \text{and}\ p\mid\lambda_j.\end{cases}
\]
\end{lem}
\begin{proof}
Observe that $(a,b)\{t^\lambda\}=\{t\}$ and $(a,b)\{t\}=\{t^\lambda\}$. By Notation \ref{N;notation} (v), note that $s=\{t^\lambda\}\wedge\{t\}\in\mathcal{T}(\lambda)_2^e$. Also notice that $a$ and $b$ are fixed under the action of $R(\{t^\lambda\})\cap R(\{t\})$ on $\mathbf{n}$. Therefore, by Lemma \ref{L;subgroup4},
$L(s)=(R(\{t^\lambda\})\cap R(\{t\}))\times \langle (a,b)\rangle$. Let $S_i=R_i(\{t^\lambda\})\setminus\{a\}$ and $S_j=R_j(\{t^\lambda\})\setminus\{b\}$. By the definition of $\{t\}$,
\begin{align}
R(\{t^\lambda\})\cap R(\{t\})=\sym{S_i}\times\sym{S_j}\times \prod_{\substack{k=1\\
                  k\neq i,j\\
                  } }^{\ell}\sym{R_k(\{t^\lambda\})}\leq\sym{\lambda}.
\end{align}
So $R(\{t^\lambda\})\cap R(\{t\})$ is $\sym{n}$-conjugate to a Young subgroup of $\sym{n}$ with the desired $p$-rank. The proof is now complete.
\end{proof}
\begin{lem}\label{L;np=2}
Let $p>2$ and $\lambda=(\lambda_1,\ldots,\lambda_\ell)\vdash n$. If $p\nmid\lambda_i$ and $p\nmid \lambda_j$ for some $1\leq i<j\leq\ell$, let $\nu=(\lambda_1,\ldots,\lambda_i-1,\ldots,\lambda_j-1,\ldots,\lambda_\ell)$ and $\mu=\overline{r_{\nu\cup(1^2)}}+\overline{q_{\nu\cup(1^2)}}$. Then $Y^\mu\mid \Lambda^2M^\lambda$ and $c_{\sym{n}}(Y^\mu)=c_{\sym{n}}(\Lambda^2M^\lambda)=\mathrm{rank}(\sym{\lambda})$.
\end{lem}
\begin{proof}
Let $r=\mathrm{rank}(\sym{\lambda})$ and $a_k\in R_k(\{t^\lambda\})$ for any $k\in\{i,j\}$. Let $t=(a_i,a_j)t^\lambda$ and note that $\{t\}\neq\{t^\lambda\}$. By Notation \ref{N;notation} (v), also note that $s=\{t^\lambda\}\wedge\{t\}\in\mathcal{T}(\lambda)_2^e$. As $i\neq j$, by Lemma \ref{L;preparation}, $L(s)=(R(\{t^\lambda\})\cap R(\{t\}))\times\langle (a_i,a_j)\rangle$. Let $\eta=\nu\cup(1^2)$ and $\tilde{\eta}=\nu\cup(2)$. As $p>2$, note that $M^{\eta}\cong M^{\tilde{\eta}}\oplus (\F^s){\uparrow^{\sym{n}}}$ by \cite[Corollary 4.6 (i)]{Lim}. By Theorem \ref{T;Donkin}, observe that $Y^\mu\mid M^\eta$. Let $\mu$ have $p$-adic expansion $\sum_{i=0}^mp^i\mu(i)$ for some $m\in \mathbb{N}_0$. For any a $p$-adic decomposition $\tilde{\eta}=\sum_{i=0}^mp^i\sigma[i]$, if $\sigma[i]\models|\mu(i)|$ for all $0\leq i\leq m$, $\mu(0)=\overline{r_\eta}$ and $\sigma[0]=r_{\tilde{\eta}}$. Moreover, $\ell(r_\eta)= \ell(r_{\overline{\nu}})+2>\ell(r_{\overline{\nu}})+1=\ell(r_{\tilde{\eta}})$. Therefore, $\overline{r_{\tilde{\eta}}}\ntriangleleft\overline{r_{\eta}}$ and $Y^\mu\nmid M^{\tilde{\eta}}$ by Theorem \ref{T;Donkin}, which implies that $Y^\mu\mid (\F^s){\uparrow^{\sym{n}}}$ by the Krull-Schmidt Theorem. As $(\F^s){\uparrow^{\sym{n}}}\mid \Lambda^2 M^\lambda$ by $(5.2)$, $Y^\mu\mid \Lambda^2M^\lambda$. As $p\nmid \lambda_i$ and $p\nmid \lambda_j$, by Lemmas \ref{L;complexity} (i) and \ref{HemmerNakano}, we have $c_{\sym{n}}(Y^\mu)=r\leq c_{\sym{n}}(\Lambda^2M^\lambda)$. Since $p>2$, by Lemmas \ref{L;compare} and \ref{L;symmetriccomplexity}, this inequality implies that
$c_{\sym{n}}(Y^\mu)=c_{\sym{n}}(\Lambda^2M^\lambda)=r$. The proof is now complete.
\end{proof}
\begin{lem}\label{L;two-part}
Let $p>2$ and $\lambda=(\lambda_1,\lambda_2)\vdash n$. If $p\mid \lambda_1$ and $0<\lambda_2<p$, let $\mu=(\lambda_1-1,\lambda_2-1)\cup(1^2)$. Then $Y^\mu\mid \Lambda^2M^\lambda$ and $c_{\sym{n}}(Y^\mu)=c_{\sym{n}}(\Lambda^2M^\lambda)=\mathrm{rank}(\sym{\lambda})-1$.
\end{lem}
\begin{proof}
Let $r=\mathrm{rank}(\sym{\lambda})$ and $t=(1,n)t^\lambda$. As $\{t\}\neq \{t^\lambda\}$, by Notation \ref{N;notation} (v), notice that $s=\{t^\lambda\}\wedge\{t\}\in\mathcal{T}(\lambda)_2^e$. By Lemma \ref{L;preparation}, $L(s)=(R(\{t^\lambda\})\cap R(\{t\}))\times\langle (1,n)\rangle$. As $p>2$, by \cite[Corollary 4.8 (i)]{Lim} and (5.2), we have $Y^\mu\mid (\F^s){\uparrow^{\sym{n}}}$ and $(\F^s){\uparrow^{\sym{n}}}\mid \Lambda^2M^\lambda$, which implies that $Y^\mu\mid \Lambda^2M^\lambda$. As $p\mid \lambda_1$ and $0<\lambda_2<p$, by Lemmas \ref{L;complexity} (i), \ref{HemmerNakano}, \ref{L;compare} and \ref{L;symmetriccomplexity}, note that $r-1=c_{\sym{n}}(Y^\mu)\leq c_{\sym{n}}(\Lambda^2M^\lambda)\leq r$. For any elementary abelian $p$-subgroup $E$ of $\sym{n}$ with $p$-rank $r$, we claim that $\dim_\F M^\lambda(E)\leq 1$. If $M^\lambda(E)\neq 0$, note that $\mathbf{n}$ has exactly $\lambda_2$ fixed points under the action of $E$. Otherwise, as $p\mid \lambda_1$ and $0<\lambda_2<p$, $\mathbf{n}$ has at least $p+\lambda_2$ fixed points under the action of $E$. We thus can choose an elementary abelian $p$-subgroup $F$ of $\sym{n}$ such that $\mathrm{rank}(F)>r$ and $M^\lambda(F)\neq 0$, which contradicts with Lemmas \ref{L;complexity} (v) and \ref{HemmerNakano}. So $\dim_\F M^\lambda(E)=1$ by Lemmas \ref{PModules}. The claim is shown. As $p>2$, by the claim and Lemma \ref{L;exteriorBrauer}, $(\Lambda^2 M^\lambda)(E)=0$. Therefore, $c_{\sym{n}}(Y^\mu)=c_{\sym{n}}(\Lambda^2M^\lambda)=r-1$ by Lemma \ref{L;complexity} (v). The proof is now complete.
\end{proof}
\begin{lem}\label{L;2pparts}
Let $p>2$ and $\lambda=(\lambda_1,\ldots,\lambda_\ell)\vdash n$, where $\lambda_i\geq p$ and $\lambda_j\geq p$ for some $1\leq i<j\leq \ell$. Let $\nu=(\lambda_1,\ldots,\lambda_i-p,\ldots,\lambda_j-p,\ldots,\lambda_\ell)$ and $\mu=\overline{r_{\nu\cup(p,p)}}+\overline{q_{\nu\cup(p,p)}}$. Then $Y^\mu\mid\Lambda^2M^\lambda$ and
$c_{\sym{n}}(Y^\mu)=c_{\sym{n}}(\Lambda^2M^\lambda)=\mathrm{rank}(\sym{\lambda})$.
\end{lem}
\begin{proof}
Let $r=\mathrm{rank}(\sym{\lambda})$, $\O_1=\{1,\ldots,p\}$ and $\O_2=\{p+1,\ldots,2p\}$. Let $\{u\}\in\mathcal{T}(\lambda)$, where we have $\O_1\subseteq R_i(\{u\})$ and $\O_2\subseteq R_j(\{u\})$. Let $\{v\}\in\mathcal{T}(\lambda)$, where we have $R_i(\{v\})=(R_i(\{u\})\setminus\O_1)\cup\O_2$, $R_j(\{v\})=(R_j(\{u\})\setminus\O_2)\cup\O_1$ and $R_k(\{v\})=R_k(\{u\})$ for all $1\leq k\leq \ell$ and $k\notin\{i,j\}$. We may assume that $t=\{u\}\wedge\{v\}\in\mathcal{T}(\lambda)_2^e$. Set $e=\prod_{i=1}^p(i,p+i)$ and note that $L(t)=(R(\{u\})\cap R(\{v\}))\rtimes\langle e\rangle$ by Lemma \ref{L;subgroup4}. Also notice that $L(t)=H\times K$, where the subgroup $H$ is $\sym{n}$-conjugate to $\sym{\overline{\nu}}$ and $K=\sym{(p,p)}\rtimes\langle e\rangle$. Recall that $W(t)\cong (\F^t){\uparrow^{\sym{n}}}$, where $\F^t$ denotes a one-dimensional $\F L(t)$-module. Therefore, $\F^t\cong\F\boxtimes\F'$ as an $\F[H\times K]$-module, where $\F'=\langle \{v\}\rangle_\F$. Note that $ev=-v$ and $gv=v$ for all $g\in \sym{(p,p)}$. Let $M=(\F'){\uparrow^{\sym{2p}}}$. By Lemma \ref{L;subgroup4} and $(5.2)$, observe that $M\mid \Lambda^2M^{(p,p)}$. As $p>2$, it is clear that both $\Lambda^2M^{(p,p)}$ and $M$ are direct sums of Young modules. By Lemmas \ref{L;complexity} (i), (iv) and \ref{HemmerNakano}, $c_{\sym{2p}}(M)=2$ and $Y^{(p,p)}\mid M$. So $(\F_H\boxtimes Y^{(p,p)}){\uparrow^{\sym{n}}}\mid(\F_H\boxtimes M){\uparrow^{\sym{n}}}\cong(\F^t){\uparrow^{\sym{n}}}$. By Lemmas \ref{L;complexity} (iii), (iv) and \ref{HemmerNakano}, $(\F_H\boxtimes Y^{(p,p)}){\uparrow^{\sym{n}}}$ has complexity $r$. Write $\eta$ for $\nu\cup(p,p)$ and note that
\begin{align*}
M^{\eta}\cong(\F_H\boxtimes M^{(p,p)}){\uparrow^{\sym{n}}}\cong (\F_H\boxtimes Y^{(p,p)}){\uparrow^{\sym{n}}}\oplus \bigoplus_{(p,p)\lhd\gamma}[M^{(p,p)}:Y^\gamma](\F_H\boxtimes Y^{\gamma}){\uparrow^{\sym{n}}}.
\end{align*}
By Lemma \ref{HemmerNakano}, $c_{\sym{n}}(M^\eta)=r$. Let ${\tilde{\eta}}=\nu\cup(2p)$. By the displayed formula, the Krull-Schmidt Theorem and Lemmas \ref{L;complexity} (i), (iii), (iv), \ref{HemmerNakano}, for any $\alpha\vdash n$ satisfying that $Y^\alpha\mid M^\eta$ and $c_{\sym{n}}(Y^\alpha)=r$, we have either $Y^\alpha\mid (\F_H\boxtimes Y^{(p,p)}){\uparrow^{\sym{n}}}$ or $Y^\alpha\mid M^{\tilde{\eta}}$. By Theorem \ref{T;Donkin} and Lemma \ref{HemmerNakano}, $Y^\mu\mid M^\eta$ and $c_{\sym{n}}(Y^\mu)=r$. Let $\mu$ have the $p$-adic expansion $\sum_{i=0}^mp^i\mu(i)$ for some $m\in\mathbb{N}$. For any a $p$-adic decomposition $\tilde{\eta}=\sum_{i=0}^mp^i\sigma[i]$, if $\sigma[i]\models|\mu(i)|$ for all $0\leq i\leq m$, $\mu(0)=\overline{r_\eta}$ and $r_\eta=r_{\tilde{\eta}}=\sigma[0]$. Moreover, $\ell(\mu(1))=\ell(q_\eta)=\ell(q_{\overline{\nu}})+2>\ell(q_{\overline{\nu}})+1=\ell(q_{\tilde{\eta}})\geq \ell(\sigma[1])$. Therefore, $Y^\mu\nmid M^{\tilde{\eta}}$ by Theorem \ref{T;Donkin}, which implies that $Y^\mu\mid (\F_H\boxtimes Y^{(p,p)}){\uparrow^{\sym{n}}}$. By $(5.2)$ and the above discussion, $Y^\mu\mid(\F^t){\uparrow^{\sym{n}}}$ and $(\F^t){\uparrow^{\sym{n}}}\mid\Lambda^2M^\lambda$. Since $p>2$, by Lemmas \ref{L;complexity} (i), \ref{L;compare} and \ref{L;symmetriccomplexity}, $c_{\sym{n}}(Y^\mu)=c_{\sym{n}}(\Lambda^2M^\lambda)=r$. The proof is now complete.
\end{proof}
We now present some lemmas for the case $p=2$. For convenience, we define some partitions of $n$. Let $p=2$,  $\lambda=(\lambda_1,\ldots,\lambda_\ell)\vdash n$ and $\lambda\neq (n)$. Set
\[\lambda_{i,j}=\begin{cases}
(\lambda_1,\ldots,\lambda_i-1,\ldots,\lambda_j-1,\ldots,\lambda_\ell)\cup(2), & \text{if}\ \ell(r_\lambda)>0\ \text{and}\ 1\leq i<j\leq\ell,\\
(\lambda_1,\ldots,\lambda_i-2,\ldots,\lambda_j-2,\ldots,\lambda_\ell)\cup(4), & \text{if}\ \ell(r_\lambda)=0\ \text{and}\ 1\leq i<j\leq\ell.\\
\end{cases}
\]
\begin{lem}\label{L;n>=2}
Let $p=2$ and $\lambda=(\lambda_1,\ldots,\lambda_\ell)\vdash n$. If $2\nmid \lambda_i$ and $2\nmid\lambda_j$ for some $1\leq i<j\leq \ell$, then $M^{\lambda_{i,j}}\mid \Lambda^2M^\lambda$ and
$c_{\sym{n}}(M^{\lambda_{i,j}})=c_{\sym{n}}(\Lambda^2M^\lambda)=\mathrm{rank}(\sym{\lambda})+1$.
\end{lem}
\begin{proof}
Let $r=\text{rank}(\sym{\lambda})$, $a\in R_i(\{t^\lambda\})$ and $b\in R_j(\{t^\lambda\})$. Let $t=(a,b)t^\lambda$ and note that $\{t\}\neq \{t^\lambda\}$. By Notation \ref{N;notation} (v), also notice that $u=\{t^\lambda\}\wedge\{t\}\in\mathcal{T}(\lambda)_2^e$. As $i\neq j$, $2\nmid\lambda_i$ and $2\nmid\lambda_j$, by Lemma \ref{L;preparation} and $(6.1)$, $L(u)=(R(\{t^\lambda\})\cap R(\{t\}))\times \langle (a,b)\rangle$, where $L(u)$ has $2$-rank $r+1$ and $L(u)$ is $\sym{n}$-conjugate to $\sym{\lambda_{i,j}}$. As $p=2$ and $(5.2)$ holds, $M^{\lambda_{i,j}}\cong(\F_{L(u)}){\uparrow^{\sym{n}}}$ and $(\F_{L(u)}){\uparrow^{\sym{n}}}\mid \Lambda^2M^\lambda$. Therefore, $M^{\lambda_{i,j}}\mid \Lambda^2M^\lambda$. By Lemma \ref{L;complexity} (iv), we have $c_{\sym{n}}(M^{\lambda_{i,j}})=r+1$. This implies that $c_{\sym{n}}(\Lambda^2M^\lambda)\geq r+1$ by Lemma \ref{L;complexity} (i). As $p=2$ and $r+1=\mathrm{rank}(L(u))\leq \mathrm{rank}(\sym{n})\leq \frac{n}{2}$, by Lemmas \ref{L;compare} and \ref{L;symmetriccomplexity}, notice that $c_{\sym{n}}(\Lambda^2M^\lambda)\leq r+1$. We have
$c_{\sym{n}}(M^{\lambda_{i,j}})=c_{\sym{n}}(\Lambda^2M^\lambda)=r+1$. The lemma follows.
\end{proof}
\begin{lem}\label{L;n=1}
Let $p=2$ and $\lambda=(\lambda_1,\ldots,\lambda_\ell)\vdash n$. If $\ell(r_\lambda)=1$ and $2\nmid \lambda_i+\lambda_j$ for some $1\leq i<j\leq\ell$, then $M^{\lambda_{i,j}}\mid \Lambda^2M^\lambda$ and
$c_{\sym{n}}(M^{\lambda_{i,j}})=c_{\sym{n}}(\Lambda^2M^\lambda)=\mathrm{rank}(\sym{\lambda})$.
\end{lem}
\begin{proof}
Let $r=\text{rank}(\sym{\lambda})$, $a\in R_i(\{t^\lambda\})$ and $b\in R_j(\{t^\lambda\})$. Let $t=(a,b)t^\lambda$ and note that $\{t\}\neq \{t^\lambda\}$. By Notation \ref{N;notation} (v), also notice that $u=\{t^\lambda\}\wedge\{t\}\in\mathcal{T}(\lambda)_2^e$. As $i\neq j$ and $2\nmid\lambda_i+\lambda_j$, by Lemma \ref{L;preparation} and $(6.1)$, $L(u)=(R(\{t^\lambda\})\cap R(\{t\}))\times \langle (a,b)\rangle$, where $L(u)$ has $2$-rank $r$ and $L(u)$ is $\sym{n}$-conjugate to $\sym{\lambda_{i,j}}$. As $p=2$ and $(5.2)$ holds, $M^{\lambda_{i,j}}\cong(\F_{L(u)}){\uparrow^{\sym{n}}}$ and
$(\F_{L(u)}){\uparrow^{\sym{n}}}\mid \Lambda^2M^\lambda$. Therefore, $M^{\lambda_{i,j}}\mid \Lambda^2M^\lambda$. By Lemma \ref{L;complexity} (iv), note that $c_{\sym{n}}(M^{\lambda_{i,j}})=r$, which implies that $c_{\sym{n}}(\Lambda^2M^\lambda)\geq r$ by Lemma \ref{L;complexity} (i). As $p=2$ and $\ell(r_\lambda)=1$, we have $r=\lfloor \frac{n}{2}\rfloor$. Therefore, by Lemma \ref{L;complexity} (v), note that $c_{\sym{n}}(\Lambda^2M^\lambda)\leq \mathrm{rank}(\sym{n})=r$. So
we obtain that $c_{\sym{n}}(M^{\lambda_{i,j}})=c_{\sym{n}}(\Lambda^2M^\lambda)=r$. The lemma follows.
\end{proof}
To continue the discussion, given $\lambda\vdash n$, recall that $Y^{s(\lambda)}\mid M^\lambda$ and $Y^{s(\lambda)}$ is a Scott module. Also recall that $s(\lambda)$ can be explicitly determined by $\lambda$.
\begin{lem}\label{L;n=0}
Let $p=2$ and $\lambda\vdash n$. If $\lambda\neq (n)$ and $\ell(r_\lambda)=0$, then $Y^{s(\lambda_{i,j})}\mid \Lambda^2M^\lambda$ and
$c_{\sym{n}}(Y^{s(\lambda_{i,j})})=c_{\sym{n}}(\Lambda^2M^\lambda)=\mathrm{rank}(\sym{\lambda})$ for all $1\leq i<j\leq\ell(\lambda)$.
\end{lem}
\begin{proof}
Let $\lambda=(\lambda_1,\ldots,\lambda_\ell)$ and $r=\text{rank}(\sym{\lambda})$. For any given $1\leq i<j\leq \ell$, since $\ell(r_\lambda)=0$, note that $2\mid \lambda_i$ and $2\mid \lambda_j$. Let $\{a,c\}\subseteq R_i(\{t^\lambda\})$ and $\{b,d\}\subseteq R_j(\{t^\lambda\})$. Let $e=(a,b)(c,d)\in\sym{n}$. Let $t=et^\lambda$ and note that $\{t\}\neq\{t^\lambda\}$. By Notation \ref{N;notation} (v), also notice that $u=\{t^\lambda\}\wedge\{t\}\in\mathcal{T}(\lambda)_2^e$. Moreover, we have $e\{t^\lambda\}=\{t\}$ and $e\{t\}=\{t^\lambda\}$. Let $S_i=R_i(\{t^\lambda\})\setminus\{a,c\}$ and  $S_j=R_j(\{t^\lambda\})\setminus\{b,d\}$. According to Lemma \ref{L;subgroup4}, observe that $L(u)=(R(\{t^\lambda\})\cap R(\{t\}))\rtimes\langle e\rangle$, where
\begin{align}
R(\{t^\lambda\})\cap R(\{t\})=\sym{S_i}\times\sym{S_j}\times (\prod_{\substack{k=1\\
                  k\neq i,j\\
                  } }^{\ell}\sym{R_k(\{t^\lambda\})})\times \sym{\{a,c\}}\times \sym{\{b,d\}}\leq\sym{\lambda}.
\end{align}
Note that a Sylow $2$-subgroup $P$ of $L(u)$ is $\sym{n}$-conjugate to a Sylow $2$-subgroup of $\sym{\lambda_{i,j}}$. Therefore, $Sc_{\sym{n}}(P)\cong Y^{s(\lambda_{i,j})}$. Since $p=2$ and $(5.2)$ holds, we have $Y^{s(\lambda_{i,j})}\mid(\F_{L(u)}){\uparrow^{\sym{n}}}$ and $(\F_{L(u)}){\uparrow^{\sym{n}}}\mid \Lambda^2M^\lambda$. So $Y^{s(\lambda_{i,j})}\mid \Lambda^2M^\lambda$. As $2\mid \lambda_i$ and $2\mid \lambda_j$, by $(6.2)$, note that $\mathrm{rank}(P)=r$. By Lemma \ref{L;complexity} (ii), we deduce that $c_{\sym{n}}(Y^{s(\lambda_{i,j})})=r$. Note that $c_{\sym{n}}(\Lambda^2M^\lambda)\geq r$ by Lemma \ref{L;complexity} (i). As $\ell(r_\lambda)=0$,  $r=\mathrm{rank}(\sym{n})=\frac{n}{2}$. Therefore, by Lemma \ref{L;complexity} (v), notice that $c_{\sym{n}}(\Lambda^2M^\lambda)\leq \mathrm{rank}(\sym{n})=r$. We have $c_{\sym{n}}(Y^{s(\lambda_{i,j})})=c_{\sym{n}}(\Lambda^2M^\lambda)=r$. The lemma follows.
\end{proof}

\begin{cor}\label{C;exteriorsquare}
Let $\lambda=(\lambda_1,\ldots,\lambda_\ell)\vdash n$, $\lambda\neq (n)$ and $r=\mathrm{rank}(\sym{\lambda})$.
\begin{enumerate}
\item [\em (i)] If $p>2$, then \[c_{\sym{n}}(\Lambda^2M^\lambda)=\begin{cases}
r-1, &\text{if}\ \ell=2,\ p\mid \lambda_1\ \text{and}\ \lambda_2<p,\\
r, &\text{otherwise}.\end{cases}\]
\item [\em (ii)] If $p=2$, then  \[c_{\sym{n}}(\Lambda^2M^\lambda)=\begin{cases}
r+1, &\text{if}\ \ell(r_\lambda)\geq 2,\\
r, &\text{otherwise}.\end{cases}\]
\end{enumerate}
\end{cor}
\begin{proof}
For (i), if $\ell\geq 3$ and $\lambda_2<p$, by Lemma \ref{L;np=2}, note that $c_{\sym{n}}(\Lambda^2M^\lambda)=r$. If $\lambda_2\geq p$, by Lemma \ref{L;2pparts}, also notice that $c_{\sym{n}}(\Lambda^2M^\lambda)=r$. For the remaining case, as $\lambda\neq (n)$, we have $\ell=2$ and $\lambda_2<p$. This case is solved by Lemmas \ref{L;np=2} and \ref{L;two-part}. The proof of (i) is now complete.

For (ii), if $\ell(r_\lambda)\geq 2$, for some $1\leq i<j\leq \ell$, we have $2\nmid \lambda_i$ and $2\nmid \lambda_j$. By Lemma \ref{L;n>=2}, we get the desired complexity for $\Lambda^2M^\lambda$. For the case $\ell(r_\lambda)<2$, according to Lemmas \ref{L;n=1} and \ref{L;n=0}, note that $c_{\sym{n}}(\Lambda^2M^\lambda)=r$. This completes the proof.
\end{proof}
\begin{prop}\label{P;exteriorYoung}
Let $\lambda\vdash n$ and $\lambda\neq (n)$. Then there exists some $\mu\vdash n$ such that $Y^\mu\mid \Lambda^2M^\lambda$ and $c_{\sym{n}}(Y^\mu)=c_{\sym{n}}(\Lambda^2M^\lambda)$. Moreover, $\mu$ is explicitly determined by $\lambda$.
\end{prop}
\begin{proof}
Let $\lambda=(\lambda_1,\ldots,\lambda_\ell)$ and $R=\{\nu\vdash n: c_{\sym{n}}(Y^\nu)=c_{\sym{n}}(\Lambda^2M^\lambda)\}$. By Corollary \ref{C;exteriorsquare} and Lemma \ref{HemmerNakano}, note that $R$ can be explicitly determined. We have two cases.
\begin{enumerate}[\text{Case} 1:]
\item $p>2$.
\end{enumerate}
Let $S=\{(i,j): 1\leq i<j\leq \ell,\ \lambda_i<p\ \text{or}\ \lambda_j\geq p\}$ and assume that $(u,v)\in S$. If $\lambda_u<p$, let $\alpha=(\lambda_1,\ldots,\lambda_u-1,\ldots,\lambda_v-1,\ldots,\lambda_\ell)\cup(1^2)$ and $\mu=\overline{r_\alpha}+\overline{q_\alpha}$. If $\lambda_v\geq p$, let $\beta=(\lambda_1,\ldots,\lambda_u-p,\ldots,\lambda_v-p,\ldots,\lambda_\ell)\cup(p,p)$ and $\mu=\overline{r_\beta}+\overline{q_\beta}$. For the two possibilities, according to Lemmas \ref{L;np=2} and \ref{L;2pparts}, note that $Y^\mu\mid \Lambda^2M^\lambda$ and
$c_{\sym{n}}(Y^\mu)=c_{\sym{n}}(\Lambda^2M^\lambda)$. We thus assume that $S=\varnothing$. As $\lambda\neq (n)$, note that we have $\ell=2$, $\lambda_1\geq p$ and $\lambda_2<p$. In this case, if $p\nmid \lambda_1$, let $\gamma=(\lambda_1-1,\lambda_2-1)\cup(1^2)$ and $\mu=\overline{r_\gamma}+\overline{q_\gamma}$. If $p\mid \lambda_1$, set $\mu=\gamma$. For the two possibilities, by Lemmas \ref{L;np=2} and \ref{L;two-part}, we also have $Y^\mu\mid \Lambda^2M^\lambda$ and $c_{\sym{n}}(Y^\mu)=c_{\sym{n}}(\Lambda^2M^\lambda)$.
\begin{enumerate}[\text{Case} 2:]
\item $p=2$.
\end{enumerate}
If $\ell(r_\lambda)>0$, as $\lambda\neq (n)$, let $T=\{\nu: 1\leq i<j\leq \ell,\ Y^\nu\mid M^{\lambda_{i,j}},\ 2\nmid \lambda_i\lambda_j\ \text{or}\ 2\nmid \lambda_i+\lambda_j\}$ and note that $T\neq \varnothing$. Moreover, $T$ can be explicitly determined by Theorem \ref{T;Donkin}. Let $\mu\in R\cap T$. If $\ell(r_\lambda)=0$, as $\lambda\neq (n)$, for any $1\leq i<j\leq \ell$, let $\mu=s(\lambda_{i,j})$. The case $p=2$ is solved by Lemmas \ref{L;n>=2}, \ref{L;n=1} and \ref{L;n=0}. The lemma follows.
\end{proof}
We illustrate Proposition \ref{P;exteriorYoung} with an example. Let $\lambda=(3^2,2)\vdash 8$. By the proof of Proposition \ref{P;exteriorYoung}, the detected $\mu$ of Proposition \ref{P;exteriorYoung} has the following possibilities.
\[\begin{cases} (2^4),\ (4,2^2), & \text{if}\ p=2,\\
(5,3), & \text{if}\ p=3,\\
(2^3,1^2)\ ,(3,2,1^3), & \text{if}\ p>3.\\
\end{cases}\]
Theorem \ref{T;B} is now proved by combining Propositions \ref{P;symmetriccase} and \ref{P;exteriorYoung}.
\section{Proof of Theorem \ref{T;C}}
In this section, we provide the proof of Theorem \ref{T;C}. For our purpose, recall the definitions in Notation \ref{N;notation}. Given $\lambda\models n$, also recall that $\overline{\lambda}$ is the partition whose parts are exactly the non-zero parts of $\lambda$.

Let $\lambda\vdash n$ and $\overline{r_\lambda}=(p-1,1)$. Therefore, we have $p\mid n$ and $d_{\lambda,p}=1$. Moreover, $\mathfrak{S}_{\lambda,p}=H_{\lambda,p}\times\sym{\{n-p+1,\ldots,n\}}$. Recall that $\mathcal{O}_{\lambda,p}=\{g\mathfrak{s}_\lambda: g\in P_{\lambda,p}\}$. By Lemma \ref{L;defi} (ii), label $\O_{\lambda,p}=\{\{s_1\},\ldots, \{s_{p}\}\}$ and recall that $t_{\lambda,p}=\{s_1\}\odot\cdots\odot\{s_{p}\}\in\mathcal{T}(\lambda)_{p}^s$. As $H_{\lambda,p}\leq K(t_{\lambda,p})$ by Lemma \ref{L;subgroup1} and $\overline{r_\lambda}=(p-1,1)$, note that $\mathfrak{S}_{\lambda,p}\leq K(t_{\lambda,p})$.
\begin{lem}\label{L;stabilizer}
Let $\lambda\vdash n$ and $\overline{r_\lambda}=(p-1,1)$. Then $K(t_{\lambda,p})=\mathfrak{S}_{\lambda,p}$.
\end{lem}
\begin{proof}
It suffices to show that $K(t_{\lambda,p})\leq\mathfrak{S}_{\lambda,p}$. The inequality is from Lemma \ref{L;subgroup2}.
\end{proof}

Let $\lambda\vdash n$ and $1<a\in \mathbb{N}$. For any $t=\{u_1\}\odot\cdots\odot\{u_a\}\in \mathcal{T}(\lambda)_a^s$, define $n_{t,i}=|\bigcap_{j=1}^aR_i(\{u_j\})|$ for all $1\leq i\leq \ell(\lambda)$ and put $\phi(t)=(n_{t,1},\ldots,n_{t,\ell(\lambda)})$. For instance, let $p=3$ and $\lambda=(5,4)\vdash 9$. As $r_\lambda=(2,1)$, $\phi(t_{\lambda,p})=q_\lambda=(3^2)$. Let $t_1$, $t_2\in \mathcal{T}(\lambda)_a^s$. If there exists some $g\in\sym{n}$ such that $gt_1=t_2$, note that $\phi(t_1)=\phi(t_2)$. We shall need the following result.
\begin{lem}\label{L;orbit}
Let $\lambda\vdash n$ and $\overline{r_\lambda}=(p-1,1)$. Then
$\mathcal{O}^s(t_{\lambda,p})=\{t\in\mathcal{T}(\lambda)_p^s: \phi(t)=q_\lambda\}$.
\end{lem}
\begin{proof}
Let $S=\{t\in\mathcal{T}(\lambda)_p^s: \phi(t)=q_\lambda\}$. By the definition of $t_{\lambda,p}$, observe that $\phi(t_{\lambda,p})=q_\lambda$. So we have $\mathcal{O}^s(t_{\lambda,p})\subseteq S$. It suffices to show that $|\O^s(t_{\lambda,p})|=|S|$. Let $\lambda=(\lambda_1,\ldots,\lambda_\ell)$, where $\lambda_u\equiv p-1\pmod p$ and $\lambda_v\equiv 1\pmod p$ for some $1\leq u\neq v\leq \ell$. As $\overline{r_\lambda}=(p-1,1)$, by the definition of $S$, we have
\begin{align}
|S|=\frac{n!}{p!(\lambda_u-p+1)!(\lambda_v-1)!\prod_{i\neq u,v}(\lambda_i!)}.
\end{align}
Recall that $\mathcal{O}^s(t_{\lambda,p})$ is an $\F$-basis of $V(t_{\lambda,p})$ and $V(t_{\lambda,p})\cong(\F_{K(t_{\lambda,p})}){\uparrow^{\sym{n}}}$. By Lemma \ref{L;stabilizer} and the definition of $\mathfrak{S}_{\lambda,p}$, note that $|\O^s(t_{\lambda,p})||K(t_{\lambda,p})|=|\mathcal{O}^s(t_{\lambda,p})||\mathfrak{S}_{\lambda,p}|=n!$ and $\mathfrak{S}_{\lambda,p}$ is $\sym{n}$-conjugate to $\sym{\eta}$, where $\eta=q_\lambda\cup (p)$. By $(7.1)$, $|\O^s(t_{\lambda,p})|=|S|$.
\end{proof}
\begin{nota}\label{N;subgroups}
We shall introduce some elementary abelian $p$-subgroups of $\sym{n}$.
\begin{enumerate}[(i)]
\item Let $m\in \mathbb{N}_0$ and $0\leq pm\leq n$. If $m>0$, let $E_m$ be the elementary abelian $p$-subgroup $\langle \bigcup_{i=1}^m\{s_i\}\rangle$ of $\sym{n}$, where $s_i=((i-1)p+1, \ldots, ip).$ Set $E_0=1$ and note that $\text{rank}(E_m)=m$. Let $s\in \mathbb{N}_0$ and $0\leq 4s\leq n$. If $s>0$, let $K_s$ be the elementary abelian $2$-subgroup $\langle \bigcup_{i=1}^s\{k_{i,1},k_{i,2}\}\rangle$ of $\sym{n}$, where $k_{i,1}=(4i-3,4i-2)(4i-1,4i)$ and $k_{i,2}=(4i-3,4i-1)(4i-2,4i).$ Set $K_0=1$ and note that $\text{rank}(K_s)=2s$.
\item Let $2\mid n$ and $x$, $y_x\in \mathbb{N}_0$, where $0\leq 4x\leq n$ and $y_x=\frac{n-4x}{2}$. If $x<\frac{n}{4}$, let $H_x$ be the elementary abelian $2$-subgroup $\langle\bigcup_{i=1}^{y_x}\{s_{x,i}\}\rangle$ of $\sym{n}$, where we define $s_{x,i}=(4x+2i-1,4x+2i)$. Set $H_{n/4}=1$ if $4\mid n$. Also note that $\text{rank}(H_x)=y_x$. If $p=2$, observe that $H_0$=$E_{n/2}$.
\end{enumerate}
\end{nota}
We need the following lemmas as preparation.
\begin{lem}\label{L;elementaryabelian}
Let $p\mid n$ and $E$ be an elementary abelian $p$-subgroup of $\sym{n}$ with $p$-rank $\frac{n}{p}$.
\begin{enumerate}
\item [\em (i)] If $p>2$, then $E$ is $\sym{n}$-conjugate to $E_{n/p}$.
\item [\em (ii)] If $p=2$, then $E$ is $\sym{n}$-conjugate to some $K_\ell\times H_{\ell}$, where $0\leq\ell\leq\lfloor\frac{n}{4}\rfloor$.
\end{enumerate}
\end{lem}
\begin{proof}
(i) is well-known and (ii) is proved in \cite[Lemma 3.7]{Jiang}.
\end{proof}
\begin{lem}\label{L;lemma}
Let $\lambda\vdash n$, $\overline{r_\lambda}=(p-1,1)$ and $b=\frac{n}{p}$. Let $t=\{u_1\}\odot\cdots\odot\{u_p\}\in\mathcal{T}(\lambda)_p^s$, where $\{u_i\}\neq\{u_j\}$ for all $1\leq i\neq j\leq p$. Let $S=\{\{u_1\},\ldots,\{u_p\}\}$ and assume that $c_{\sym{n}}(V(t))=c_{\sym{n}}(S^pM^\lambda)$.
\begin{enumerate}
\item [\em (i)] There exists an elementary abelian $p$-subgroup $E$ of $K(t)$ with $p$-rank $b$.
\item [\em (ii)] If $E_b\leq K(t)$, then there exists a unique $s_k\in\{s_1,\ldots,s_b\}$ such that $\langle s_k\rangle$ acts on $S$ transitively. Moreover, $s_i\{u_j\}=\{u_j\}$ for all $1\leq i\neq k\leq b$ and $1\leq j\leq p$.
\item [\em (iii)] If $p=2$ and $K_\ell\times H_\ell\leq K(t)$ for some $0\leq \ell\leq \lfloor\frac{n}{4}\rfloor$, then $\ell\neq \frac{n}{4}$ and there exists a unique $s_{\ell,q}\in \{s_{\ell,1},\ldots,s_{\ell,y_\ell}\}$ such that $\langle s_{\ell,q}\rangle$ acts on $S$ transitively. Moreover, let $U=\bigcup_{i=1}^\ell\{k_{i,1},k_{i,2}\}\cup\{s_{\ell,1},\ldots,s_{\ell,y_\ell}\}$. Then $g\{u_j\}=\{u_j\}$ for all $g\in U\setminus\{s_{\ell,q}\}$ and $1\leq j\leq 2$.
\end{enumerate}
\end{lem}
\begin{proof}
As $\overline{r_\lambda}=(p-1,1)$, note that $\text{rank}(\sym{\lambda})=b-1$. By Lemma \ref{L;symmetriccomplexity}, notice that  $c_{\sym{n}}(V(t))=c_{\sym{n}}(S^pM^\lambda)=b$. As $V(t)\cong(\F_{K(t)}){\uparrow^{\sym{n}}}$, by Lemma \ref{L;complexity} (iv), there exists an elementary abelian $p$-subgroup $E$ of $K(t)$ with $p$-rank $b$. (i) is shown. By the definition of $K(t)$, observe that $K(t)$ acts on $S$. For any element $g$ of $K(t)$ with $p$-power order, if $\langle g\rangle$ does not act transitively on $S$, $g$ fixes $\{u_i\}$ for all $1\leq i\leq p$.

For (ii), as $E_b\leq K(t)$, note that there exists some $s_k\in \{s_1,\ldots,s_b\}\subseteq E_b$ such that $\langle s_k\rangle$ acts transitively on $S$. Otherwise, we have $g\{u_j\}=\{u_j\}$ for all $g\in E_b$ and $1\leq j\leq p$. So $E_b$ is $\sym{n}$-conjugate to a subgroup of $\sym{\lambda}$. This fact implies that $b=\text{rank}(E_b)\leq \text{rank}(\sym{\lambda})=b-1$, which is absurd. Also note that there does not exist $s_\ell\in\{s_1,\ldots,s_b\}$ such that $\ell\neq k$ and $\langle s_\ell\rangle$ acts transitively on $S$. Otherwise, there exists some $m\in \mathbb{N}$ such that $1\leq m<p$ and $s_\ell^ms_k\{u_i\}=\{u_i\}$ for all $1\leq i\leq p$. This implies that the numbers $(k-1)p+1,\ldots,kp$ are totally contained in some row of $\{u_i\}$ for all $1\leq i\leq p$. Therefore, by Lemma \ref{PModules}, $s_k$ fixes $\{u_i\}$ for all $1\leq i\leq p$, which is a contradiction. (ii) thus follows.

For (iii), as $K_\ell\times H_\ell\leq K(t)$, there exists some $g\in U$ such that $\langle g\rangle$ acts on $S$ transitively. Otherwise, $z\{u_j\}=\{u_j\}$ for all $z\in K_\ell\times H_\ell$ and $1\leq j\leq 2$. So $K_\ell\times H_\ell$ is $\sym{n}$-conjugate to a subgroup of $\sym{\lambda}$. We have $b=\mathrm{rank}(K_\ell\times H_\ell)\leq \mathrm{rank}(\sym{\lambda})=b-1$, which is absurd. For such an element $g$, we claim that $g\in \{s_{\ell,1},\ldots,s_{\ell,y_\ell}\}$. If $g=k_{i,1}$ or $k_{i,2}$ for some $1\leq i\leq\ell$, notice that $S=\{\{u_1\},\{u_2\}\}$ and $\langle k_{i,1}, k_{i,2}\rangle$ acts on $S$. Therefore, there exists some $e=(a,b)(c,d)\in\langle k_{i,1},k_{i,2}\rangle$ such that $e\{u_1\}=\{u_1\}$. By Lemma \ref{PModules}, this implies that $\{a,b\}\subseteq R_x(\{u_1\})$, $\{c,d\}\subseteq  R_y(\{u_1\})$ and $x\neq y$. Otherwise, if $a,b,c,d$ lie in a single row of $\{u_1\}$, then $g\{u_1\}=\{u_1\}$ by Lemma \ref{PModules}. This is a contradiction. As $g\{u_1\}=\{u_2\}$ and $e\in\langle k_{i,1},k_{i,2}\rangle$, we get that $\{a,b\}\subseteq R_y(\{u_2\})$ and $\{c,d\}\subseteq  R_x(\{u_2\})$. Therefore, for all $h\in U\setminus\{k_{i,1},k_{i,2}\}$, as $\{a,b\}\subseteq R_x(h\{u_1\})$ and $K(t)$ acts on $S$, $h\{u_1\}=\{u_1\}$. By Lemma \ref{PModules}, this forces that all the parts of $\lambda$ are even, which contradicts with the fact that $\overline{r_\lambda}=(1^2)$. The claim is shown. By this claim, $\ell\neq \frac{n}{4}$ and there exists some $s_{\ell,q}\in U$ such that $\langle s_{\ell,q}\rangle$ acts on $S$ transitively. One can write a proof as the one given in (ii) to show that $g\{u_j\}=\{u_j\}$ for all $g\in U\setminus\{s_{\ell,q}\}$ and $1\leq j\leq 2$. The proof is now complete.
\end{proof}
\begin{lem}\label{L;lemma2}
Let $p=2$, $\lambda\vdash n$, $\overline{r_\lambda}=(1^2)$ and
$t=\{u_1\}\wedge\{u_2\}\in\mathcal{T}(\lambda)_2^e$. Assume that $c_{\sym{n}}(W(t))=c_{\sym{n}}(\Lambda^2 M^\lambda)$.
\begin{enumerate}
\item [\em (i)] There exists an elementary abelian $2$-subgroup $E$ of $L(t)$ with $2$-rank $\frac{n}{2}$.
\item [\em (ii)] If $K_\ell\times H_\ell\leq L(t)$ for some $0\leq \ell\leq \lfloor\frac{n}{4}\rfloor$, then $\ell\neq\frac{n}{4}$ and there exists a unique $s_{\ell,q}\in \{s_{\ell,1},\ldots,s_{\ell,y_\ell}\}$ such that $\langle s_{\ell,q}\rangle$ acts on $\{\{u_1\},\{u_2\}\}$ transitively. Moreover, let $U=\bigcup_{i=1}^\ell\{k_{i,1},k_{i,2}\}\cup\{s_{\ell,1},\ldots,s_{\ell,y_\ell}\}$. Then $g\{u_j\}=\{u_j\}$ for all $g\in U\setminus\{s_{\ell,q}\}$ and $1\leq j\leq 2$.
\end{enumerate}
\end{lem}
\begin{proof}
As $p=2$ and $\overline{r_\lambda}=(1^2)$, note that $W(t)\cong(\F_{L(t)}){\uparrow^{\sym{n}}}$ and $\mathrm{rank}(\sym{\lambda})=\frac{n}{2}-1$. By Corollary \ref{C;exteriorsquare} (ii), $c_{\sym{n}}(W(t))=c_{\sym{n}}(\Lambda^2M^\lambda)=\frac{n}{2}$. Therefore, (i) and (ii) are shown by mimicking the proofs given in (i) and (iii) of Lemma \ref{L;lemma} respectively.
\end{proof}

\begin{lem}\label{L;disjoint}
Let $\lambda\vdash n$ and $\overline{r_\lambda}=(p-1,1)$. Let $t=\{u_1\}\odot\cdots\odot\{u_p\}\in\mathcal{T}(\lambda)_p^s$ and $c_{\sym{n}}(V(t))=c_{\sym{n}}(S^pM^\lambda)$. Then $\{u_i\}\neq\{u_j\}$ for all $1\leq i\neq j\leq p$.
\end{lem}
\begin{proof}
Let $r=\mathrm{rank}(\sym{\lambda})$. Note that $K(t)$ acts on the set of all the mutually distinct $\lambda$-tabloids among $\{u_1\},\ldots,\{u_p\}$. If the desired assertion is false, a Sylow $p$-subgroup $P$ of $K(t)$ fixes $\{u_i\}$ for all $1\leq i\leq p$. So $P$ is $\sym{n}$-conjugate to a subgroup of $\sym{\lambda}$. As $\overline{r_\lambda}=(p-1,1)$ and $V(t)\cong(\F_{K(t)}){\uparrow^{\sym{n}}}$, by Lemmas \ref{L;complexity} (iv) and \ref{L;symmetriccomplexity}, $c_{\sym{n}}(V(t))\leq r$ and $c_{\sym{n}}(S^pM^\lambda)=r+1$, which contradicts with the fact $c_{\sym{n}}(V(t))=c_{\sym{n}}(S^pM^\lambda)$.
\end{proof}

\begin{lem}\label{L;p>2}
Let $p>2$, $\lambda\vdash n$ and $\overline{r_\lambda}=(p-1,1)$. If $t=\{u_1\}\odot\cdots\odot\{u_p\}\in\mathcal{T}(\lambda)_p^s$ and $c_{\sym{n}}(V(t))=c_{\sym{n}}(S^pM^\lambda)$, then $\phi(t)=q_\lambda$.
\end{lem}
\begin{proof}
Let $\lambda=(\lambda_1,\ldots,\lambda_\ell)$ and $b=\frac{n}{p}$. By Lemma \ref{L;disjoint}, notice that $\{u_1\},\ldots,\{u_p\}$ are mutually distinct. Let $S=\{\{u_1\},\ldots,\{u_p\}\}$. By Lemma \ref{L;lemma} (i), there exists an elementary abelian $p$-subgroup $E$ of $K(t)$ with $p$-rank $b$. As $p>2$, by Lemma \ref{L;elementaryabelian} (i), $E$ is $\sym{n}$-conjugate to $E_b$. As $\phi(t)=\phi(gt)$ for all $g\in \sym{n}$, we may assume further that $E=E_b$.  By Lemmas \ref{L;lemma} (ii) and \ref{PModules}, there exists some $s_k$ such that $1\leq k\leq b$ and $\langle s_k\rangle$ acts transitively on $S$. Moreover, for all $1\leq i\neq k\leq b$, $s_i$ fixes $\{u_1\}$ and the numbers $(i-1)p+1,\ldots,ip$ are totally contained in some row of $\{u_1\}$. As $\overline{r_\lambda}=(p-1,1)$, also note that $p-1$ members of $\{(k-1)p+1,\ldots, kp\}$ are in $R_x(\{u_1\})$ and the remaining member of $\{(k-1)p+1,\ldots, kp\}$ is in $R_y(\{u_1\})$, where  $\lambda_x\equiv p-1\pmod p$ and $\lambda_y\equiv 1\pmod p$.
Moreover, for any $a\in\{(k-1)p+1,\ldots, kp\}$, we have $a\notin\bigcap_{i=1}^pR_x(\{u_i\})$ and $a\notin\bigcap_{i=1}^pR_y(\{u_i\})$. As $\langle s_k\rangle$ acts transitively on $S$, by the above discussion and the definition of $\phi(t)$, we have $\phi(t)=q_\lambda$.
\end{proof}
\begin{lem}\label{L;p=2}
Let $p=2$, $\lambda\vdash n$ and $\overline{r_\lambda}=(1^2)$. If $t=\{u_1\}\odot\{u_2\}\in\mathcal{T}(\lambda)_2^s$ and $c_{\sym{n}}(V(t))=c_{\sym{n}}(S^2M^\lambda)$, then $\phi(t)=q_\lambda$.
\end{lem}
\begin{proof}
Let $\lambda=(\lambda_1,\ldots,\lambda_\ell)$, $b=\frac{n}{2}$ and $S=\{\{u_1\},\{u_2\}\}$. By Lemma \ref{L;disjoint}, we have $\{u_1\}\neq\{u_2\}$. By Lemma \ref{L;lemma} (i), there exists an elementary abelian $2$-subgroup $E$ of $K(t)$ with $2$-rank $b$. By Lemma \ref{L;elementaryabelian} (ii), $E$ is $\sym{n}$-conjugate to $K_c\times H_c$ for some $0\leq c\leq\lfloor \frac{b}{2}\rfloor$. As $\phi(t)=\phi(gt)$ for all $g\in\sym{n}$, we may assume that $E=K_c\times H_c$. By Lemmas \ref{L;lemma} (iii) and \ref{PModules}, there exists some $s_{c,q}$ such that $1\leq q\leq y_c$ and $\langle s_{c,q}\rangle$ acts transitively on $S$. Moreover, let $U=\bigcup_{i=1}^c\{k_{i,1},k_{i,2}\}\cup\{s_{c,1},\ldots,s_{c,y_c}\}$. For all $g\in U\setminus\{s_{c,q}\}$, $g$ fixes $\{u_1\}$ and each orbit of $\mathrm{Supp}(g)$ under the action of $\langle g\rangle$ is totally contained in some row of $\{u_1\}$, where $\mathrm{Supp}(g)=\{i\in \mathbf{n}: g(i)\neq i\}$. Since $\overline{r_\lambda}=(1^2)$, note that $R_x(\{u_1\})$ contains a member of $\{4c+2q-1, 4c+2q\}$ and $R_y(\{u_1\})$ contains the other member of $\{4c+2q-1, 4c+2q\}$, where $2\nmid\lambda_x$ and $2\nmid\lambda_y$. For any $a\in\{4c+2q-1,4c+2q\}$, note that $a\not\in (\bigcap_{i=1}^2R_x(\{u_i\}))\cup(\bigcap_{i=1}^2R_y(\{u_i\}))$. As $s_{c,q}\{u_1\}=\{u_2\}$, by the above discussion and the definition of $\phi(t)$, $\phi(t)=q_\lambda$.
\end{proof}
\begin{cor}\label{C;decomposition}
Let $\lambda\vdash n$ and $\overline{r_\lambda}=(p-1,1)$. For some $z\in\mathbb{N}$, there exists a decomposition
$$ S^pM^\lambda=V(t_{\lambda,p})\oplus\bigoplus_{i=1}^zV(t_i),$$ where $t_1,\ldots, t_z\in \mathcal{T}(\lambda)_p^s$, $V(t_{\lambda,p})\cong M^\mu$, $\mu=q_\lambda\cup(p)$ and $c_{\sym{n}}(V(t_i))<c_{\sym{n}}(S^pM^\lambda)$ for all $1\leq i\leq z$.
\end{cor}
\begin{proof}
By $(5.1)$, for some $z\in \mathbb{N}$, recall that $S^pM^\lambda= V(t_{\lambda,p})\oplus\bigoplus_{i=1}^zV(t_i)$, where $t_1,\ldots,t_z\in\mathcal{T}(\lambda)_p^s$.
By Lemma \ref{L;stabilizer} and the definition of $\mathfrak{S}_{\lambda,p}$, note that $V(t_{\lambda,p})\cong M^\mu$, where $\mu=q_\lambda\cup(p)$. For all $1\leq i\leq z$, as $V(t_i)=\langle \mathcal{O}^s(t_i)\rangle_\F$ and $\mathcal{O}^s(t_{\lambda,p})\cap\mathcal{O}^s(t_i)=\varnothing$, by Lemmas \ref{L;complexity} (i), \ref{L;orbit}, \ref{L;p>2} and \ref{L;p=2}, $c_{\sym{n}}(V(t_i))<c_{\sym{n}}(S^pM^\lambda)$ for all $1\leq i\leq z$.
\end{proof}
\begin{cor}\label{C;decomposition2}
Let $p=2$, $\lambda\vdash n$ and $\overline{r_\lambda}=(1^2)$. For some $z\in\mathbb{N}$, there exists a decomposition
$$ \Lambda^2M^\lambda=\bigoplus_{i=1}^zW(t_i),$$ where $t_1,\ldots, t_z\in \mathcal{T}(\lambda)_2^e$, $W(t_1)\cong M^\mu$, $\mu=q_\lambda\cup(2)$ and $c_{\sym{n}}(W(t_i))<c_{\sym{n}}(\Lambda^2M^\lambda)$ for all $1<i\leq z$.
\end{cor}
\begin{proof}
Let $\lambda=(\lambda_1,\ldots,\lambda_\ell)$. By $(5.2)$, for some $z\in \mathbb{N}$, we have $\Lambda^2M^\lambda=\bigoplus_{i=1}^zW(t_i)$, where $t_1,\ldots,t_z\in \mathcal{T}(\lambda)_2^e$. By Lemma \ref{L;complexity} (i), we can assume that $t_1=\{u\}\wedge\{v\}$ and $c_{\sym{n}}(W(t_1))=c_{\sym{n}}(\Lambda^2M^\lambda)$. By Lemmas \ref{L;lemma2} (i), (ii) and \ref{L;elementaryabelian} (ii), there exists some $(a,b)\in L(t_1)$ such that $(a,b)\{u\}=\{v\}$. Moreover, as $\overline{r_\lambda}=(1^2)$, a member of $\{a,b\}$ is in $R_c(\{u\})$ and the other member of $\{a,b\}$ is in $R_d(\{u\})$, where $1\leq c\neq d\leq \ell$, $2\nmid\lambda_c$ and $2\nmid\lambda_d$. As $p=2$, by Lemma \ref{L;subgroup4}, $W(t_1)\cong(\F_{L(t_1)}){\uparrow^{\sym{n}}}\cong M^\mu$, where $\mu=q_\lambda\cup (2)$. If we have $c_{\sym{n}}(W(t_x))=c_{\sym{n}}(\Lambda^2M^\lambda)$ for some $1< x\leq z$, note that $\O^e(t_x)\cap\O^e(t_1)=\varnothing$. Let $t_x=\{\tilde{u}\}\wedge\{\tilde{v}\}$. By Lemmas \ref{L;lemma2} (i), (ii) and \ref{L;elementaryabelian} (ii) again, there exists some $(j,k)\in L(t_x)$ such that $(j,k)\{\tilde{u}\}=\{\tilde{v}\}$. Moreover, as $\overline{r_\lambda}=(1^2)$, a member of $\{j,k\}$ is in $R_c(\{\tilde{u}\})$ and the other member of $\{j,k\}$ is in $R_d(\{\tilde{u}\})$. It is easy to note that there exists some $g\in\sym{n}$ such that $gt_1=t_x$, which contradicts with the fact $\O^e(t_x)\cap\O^e(t_1)=\varnothing$. Therefore, by Lemma \ref{L;complexity} (i), $c_{\sym{n}}(W(t_i))<c_{\sym{n}}(\Lambda^2M^\lambda)$ for all $1<i\leq z$.
The proof is now complete.
\end{proof}
Theorem \ref{T;C} is proved by Lemma \ref{L;complexity} (i), Corollaries \ref{C;decomposition}, \ref{C;decomposition2} and Theorem \ref{T;Donkin}. Theorems \ref{T;B} and \ref{T;C} motivate us to ask the following question.
\begin{ques*}
Let $\lambda\vdash n$ and $1<a\in \mathbb{N}$. For any $\mu\vdash n$ satisfying $Y^\mu\mid S^aM^\lambda$ and $c_{\sym{n}}(Y^\mu)=c_{\sym{n}}(S^aM^\lambda)$, can one determine the partition $\mu$ explicitly ? Similarly, if $\lambda\neq(n)$, for any $\mu\vdash n$ satisfying $Y^\mu\mid \Lambda^2M^\lambda$ and $c_{\sym{n}}(Y^\mu)=c_{\sym{n}}(\Lambda^2M^\lambda)$, can one determine the partition $\mu$ explicitly ?
\end{ques*}
\section{A result of some $p$-subgroups of $\sym{n}$}
The main result of this section is Theorem \ref{T;conjugacy}. This result is important for the proof of Theorem \ref{T;D}. Given $m\in \mathbb{N}$, let $\mathcal{M}_m$ be the set of all the $(m\times m)$-matrices over $\mathbb{N}_0$. We fix the required notation as follows.
\begin{nota}\label{N;notation1}
Let $1<m\in \mathbb{N}$ and $\mathcal{D}_m=\{M\in \mathcal{M}_m: M\ \text{is a diagonal matrix}\}$.
\begin{enumerate}[(i)]
\item Let $S\subseteq \mathbf{n}$ and $P_S$ be a fixed Sylow $p$-subgroup of $\sym{S}$. For any $g\in \sym{n}$, set $\mathrm{Supp}(g)=\{i\in \mathbf{n}:g(i)\neq i\}$. So $g=1$ if and only if $\mathrm{Supp}(g)=\varnothing$. Set $g\varnothing=\varnothing$ and put $gT=\{g(i):i\in T\}$ for any non-empty subset $T$ of $\mathbf{n}$.
\item Let $M=(a_{i,j})\in\mathcal{M}_m$. For all $1\leq i,j\leq m$, write $a_{i,j}=\sum_{k\geq 0}p^ka_{i,j,k}$, where $a_{i,j,k}\in \mathbb{N}_0$ and $a_{i,j,k}<p$ for all $k\geq0$. Define
    \begin{align*}
    &D(M)=(\sum_{1\leq i\leq m}a_{i,i,0},\cdots, \sum_{1\leq i\leq m}a_{i,i,k},\cdots),\\ &U(M)=(\sum_{1\leq i<j\leq m}a_{i,j,0},\cdots, \sum_{1\leq i<j\leq m}a_{i,j,k},\cdots),\\
    &T(M)=(\sum_{1\leq i,j\leq m}a_{i,j,0},\cdots, \sum_{1\leq i,j\leq m}a_{i,j,k},\cdots).
    \end{align*}
Let $M^t$ be the transpose of $M$. If $M=M^t$, then $T(M)=D(M)+2U(M)$, where the computation is the usual computation of sequences. To write the defined sequences down explicitly, we shall end these sequences with their last non-zero entries and use $(0)$ to denote an infinite sequence of zeros.
\item Define $\mathcal{M}_{m,n}=\{(a_{i,j})\in\mathcal{M}_m:
    \sum_{i=1}^m\sum_{j=1}^ma_{i,j}=n\}\setminus \mathcal{D}_m$. For any given $M=(b_{i,j})\in \mathcal{M}_{m,n}$, we assign a subset $s_{i,j,M}$ of $\mathbf{n}$ to $b_{i,j}$ for any $1\leq i, j\leq m$. In this assignment, for all $1\leq i,j,u,v\leq m$, we require that $|s_{i,j,M}|=b_{i,j}$ and $s_{i,j,M}\cap s_{u,v,M}=\varnothing$ if $(i,j)\neq(u,v)$. By this assignment, we define
   \begin{align}
    P_M=\begin{cases} \displaystyle\prod_{1\leq i,j\leq m}P_{s_{i,j,M}}, & \text{if}\ M\neq M^t\ \text{or}\ p>2,\\
    \displaystyle (\!\prod_{1\leq k\leq m}P_{s_{k,k,M}}\times \!\!\!\!\prod_{1\leq i<j\leq m}\!\!\!\!\!(P_{s_{i,j,M}}\!\times\!P_{s_{i,j,M}}^{e_M}))\rtimes\langle e_M\rangle, & \text{otherwise},
    \end{cases}
    \end{align}
 where $e_M$ is a fixed involution of $\sym{n}$ and we have $\mathrm{Supp}(e_M)=\mathbf{n}\setminus \bigcup_{k=1}^ms_{k,k,M}$, $P_{s_{i,j,M}}^{e_M}=e_{M}P_{s_{i,j,M}}e_{M}$ and $e_{M}s_{i,j,M}=s_{j,i,M}$ for all $1\leq i,j\leq m$. For any $1\leq i<j\leq m$, note that $P_{s_{i,j,M}}^{e_M}$ is $\sym{s_{j,i,M}}$-conjugate to $P_{s_{j,i,M}}$. In particular, for any $1\leq i<j\leq m$, $|P_{s_{i,j,M}}^{e_M}|=|P_{s_{j,i,M}}|$. A $p$-subgroup $P$ of $\sym{n}$ is called a Young $p$-subgroup of $\sym{n}$ if $P$ is $\sym{n}$-conjugate to a Sylow $p$-subgroup of some Young subgroup of $\sym{n}$. Let $\mathcal{Y}_n$ be the set of all the Young $p$-subgroups of $\sym{n}$.
\end{enumerate}
\end{nota}
\begin{eg}\label{E;example2}
An example illustrates most of the definitions in Notation \ref{N;notation1}. Let $p=2$ and $M_1$, $M_2\in \mathcal{M}_{3,10}$, where
\begin{equation*}
M_1=\begin{pmatrix}
2 & 1 & 2\\
1 & 2 & 0\\
2 & 0 & 0\\
\end{pmatrix}\ \text{and}\
M_2=\begin{pmatrix}
0 & 2 & 2\\
2 & 0 & 1\\
2 & 1 & 0\\
\end{pmatrix}.
\end{equation*}
Note that $D(M_1)=(0,2)\neq (0)=D(M_2)$, $U(M_1)=(1^2)\neq(1,2)=U(M_2)$ and $T(M_1)=(2,4)=T(M_2)$.
Moreover, by $(8.1)$, observe that $P_{M_1}$ is $\sym{10}$-conjugate to $\langle (1,2), (3,4), (5,6), (7,8), (1,5)(2,6)(9,10)\rangle$. By $(8.1)$ again, we also observe that $P_{M_2}$ is $\sym{10}$-conjugate to $\langle (1,2), (3,4), (5,6), (7,8), (1,3)(2,4)(5,7)(6,8)(9,10)\rangle$.
\end{eg}
Let $\varnothing\neq S\subseteq \mathbf{n}$. We recall some facts of the construction of a Sylow $p$-subgroup of $\sym{S}$. Let $i\in\mathbb{N}_0$ and $P_i$ be a Sylow $p$-subgroup of $\sym{p^i}$. If $|S|=\sum_{j\geq 0}p^jn_j$, where $n_j\in\mathbb{N}_0$ and $n_j<p$ for all $j\geq 0$, then $P_S\cong\prod_{j\geq 0}(P_j)^{n_j}$. If $p=2$ and $i>0$, recall that $P_i=(P_{i-1,1}\times P_{i-1,2})\rtimes\langle w\rangle$, where the groups $P_{i-1,1}$ and $P_{i-1,2}$ are $\sym{2^i}$-conjugate to $P_{i-1}$ and $wP_{i-1,1}w=P_{i-1,2}$ for the involution $w$. Also recall that $\mathbf{n}/P$ is the set of orbits of $\mathbf{n}$ under the action of a $p$-subgroup $P$ of $\sym{n}$. To prove the main result of this section, we need some lemmas as a preparation. The first lemma is well-known.
\begin{lem}\label{L;Youngsubgroups}
Let $P$ be a $p$-subgroup of $\sym{n}$ and $\mathbf{n}/P=\{\mathcal{O}_1,\ldots,\mathcal{O}_s\}$. Then $P\in \mathcal{Y}_n$ if and only if $P$ is a Sylow $p$-subgroup of $\sym{\mathcal{O}_1}\times\cdots\times\sym{\mathcal{O}_s}$.
\end{lem}
\begin{lem}\label{L;p=2Youngsubgroups}
Let $p=2$ and $1<m\in \mathbb{N}$. If $M\in \mathcal{M}_{m,n}$ and $M=M^t$, then $P_M\in \mathcal{Y}_n$ if and only if $U(M)=(0^a,1)$ for some $a\in \mathbb{N}_0$.
\end{lem}
\begin{proof}
Recall the assignment of $M$ in Notation \ref{N;notation1} (iii) and let $D_M=\prod_{i=1}^mP_{s_{i,i,M}}$. As $M=M^t$, notice that $|s_{i,j,M}|=|s_{j,i,M}|$ for any $1\leq i,j\leq m$. If $U(M)=(0^a,1)$ for some $a\in \mathbb{N}_0$, by $(8.1)$ and the definition of $U(M)$, it is obvious to see that $P_M\in \mathcal{Y}_n$. Conversely, for any $1\leq i<j\leq m$, let $S_{i,j,M}=s_{i,j,M}\cup s_{j,i,M}$ and $Q_{i,j,M}$ be a fixed Sylow $2$-subgroup of $\sym{S_{i,j,M}}$. We define $Q_M=D_M\times \prod_{1\leq i<j\leq m}Q_{i,j,M}$ and note that $Q_M\in \mathcal{Y}_n$. Moreover, by $(8.1)$, for all $i\geq 0$, the number of $P_M$-orbits of $\mathbf{n}$ of size $2^i$ is equal to the number of $Q_M$-orbits of $\mathbf{n}$ of size $2^i$. As we have assumed that $P_M\in \mathcal{Y}_n$, by Lemma \ref{L;Youngsubgroups}, note that $P_M$ is $\sym{n}$-conjugate to $Q_M$. In particular, $|P_M|=|Q_M|$. For any $1\leq i<j\leq m$, notice that $|P_{s_{i,j,M}}\times P_{s_{j,i,M}}|\leq |Q_{i,j,M}|$. For all $w\in \mathbb{N}_0$, if we have $1\leq u<v\leq m$, $s_{u,v,M}\neq \varnothing$ and $P_{s_{u,v,M}}$ is not $\sym{n}$-conjugate to a Sylow $2$-subgroup of $\sym{2^w}$, by the construction of a Sylow $2$-subgroup of $\sym{S_{u,v,M}}$, note that $|P_{s_{u,v,M}}\times P_{s_{v,u,M}}|<4|P_{s_{u,v,M}}\times P_{s_{v,u,M}}|\leq |Q_{u,v,M}|$. By $(8.1)$, we thus have
\begin{align*}
|P_M|&=2|D_M||P_{s_{u,v,M}}\times P^{e_M}_{s_{u,v,M}}|\prod_{\substack{1\leq i<j\leq m\\
                  (i,j)\neq (u,v)\\
                  } }|P_{s_{i,j,M}}\times P^{e_M}_{s_{i,j,M}}|\\
&=2|D_M||P_{s_{u,v,M}}\times P_{s_{v,u,M}}|\prod_{\substack{1\leq i<j\leq m\\
                  (i,j)\neq (u,v)\\
                  } }|P_{s_{i,j,M}}\times P_{s_{j,i,M}}|\\
&<4|D_M||P_{s_{u,v,M}}\times P_{s_{v,u,M}}|\prod_{\substack{1\leq i<j\leq m\\
                  (i,j)\neq (u,v)\\
                  } }|P_{s_{i,j,M}}\times P_{s_{j,i,M}}|\leq |D_M|\prod_{1\leq i<j\leq m}|Q_{i,j,M}|,
\end{align*}
which contradicts with the fact that $|P_M|=|Q_M|$. Hence, for any $1\leq i<j\leq m$, $P_{s_{i,j,M}}$ is $\sym{n}$-conjugate to a Sylow $2$-subgroup of $\sym{2^k}$ for some $k\in \mathbb{N}_0$. Therefore, $2|P_{s_{i,j,M}}\times P_{s_{j,i,M}}|=|Q_{i,j,M}|$ for all $1\leq i<j\leq m$ and $s_{i,j,M}\neq \varnothing$. By $(8.1)$, we have
\begin{align}
|P_M|=2|D_M|\prod_{1\leq i<j\leq m}|P_{s_{i,j,M}}\times P_{s_{j,i,M}}|
=|D_M|\prod_{1\leq i<j\leq m}|Q_{i,j,M}|=|Q_M|,
\end{align}
which forces that $|s_{u,v}|=|s_{v,u}|=2^k$ for some $1\leq u<v\leq m$ and $k\in\mathbb{N}_0$. For any $1\leq i<j\leq m$, $(8.2)$ also implies that $s_{i,j,M}\neq \varnothing$ only if $i=u$ and $j=v$. By the definition of $U(M)$, $U(M)=(0^a,1)$ for some $a\in\mathbb{N}_0$. The lemma follows.
\end{proof}
Let $1<m\in \mathbb{N}$ and $M\in \mathcal{M}_{m,n}$. If $p=2$, $M=M^t$ and $U(M)=(0^a,1)$ for some $a\in \mathbb{N}_0$, set $T'(M)=D(M)+(0^{a+1},1,0,0,\dots)$, where the sum is the usual sum of sequences. Otherwise, put $T'(M)=T(M)$. Recall the assignment of $M$ in $(8.1)$ and let $p=2$. For any $1\leq i,j\leq m$, if $s_{i,j,M}\neq \varnothing$, there exist unique $x\in \mathbb{N}$ and $m_1,\ldots,m_x\in \mathbb{N}_0$ such that $0\leq m_1<\cdots<m_x$, $|s_{i,j,M}|=\sum_{u=1}^x2^{m_u}$ and
\begin{equation}
P_{s_{i,j,M}}=Q_{i,j,M,m_1}\times\cdots\times Q_{i,j,M,m_x},
\end{equation}
where, for any $k\in\{m_1,\ldots,m_x\}$, $Q_{i,j,M,k}$ is $\sym{n}$-conjugate to a Sylow $2$-subgroup of $\sym{2^k}$. The permutations are multiplied from right to left. So $(1,2)(2,3)=(1,2,3)$.
\begin{lem}\label{L;T'(M)}
Let $1<m\in \mathbb{N}$ and $M_1$, $M_2\in \mathcal{M}_{m,n}$. If $P_{M_1}$, $P_{M_2}\in \mathcal{Y}_n$, then $P_{M_1}$ is $\sym{n}$-conjugate to $P_{M_2}$ if and only if $T'(M_1)=T'(M_2)$.
\end{lem}
\begin{proof}
Let $M\in \mathcal{M}_{m,n}$ and $P_M\in \mathcal{Y}_n$. By $(8.1)$, Lemma \ref{L;p=2Youngsubgroups} and the definition of $T'(M)$, for any $i>0$, note that the $i$th entry of $T'(M)$ is the number of orbits of $\mathbf{n}$ of size $p^{i-1}$ under the action of $P_M$. If $T'(M_1)=T'(M_2)$, as $P_{M_1}$, $P_{M_2}\in \mathcal{Y}_n$, by Lemma \ref{L;Youngsubgroups}, $P_{M_1}$ is $\sym{n}$-conjugate to $P_{M_2}$. Conversely, if $P_{M_1}$ is $\sym{n}$-conjugate to $P_{M_2}$, then there is a size preserving one-to-one correspondence between the members of $\mathbf{n}/P_{M_1}$ and $\mathbf{n}/P_{M_2}$. By the explanation of the entries of $T'(M)$, $T'(M_1)=T'(M_2)$.
\end{proof}
\begin{lem}\label{L;conjugacy}
Let $p=2$, $1<m\in \mathbb{N}$ and $M_1$, $M_2\in\mathcal{M}_{m,n}$. If $M_1=M_1^t$, $M_2=M_2^t$, $D(M_1)=D(M_2)$ and $U(M_1)=U(M_2)$, then $P_{M_1}$ is $\sym{n}$-conjugate to $P_{M_2}$.
\end{lem}
\begin{proof}
Recall the assignments of $M_1$ and $M_2$ in Notation \ref{N;notation1} (iii). By $(8.1)$, we have
\begin{align*}
&  P_{M_1}=\displaystyle (\prod_{1\leq k\leq m}P_{s_{k,k,M_1}}\times \!\!\!\!\prod_{1\leq i<j\leq m}\!\!\!\!(P_{s_{i,j,M_1}}\times P^{e_{M_1}}_{s_{i,j,M_1}}))\rtimes\langle e_{M_1}\rangle,\\
& P_{M_2}=\displaystyle (\prod_{1\leq k\leq m}P_{s_{k,k,M_2}}\times \!\!\!\!\prod_{1\leq i<j\leq m}\!\!\!\!(P_{s_{i,j,M_2}}\times P^{e_{M_2}}_{s_{i,j,M_2}}))\rtimes\langle e_{M_2}\rangle.
\end{align*}
For any given $s_{a,b,M_1}\neq \varnothing$, if we have $a\leq b$, by $(8.3)$, there exist unique $c\in \mathbb{N}$ and $d_1,\ldots,d_c\in \mathbb{N}_0$ such that $0\leq d_1<\cdots<d_c$ and
\begin{align}
P_{s_{a,b,M_1}}=Q_{a,b,M_1,d_1}\times\cdots\times Q_{a,b,M_1,d_c}.
\end{align}
If $a<b$, as $U(M_1)=U(M_2)$, by the definitions of $U(M_1)$ and $U(M_2)$, there are $Q_{a_1,b_1,M_2,d_1},\ldots, Q_{a_c,b_c,M_2,d_c}$, where $1\leq a_i<b_i\leq m$ for all $1\leq i\leq c$. Moreover, we also have $e_{M_2}Q_{a_1,b_1,M_2,d_1}e_{M_2},\ldots, e_{M_2}Q_{a_c,b_c,M_2,d_c}e_{M_2}$. According to $(8.4)$ and the Sylow Theorem, it is clear that there exists some $g_{a,b}\in \sym{n}$ such that $$g_{a,b}Q_{a,b,M_1,d_i}g_{a,b}^{-1}=Q_{a_i,b_i,M_2,d_i}\ \text{and}\  g_{a,b}Q^{e_{M_1}}_{a,b,M_1,d_i}g_{a,b}^{-1}=Q^{e_{M_2}}_{a_i,b_i,M_2,d_i}$$
for all $1\leq i\leq c$. If $a=b$, as $D(M_1)=D(M_2)$, by the definitions of $D(M_1)$ and $D(M_2)$, there are $Q_{a_1,a_1,M_2,d_1},\ldots, Q_{a_c,a_c,M_2,d_c}$, where $1\leq a_i\leq m$ for all $1\leq i\leq c$. By $(8.4)$ and the Sylow Theorem, notice that there exists some $g_{a,a}\in\sym{n}$ such that $$g_{a,a}Q_{a,a,M_1,d_i}g_{a,a}^{-1}=Q_{a_i,a_i,M_2,d_i}$$ for all $1\leq i\leq c$. Also set $g_{i,j}=1$ for all $s_{i,j,M_1}=\varnothing$. Since $D(M_1)=D(M_2)$ and $U(M_1)=U(M_2)$, for any $1\leq i,j,u,v\leq m$, if $(i,j)\neq (u,v)$, by the definitions of $g_{i,j}$ and $g_{u,v}$, we can require that $\mathrm{Supp}(g_{i,j})\cap\mathrm{Supp}(g_{u,v})=\varnothing$. Set $g=\prod_{i=1}^m\prod_{j=1}^mg_{i,j}$. By the definition of $g$, note that
\begin{equation*}
\displaystyle g(\!\!\prod_{1\leq k\leq m} P_{s_{k,k,M_1}}\times \!\!\!\!\prod_{1\leq i<j\leq m}\!\!\!\!(P_{s_{i,j,M_1}}\times P^{e_{M_1}}_{s_{i,j,M_1}}))g^{-1}=\displaystyle (\!\!\prod_{1\leq k\leq m}P_{s_{k,k,M_2}}\times \!\!\!\!\prod_{1\leq i<j\leq m}\!\!\!\!(P_{s_{i,j,M_2}}\times P^{e_{M_2}}_{s_{i,j,M_2}})).
\end{equation*}

As $U(M_1)=U(M_2)$, by the definitions of $e_{M_1}$ and $e_{M_2}$, if $ge_{M_1}g^{-1}$ is a product of $h$ mutually disjoint $2$-cycles, so is $e_{M_2}$. For all $1\leq i<j\leq m$, we claim that $ge_{M_1}g^{-1}$ swaps $s_{i,j,M_2}$ and $s_{j,i,M_2}$. Let $e_{M_1}=(i_1,j_1)\cdots(i_h,j_h)$ for some $h\in \mathbb{N}$, where $\mathrm{Supp}((i_u,j_u))\cap\mathrm{Supp}((i_v,j_v))=\varnothing$ for any $1\leq u\neq v\leq h$. By the definition of $e_{M_1}$, for any $1\leq k\leq h$, note that $i_k\in s_{x,y,M_1}$ and $j_k\in s_{y,x,M_1}$ for some $x\neq y$. We have $ge_{M_1}g^{-1}=(g(i_1),g(j_1))\cdots(g(i_h),g(j_h))$. The claim follows by the definition of $g$. By this claim, both $gP_{M_1}g^{-1}$ and $P_{M_2}$ are $\sym{n}$-conjugate to Sylow $2$-subgroups of $(\prod_{k=1}^m\sym{s_{k,k,M_2}}\times\prod_{1\leq i<j\leq m}(\sym{s_{i,j,M_2}}\times \sym{s_{j,i,M_2}}))\rtimes\langle e_{M_2}\rangle.$ The lemma thus follows.
\end{proof}
\begin{lem}\label{L;conjugacy1}
Let $p=2$, $1<m\in\mathbb{N}$ and $M_1$, $M_2\in \mathcal{M}_{m,n}$. Let $M_1=M_1^t$, $M_2=M_2^t$ and $P_{M_1}$, $P_{M_2}\not\in\mathcal{Y}_n$. If $P_{M_1}$ is $\sym{n}$-conjugate to $P_{M_2}$, then $D(M_1)=D(M_2)$ and $U(M_1)=U(M_2)$.
\end{lem}
\begin{proof}
Recall the assignments of $M_1$ and $M_2$ in Notation \ref{N;notation1} (iii). Since $P_{M_1}$ is $\sym{n}$-conjugate to $P_{M_2}$, let $gP_{M_1}g^{-1}=P_{M_2}$ for some $g\in \sym{n}$. For any $x\in\{1,2\}$, set
\begin{align*}
\mathfrak{S}_{M_x}=\prod_{1\leq i,j\leq m}\!\!\sym{s_{i,j,M_x}} \ \text{and}\ R_{M_x}=\displaystyle \prod_{1\leq k\leq m}P_{s_{k,k,M_x}}\times \!\!\!\!\prod_{1\leq i<j\leq m}\!\!\!\!(P_{s_{i,j,M_x}}\times P^{e_{M_x}}_{s_{i,j,M_x}}).
\end{align*}
By $(8.1)$, for any $x\in\{1,2\}$, note that $R_{M_x}\leq \mathfrak{S}_{M_x}$, $P_{M_x}=R_{M_x}\rtimes\langle e_{M_x}\rangle$ and $|s_{i,j,M_x}|=|s_{j,i,M_x}|$ for any $1\leq i,j\leq m$. We split the whole proof into four steps.
\begin{enumerate}[\text{Step} 1:]
\item Prove that $gR_{M_1}g^{-1}=R_{M_2}$.
\end{enumerate}
For any $u\leq v$ and $s_{u,v,M_2}\neq \varnothing$, by $(8.3)$, we have $P_{s_{u,v,M_2}}=Q_{u,v,M_2,k_1}\times\cdots\times Q_{u,v,M_2,k_s}$ for unique $s\in \mathbb{N}$ and mutually distinct non-negative integers $k_1,\ldots,k_s$. We may assume that $Q_{u,v,M_2,k_i}\neq 1$ for all $1\leq i\leq s$. Therefore, for any given $1\leq i\leq s$, $Q_{u,v,M_2,k_i}=(Q_{k_i,1}\times Q_{k_i,2})\rtimes\langle w\rangle$, where $Q_{k_i,j}$ is $\sym{n}$-conjugate to a Sylow $2$-subgroup of $\sym{2^{k_i-1}}$ for any $j\in\{1,2\}$ and $wQ_{k_i,1}w=Q_{k_i,2}$ for the involution $w$. Let $\sigma_{k_i,1}$ be a $2^{k_i-1}$-cycle of $Q_{k_i,1}$ and set $\sigma_{k_i,2}=w\sigma_{k_i,1}w$. Define $\rho_{k_i}=\sigma_{k_i,1}w$ and note that $\rho_{k_i}$ is a $2^{k_i}$-cycle of $Q_{u,v,M_2,k_i}$. Moreover, $\rho_{k_i}^2=\sigma_{k_i,1}w\sigma_{k_i,1}w=\sigma_{k_i,1}\sigma_{k_i,2}$. As $gP_{M_1}g^{-1}=P_{M_2}$, we have $\rho_{k_i}=gqg^{-1}$ or $\rho_{k_i}=gqe_{M_1}g^{-1}$ for some $q\in R_{M_1}$. As $e_{M_1}\in N_{\sym{n}}(R_{M_1})$, we have $\rho_{k_i}^2=grg^{-1}$ for some $r\in R_{M_1}$. Since $\mathrm{Supp}(\sigma_{k_i,1})\cap\mathrm{Supp}(\sigma_{k_i,2})=\varnothing$, there are $r_1$, $r_2\in\sym{n}$ such that $r=r_1r_2$, $\mathrm{Supp}(r_1)\cap \mathrm{Supp}(r_2)=\varnothing$ and $\sigma_{k_i,j}=gr_jg^{-1}$ for any $j\in\{1,2\}$. As $r\in R_{M_1}$ and both $\sigma_{k_i,1}$ and $\sigma_{k_i,2}$ are $2^{k_i-1}$-cycles, note that $\mathrm{Supp}(\sigma_{k_i,1})=\mathrm{Supp}(gr_1g^{-1})\subseteq gs_{a,b,M_1}$ and $\mathrm{Supp}(\sigma_{k_i,2})=\mathrm{Supp}(gr_2g^{-1})\subseteq gs_{c,d,M_1}$ for some $1\leq a,b,c,d\leq m$. We claim that $\rho_{k_i}=ghg^{-1}$ for some $h\in R_{M_1}$. Otherwise, if $\rho_{k_i}=ghe_{M_1}g^{-1}$ for some $h\in R_{M_1}$, we shall show the following results.
\begin{enumerate}[(i)]
\item We have $a\neq b$ and $(c,d)=(b,a)$.
\item We have $|gs_{a,b,M_1}|=|gs_{b,a,M_1}|=|\mathrm{Supp}(\sigma_{k_i,1})|=2^{k_i-1}$.
\item We have $|gs_{j,k,M_1}|=|gs_{k,j,M_1}|=0$ for any $1\leq j\neq k\leq m$ and $\{j,k\}\neq\{a,b\}$.
\end{enumerate}

For (i), assume that $a=b$. Recall that $\rho_{k_i}=\sigma_{k_i,1}w=ghe_{M_1}g^{-1}$ and  $\sigma_{k_i,j}=gr_jg^{-1}$ for any $j\in\{1,2\}$. Let $w=gxg^{-1}$ for some $x\in P_{M_1}$. Note that $r_2\neq 1$. Otherwise, as $|\mathrm{Supp}(\sigma_{k_i,1})|=|\mathrm{Supp}(\sigma_{k_i,2})|$, we have $\sigma_{k_i,1}=\sigma_{k_i,2}=gr_2g^{-1}=1$ and $\rho_{k_i}$ is a transposition. As $\rho_{k_i}=ghe_{M_1}g^{-1}$, it is clear to see that $he_{M_1}=(\alpha,\beta)\in P_{M_1}$ and $h=(\alpha,\beta)e_{M_1}\in R_{M_1}$. As $P_{M_1}\not\in\mathcal{Y}_n$, by Lemma \ref{L;p=2Youngsubgroups} and the definition of $e_{M_1}$, note that $|\mathrm{Supp}(e_{M_1})|>3$. Therefore, by the definition of $e_{M_1}$, there exist $\gamma$, $\eta\in \mathbf{n}$ such that $e_{M_1}(\gamma)=\eta$, $\eta\notin\{\alpha,\beta\}$, $\gamma\in s_{\lambda,\mu,M_1}$ and $\eta\in s_{\mu,\lambda,M_1}$ for some $1\leq \lambda\neq \mu\leq m$. So $h(\gamma)=\eta$, which contradicts with the fact that $h\in R_{M_1}$. As $w^2=x^2=1$,
\begin{align*}
r_2=g^{-1}gr_2g^{-1}g=g^{-1}\sigma_{k_i,2}g=g^{-1}w\sigma_{k_i,1}wg=g^{-1}gxg^{-1}gr_1g^{-1}gxg^{-1}g=xr_1x.
\end{align*}
As $\mathrm{Supp}(\sigma_{k_i,1})\subseteq gs_{a,a,M_1}$,
$\mathrm{Supp}(\sigma_{k_i,2})\subseteq gs_{c,d,M_1}$, $\sigma_{k_i,1}=gr_1g^{-1}$ and $\sigma_{k_i,2}=gr_2g^{-1}$, we have $\mathrm{Supp}(r_1)\subseteq s_{a,a,M_1}$ and $\mathrm{Supp}(r_2)\subseteq s_{c,d,M_1}$. As $x\in P_{M_1}$, by $(8.1)$, note that $xs_{a,a,M_1}=s_{a,a,M_1}$ and $\mathrm{Supp}(r_2)=\mathrm{Supp}(xr_1x)\subseteq s_{a,a,M_1}\cap s_{c,d,M_1}$, which forces that $c=d=a$. Otherwise, as $s_{a,a,M_1}\cap s_{c,d,M_1}=\varnothing$, we get $r_2=1$, which is a contradiction. So
$\mathrm{Supp}(he_{M_1})=\mathrm{Supp}(g^{-1}\rho_{k_i}g)=\mathrm{Supp}(r_1)\cup\mathrm{Supp}(r_2)\subseteq s_{a,a,M_1}$ and $he_{M_1}\in\sym{s_{a,a,M_1}}\leq \sym{M_1}.$ As $h\in R_{M_1}\leq \sym{M_1}$, we get $e_{M_1}\in \sym{M_1}$, which contradicts with the definition of $e_{M_1}$. So $a\neq b$. As $r_2=xr_1x$, $\mathrm{Supp}(r_1)\subseteq s_{a,b,M_1}$ and $x\in P_{M_1}$, by $(8.1)$, we have $\mathrm{Supp}(r_2)\subseteq s_{a,b,M_1}$ or $\mathrm{Supp}(r_2)\subseteq s_{b,a,M_1}$. If $\{c,d\}\neq \{a,b\}$, then $s_{a,b,M_1}\cap s_{c,d,M_1}=s_{b,a,M_1}\cap s_{c,d,M_1}=\varnothing$. Hence, we get $\mathrm{Supp}(r_2)=\varnothing$ and $r_2=1$, which is absurd. So $\{c,d\}=\{a,b\}$. If $(a,b)=(c,d)$, then we have
$\mathrm{Supp}(he_{M_1})=\mathrm{Supp}(g^{-1}\rho_{k_i}g)=\mathrm{Supp}(r_1)\cup\mathrm{Supp}(r_2)\subseteq s_{a,b,M_1}$ and $he_{M_1}\in \sym{s_{a,b,M_1}}\leq \sym{M_1}$. As $h\in R_{M_1}\leq \sym{M_1}$, we also get $e_{M_1}\in \sym{M_1}$, which is absurd. So we have $(c,d)=(b,a)$.

To prove (ii), let $S=\mathrm{Supp}(\sigma_{k_i,1})$ and $T=\mathrm{Supp}(\sigma_{k_i,2})$. Recall that $S\subseteq gs_{a,b,M_1}$. So $|S|\leq |gs_{a,b,M_1}|=|gs_{b,a,M_1}|$. Assume that this inequality is strict. We pick $y\in (gs_{a,b,M_1})\setminus S$ and note that $y\not\in S\cup T=\mathrm{Supp}(\rho_{k_i})$. Since $h\in R_{M_1}$ and $e_{M_1}s_{a,b,M_1}=s_{b,a,M_1}$, we have $hs_{b,a,M_1}=s_{b,a,M_1}$ and $z=ge_{M_1}g^{-1}(y)\in gs_{b,a,M_1}$. We thus have $y=\rho_{k_i}(y)=ghg^{-1}ge_{M_1}g^{-1}(y)=ghg^{-1}(z)\in gs_{a,b,M_1}\cap gs_{b,a,M_1},$ which forces that $a=b$. By (i), this is impossible. So $|S|=|gs_{a,b,M_1}|=|gs_{b,a,M_1}|=2^{k_i-1}$.

For (iii), suppose that $|gs_{j,k,M_1}|=|gs_{k,j,M_1}|\neq 0$ for some $1\leq j\neq k\leq m$ and $\{j,k\}\neq\{a,b\}$. Then $g(s_{j,k,M_1}\cup s_{k,j,M_1})\cap g(s_{a,b,M_1}\cup s_{b,a,M_1})=\varnothing$. By (i), recall that $\mathrm{Supp}(\rho_{k_i})=\mathrm{Supp}(\sigma_{k_i,1})\cup\mathrm{Supp}(\sigma_{k_i,2})\subseteq gs_{a,b,M_1}\cup gs_{b,a,M_1}$. So there exists some $y\in gs_{j,k,M_1}$ such that $y\not\in \mathrm{Supp}(\rho_{k_i})$. As $h\in R_{M_1}$ and $e_{M_1}s_{j,k,M_1}=s_{k,j,M_1}$, we have $hs_{k,j,M_1}=s_{k,j,M_1}$ and $z=ge_{M_1}g^{-1}(y)\in gs_{k,j,M_1}$. As $j\neq k$, by mimicking the proof of (ii), we deduce that $y\in gs_{j,k,M_1}\cap gs_{k,j,M_1}=\varnothing$, which is absurd. (iii) is shown. 

By (i), (ii), (iii), the definition of $U(M_1)$ and Lemma \ref{L;p=2Youngsubgroups}, we get $P_{M_1}\in \mathcal{Y}_n$, which is a contradiction. So $\rho_{k_i}=ghg^{-1}$ for some $h\in R_{M_1}$. The claim is shown. Since $\rho_{k_i}$ is a $2^{k_i}$-cycle and $h\in R_{M_1}$, there exist some $1\leq e,f\leq m$ such that $\mathrm{Supp}(h)\subseteq s_{e,f,M_1}$. For any $\varrho\in Q_{u,v,M_2,k_i}$, if $\varrho=g\kappa e_{M_1}g^{-1}$ for some $\kappa\in R_{M_1}$, note that $\mathrm{Supp}(\kappa e_{M_1})=\mathrm{Supp}(g^{-1}\varrho g)\subseteq\mathrm{Supp}(g^{-1}\rho_{k_i}g)=\mathrm{Supp}(h)\subseteq s_{e,f,M_1}.$ In particular, $\kappa e_{M_1}\in\sym{s_{e,f,M_1}}\leq \sym{M_1}$. Since $\kappa\in R_{M_1}\leq \sym{M_1}$, we get $e_{M_1}\in\mathfrak{S}_{M_1}$, which contradicts with the definition of $e_{M_1}$. So $\varrho\in gR_{M_1}g^{-1}$. As we choose $Q_{u,v,M_2,k_i}$ arbitrarily, we have $P_{s_{u,v,M_2}}\leq gR_{M_1}g^{-1}.$ As $gP_{M_1}g^{-1}=P_{M_2}$ and $P_{M_1}\leq N_{\sym{n}}(R_{M_1})$, if $u<v$, then $g^{-1}e_{M_2}g\in P_{M_1}$ and $P^{e_{M_2}}_{s_{u,v,M_2}}\leq gR_{M_1}g^{-1}$. As $P_{s_{u,v,M_2}}$ is chosen arbitrarily, we have shown $R_{M_2}\leq gR_{M_1}g^{-1}$. As $|R_{M_1}|=|R_{M_2}|$, we have $gR_{M_1}g^{-1}=R_{M_2}$.

\begin{enumerate}[\text{Step} 2:]
\item Prove that $e_{M_2}=g\rho e_{M_1}g^{-1}$ for some $\rho\in R_{M_1}$.
\end{enumerate}
If $e_{M_2}=g\rho g^{-1}$ for some $\rho\in R_{M_1}$, by $(8.1)$ and Step $1$, we have $P_{M_2}\leq gR_{M_1}g^{-1}$, which contradicts with the equality $|P_{M_2}|=|gP_{M_1}g^{-1}|=2|R_{M_1}|$. We are done.
\begin{enumerate}[\text{Step} 3:]
\item Prove that $g(\prod_{k=1}^mP_{s_{k,k,M_1}})g^{-1}=\prod_{k=1}^mP_{s_{k,k,M_2}}$.
\end{enumerate}
It suffices to show that $\prod_{k=1}^mP_{s_{k,k,M_2}}\leq g(\prod_{k=1}^mP_{s_{k,k,M_1}})g^{-1}$. For any $s_{u,u,M_2}\neq\varnothing$, by $(8.3)$, recall that $P_{s_{u,u,M_2}}=Q_{u,u,M_2,k_1}\times\cdots\times Q_{u,u,M_2,k_s}$ for unique $s\in \mathbb{N}$ and mutually distinct non-negative integers $k_1,\ldots,k_s$. We may assume that $Q_{u,u,M_2,k_i}\neq 1$ for all $1\leq i\leq s$. For any given $Q_{u,u,M_2,k_i}$, let $\rho_{k_i}$ be a $2^{k_i}$-cycle of $Q_{u,u,M_2,k_i}$. By Steps $1$ and $2$, we have $\rho_{k_i}=gzg^{-1}$ and $e_{M_2}=g\rho e_{M_1}g^{-1}$ for some $z$, $\rho\in R_{M_1}$. As $\rho_{k_i}$ is a $2^{k_i}$-cycle, there exist some $1\leq a,b\leq m$ such that $\mathrm{Supp}(z)\subseteq s_{a,b,M_1}$. As $\rho\in R_{M_1}$ and $e_{M_1}s_{a,b,M_1}=s_{b,a,M_1}$, notice that $\rho s_{b,a,M_1}=s_{b,a,M_1}$ and $\rho e_{M_1}s_{a,b,M_1}=s_{b,a,M_1}$. If $a\neq b$, as $\rho_{k_i}=e_{M_2}\rho_{k_i}e_{M_2}=g\rho e_{M_1}g^{-1}gzg^{-1}ge_{M_1}\rho^{-1}g^{-1}=g\rho e_{M_1}ze_{M_1}\rho^{-1}g^{-1}$, we have
$\mathrm{Supp}(\rho_{k_i})=\mathrm{Supp}(gzg^{-1})=\mathrm{Supp}(g\rho e_{M_1}ze_{M_1}\rho^{-1}g^{-1})\subseteq gs_{a,b,M_1}\cap gs_{b,a,M_1}=\varnothing$. Therefore, we get $\rho_{k_i}=1$, which is a contradiction as $Q_{u,u,M_2, k_i}\neq 1$. So we have $\mathrm{Supp}(\rho_{k_i})=\mathrm{Supp}(gzg^{-1})\subseteq gs_{a,a,M_1}$. For any $\tau\in Q_{u,u,M_2,k_i}$, we have $\mathrm{Supp}(\tau)\subseteq \mathrm{Supp}(\rho_{k_i})\subseteq gs_{a,a,M_1}$. By Step $1$, we get $\tau\in gP_{s_{a,a,M_1}}g^{-1}\leq g(\prod_{k=1}^mP_{s_{k,k,M_1}})g^{-1}$. As we choose $P_{s_{u,u,M_2}}$ and $Q_{u,u,M_2,k_i}$ arbitrarily, the desired inequality is shown.
\begin{enumerate}[\text{Step} 4:]
\item Obtain the conclusion.
\end{enumerate}
By Step $1$, $(8.1)$, $(8.3)$, the definitions of $T(M_1)$ and $T(M_2)$, it is clear to see that $T(M_1)=T(M_2)$. By Step $3$, $(8.1)$, $(8.3)$, the definitions of $D(M_1)$ and $D(M_2)$, also note that $D(M_1)=D(M_2)$. Since $M_1=M_1^t$ and $M_2=M_2^t$, recall that we have $T(M_i)=D(M_i)+2U(M_i)$ for any $i\in\{1,2\}$, where the computation is the usual computation of sequences. So $U(M_1)=U(M_2)$. The proof is now complete.
\end{proof}

We are ready to deduce the main result of this section.
\begin{thm}\label{T;conjugacy}
Let $1<m\in \mathbb{N}$ and $M_1$, $M_2\in \mathcal{M}_{m,n}$. Then $P_{M_1}$ is $\sym{n}$-conjugate to $P_{M_2}$ if and only if precisely one of the following cases occurs:
\begin{enumerate}[(i)]
\item [\em (i)] $P_{M_1}$, $P_{M_2}\in \mathcal{Y}_n$ and $T'(M_1)=T'(M_2)$;
\item [\em (ii)] $P_{M_1}$, $P_{M_2}\not\in \mathcal{Y}_n$, $D(M_1)=D(M_2)$ and $U(M_1)=U(M_2)$.
\end{enumerate}
\end{thm}
\begin{proof}
If $P_{M_1}$ is $\sym{n}$-conjugate to $P_{M_2}$, it is clear that we have either $P_{M_1}$, $P_{M_2}\in \mathcal{Y}_n$ or $P_{M_1}$, $P_{M_2}\notin \mathcal{Y}_n$. When $P_{M_1}$, $P_{M_2}\notin \mathcal{Y}_n$. By $(8.1)$, we deduce that $p=2$, $M_1=M_1^t$ and $M_2=M_2^t$. Therefore, the theorem follows by Lemmas \ref{L;T'(M)}, \ref{L;conjugacy} and \ref{L;conjugacy1}.
\end{proof}
We close this section by providing an example of Theorem \ref{T;conjugacy}. In Example \ref{E;example2}, by Lemma \ref{L;p=2Youngsubgroups}, note that $P_{M_1}$, $P_{M_2}\notin \mathcal{Y}_{10}$. Since $D(M_1)\neq D(M_2)$ and $U(M_1)\neq U(M_2)$, Theorem \ref{T;conjugacy} tells us that $P_{M_1}$ is not $\sym{10}$-conjugate to $P_{M_2}$.
\section{Proof of Theorem \ref{T;D}}
In this section, we provide the proof of Theorem \ref{T;D}. Let $\lambda \vdash n$. For convenience, we shall mainly work on $\Lambda^2M^\lambda$ and show the corresponding results for $S^2M^\lambda$ similarly. For our purpose, recall the definitions in Notations \ref{N;notation} (v) and \ref{N;notation1}. Let $\sym{\lambda}\backslash\sym{n}/\sym{\lambda}$ be a fixed complete set of representatives of $(\sym{\lambda},\sym{\lambda})$-double cosets in $\sym{n}$. As usual, we require that $1\in \sym{\lambda}\backslash\sym{n}/\sym{\lambda}$.

To start our discussion, we introduce the required notation. Let $m\in \mathbb{N}$ and $M=(a_{i,j})\in \mathcal{M}_{m,n}$. We set $r(M)=(\sum_{i=1}^m a_{1,i},\ldots,\sum_{i=1}^m a_{m,i})$. Similarly, we put $c(M)=(\sum_{i=1}^ma_{i,1},\ldots,\sum_{i=1}^m a_{i,m})$. Let
$\lambda=(\lambda_1,\ldots,\lambda_\ell)\vdash n$ and $\lambda\neq(n)$. Define $\mathcal{M}_\lambda=\{M\in \mathcal{M}_{\ell,n}: \lambda=r(M)=c(M)\}$ and
$\mathbf{n}_i^\lambda=\{(\sum_{j=1}^{i-1}\lambda_j)+1,\ldots,\sum_{j=1}^{i}\lambda_j\}$ for all $1\leq i\leq \ell$. We shall repeatedly use the following theorem. It is in fact a special case of \cite[Lemma 1.3.8,\ Theorem 1.3.10]{GJ3}.
\begin{thm}\label{T;JamesKerber}
Let $\lambda\vdash n$ and $\lambda\neq (n)$. If $g\in \sym{\lambda}\backslash\sym{n}/\sym{\lambda}$ and $h\in \sym{\lambda}g\sym{\lambda}$, then $|\mathbf{n}_i^\lambda\cap h\mathbf{n}_{j}^\lambda|=|\mathbf{n}_i^\lambda\cap g\mathbf{n}_{j}^\lambda|$ for all $1\leq i,j\leq \ell(\lambda)$. Moreover, the map sending $g$ to $(|\mathbf{n}_i^\lambda\cap g\mathbf{n}_{j}^\lambda|)$ is a bijection between $(\sym{\lambda}\backslash\sym{n}/\sym{\lambda})\setminus\{1\}$ and $\mathcal{M}_\lambda$.
\end{thm}
Let $\lambda\vdash n$, $\lambda\neq (n)$ and $M=(a_{i,j})\in\mathcal{M}_\lambda$. By Theorem \ref{T;JamesKerber}, there exists the unique $g\in(\sym{\lambda}\backslash\sym{n}/\sym{\lambda})\setminus\{1\}$ such that $|\mathbf{n}_i^\lambda\cap g\mathbf{n}_{j}^\lambda|=a_{i,j}$ for all $1\leq i,j\leq \ell(\lambda)$. In particular, we have $\{t^\lambda\}\neq g\{t^\lambda\}$. Set $t_{M,s}=\{t^{\lambda}\}\odot g\{t^\lambda\}$ and $t_{M,e}=\{t^{\lambda}\}\wedge g\{t^\lambda\}$. According to Notation \ref{N;notation} (v), note that $t_{M,s}\in \mathcal{T}(\lambda)_2^s$ and $t_{M,e}\in \mathcal{T}(\lambda)_2^e$.
\begin{lem}\label{L;Orbits}
Let $\lambda\vdash n$, $\lambda\neq (n)$ and $M\in \mathcal{M}_\lambda$. Then $\mathcal{O}^e(t_{M,e})=\mathcal{O}^e(t_{M^t,e})$. In particular, $W(t_{M,e})=W(t_{M^t,e})$.
\end{lem}
\begin{proof}
Let $M=(a_{i,j})$ and $M^t=(b_{i,j})$. By Theorem \ref{T;JamesKerber}, there exists the unique $g\in(\sym{\lambda}\backslash\sym{n}/\sym{\lambda})\setminus\{1\}$ such that $t_{M,e}=\{t^\lambda\}\wedge g\{t^\lambda\}$ and $|\mathbf{n}_i^\lambda\cap g\mathbf{n}_{j}^\lambda|=a_{i,j}$ for all $1\leq i,j\leq \ell(\lambda)$. Note that $b_{i,j}=|\mathbf{n}_j^\lambda\cap g\mathbf{n}_i^\lambda|=|g^{-1}\mathbf{n}_j^\lambda\cap \mathbf{n}_i^\lambda|=|\mathbf{n}_i^\lambda\cap  g^{-1}\mathbf{n}_j^\lambda|$ for all $1\leq i,j\leq \ell(\lambda)$. By Theorem \ref{T;JamesKerber}, there exists the unique $h\in(\sym{\lambda}\backslash\sym{n}/\sym{\lambda})\setminus\{1\}$ such that $g^{-1}\in\sym{\lambda}h\sym{\lambda}$ and $t_{M^t,e}=\{t^\lambda\}\wedge h\{t^\lambda\}$. Therefore, $g^{-1}=uhv$ for some $u$, $v\in \sym{\lambda}$. Let $x=(gu)^{-1}\in\sym{n}$ and notice that $u\{t^\lambda\}=v\{t^\lambda\}=\{t^\lambda\}$. We thus have
\begin{align*}
xt_{M,e}=(u^{-1}g^{-1})(\{t^\lambda\}\wedge g\{t^\lambda\})&=u^{-1}(g^{-1}\{t^\lambda\}\wedge\{t^\lambda\})\\
&=-u^{-1}(\{t^\lambda\}\wedge g^{-1}\{t^\lambda\})\\
&=-u^{-1}(\{t^\lambda\}\wedge uhv\{t^\lambda\})\\
&=-(\{t^\lambda\}\wedge h\{t^\lambda\})=-t_{M^t,e},
\end{align*}
which shows that $\mathcal{O}^e(t_{M,e})=\O^e(t_{M^t,e})$ by the definitions of $\mathcal{O}^e(t_{M,e})$ and $\mathcal{O}^e(t_{M^t,e})$. The second equality also follows by the definitions of $W(t_{M,e})$ and $W(t_{M^t,e})$.
\end{proof}
Let $\lambda\vdash n$ and $\{t\}\in\mathcal{T}(\lambda)$. Recall that $R(\{t\})=\{g\in\sym{n}: g\{t\}=\{t\}\}$.
\begin{lem}\label{L;Orbits1}
Let $\lambda\vdash n$, $\lambda\neq (n)$ and $M_1$, $M_2\in \mathcal{M}_\lambda$. If $M_1\neq M_2$, then   we have $\O^e(t_{M_1,e})=\O^e(t_{M_2,e})$ if and only if $M_1^t=M_2$. In particular, if $M_1\neq M_2$, then we have $W(t_{M_1,e})=W(t_{M_2,e})$ if and only if $M_1^t=M_2$.
\end{lem}
\begin{proof}
By Lemma \ref{L;Orbits}, it suffices to show that the equality $\O^e(t_{M_1,e})=\O^e(t_{M_2,e})$ implies that $M_1^t=M_2$. Let $t_{M_1,e}=\{t^\lambda\}\wedge g_1\{t^\lambda\}$ and $t_{M_2,e}=\{t^\lambda\}\wedge g_2\{t^\lambda\}$, where $g_1$, $g_2\in (\sym{\lambda}\backslash\sym{n}/\sym{\lambda})\setminus\{1\}$ and $g_1\neq g_2$. If $\O^e(t_{M_1,e})=\O^e(t_{M_2,e})$, we claim that $g_1^{-1}\in \sym{\lambda}g_2\sym{\lambda}$. As $g_1\neq g_2$ and $\O^e(t_{M_1,e})=\O^e(t_{M_2,e})$, by the definitions of $\O^e(t_{M_1,e})$ and $\O^e(t_{M_2,e})$, we have $gt_{M_2,e}=-t_{M_1,e}$ for some $g\in\sym{n}$. So $g\{t^\lambda\}=g_1\{t^\lambda\}$ and $gg_2\{t^\lambda\}=\{t^\lambda\}$. Therefore, $g_1^{-1}g=u$ and $gg_2=v$, where $u$, $v\in R(\{t^{\lambda}\})=\sym{\lambda}$. We thus have   $g_1^{-1}=ug^{-1}vv^{-1}=ug^{-1}gg_2v^{-1}=ug_2v^{-1}\in\sym{\lambda}g_2\sym{\lambda}$. The claim is shown. Let $M_1=(a_{i,j})$ and $M_2=(b_{i,j})$. By this claim and Theorem \ref{T;JamesKerber}, for any $1\leq i,j\leq \ell(\lambda)$, $b_{i,j}=|\mathbf{n}_i^\lambda\cap g_1^{-1}\mathbf{n}_j^\lambda|=|g_1\mathbf{n}_i^\lambda\cap \mathbf{n}_j^\lambda|=|\mathbf{n}_j^\lambda\cap g_1\mathbf{n}_i^\lambda|=a_{j,i}$. So $M_1^t=M_2$ and the first assertion is shown. By Notation \ref{N;notation} (v), note that $W(t_{M_1,e})=W(t_{M_2,e})$ if and only if $\O^e(t_{M_1,e})=\O^e(t_{M_2,e})$. The second assertion follows by the first one.
\end{proof}
\begin{lem}\label{L;Orbits2}
Let $\lambda\vdash n$, $\lambda\neq (n)$ and $M\in \mathcal{M}_\lambda$. Let  $g\in (\sym{\lambda}\backslash\sym{n}/\sym{\lambda})\setminus\{1\}$ and $t_{M,e}=\{t^\lambda\}\wedge g\{t^\lambda\}$. Then $L(t_{M,e})=(R(\{t^\lambda\})\cap R(g\{t^\lambda\}))\rtimes\langle z\rangle$ for some involution $z\in\sym{n}$ if and only if $M=M^t$.
\end{lem}
\begin{proof}
Let $M=(a_{i,j})$ and $R=R(\{t^\lambda\})\cap R(g\{t^\lambda\})$. Note that
\begin{align}
R=\sym{\lambda}\cap g\sym{\lambda}g^{-1}=\prod_{1\leq i,j\leq\ell(\lambda)} \sym{\mathbf{n}_i^\lambda\cap g\mathbf{n}_j^\lambda}.
\end{align}
Moreover, for any $1\leq i\leq \ell(\lambda)$, the $i$th row of $\{t^\lambda\}$ contains exactly all the members of $\mathbf{n}_i^\lambda$. Similarly, the $i$th row of $g\{t^\lambda\}$ contains exactly all the members of $g\mathbf{n}_i^\lambda$. If $L(t_{M,e})=R\rtimes \langle z\rangle$ and $z^2=1$, as $z$ normalizes $R$, by $(9.1)$, given $1\leq i,j\leq \ell(\lambda)$, note that $z(\mathbf{n}_i^\lambda\cap g\mathbf{n}_j^\lambda)=\mathbf{n}_x^\lambda\cap g\mathbf{n}_y^\lambda$ for some $1\leq x,y\leq \ell(\lambda)$. Moreover, by Lemmas \ref{L;subgroup3} and \ref{L;subgroup4}, also note that $z$ swaps $\{t^\lambda\}$ and $g\{t^\lambda\}$. For any $1\leq i,j\leq\ell(\lambda)$, $\mathbf{n}_i^\lambda\cap g\mathbf{n}_j^\lambda$ lies in the $i$th row of $\{t^\lambda\}$. As $z\{t^\lambda\}=g\{t^\lambda\}$, there exists some $1\leq k\leq \ell(\lambda)$ such that $z(\mathbf{n}_i^\lambda\cap g\mathbf{n}_j^\lambda)=\mathbf{n}_k^\lambda\cap g\mathbf{n}_i^\lambda$. Since $z^2=1$, we also have  $z(\mathbf{n}_k^\lambda\cap g\mathbf{n}_i^\lambda)=\mathbf{n}_i^\lambda\cap g\mathbf{n}_j^\lambda$. This forces that $j=k$ as $\mathbf{n}_k^\lambda\cap g\mathbf{n}_i^\lambda$ is in the $k$th row of $\{t^\lambda\}$. Therefore, by the definition of $t_{M,e}$ and Theorem \ref{T;JamesKerber}, we have $a_{i,j}=|\mathbf{n}_i^\lambda\cap g\mathbf{n}_j^\lambda|=|\mathbf{n}_j^\lambda\cap g\mathbf{n}_i^\lambda|=a_{j,i}$ for any $1\leq i,j\leq \ell(\lambda)$. We thus have $M=M^t$.

Conversely, we claim that $h\{t^\lambda\}=g\{t^\lambda\}$ and $hg\{t^\lambda\}=\{t^\lambda\}$ for some $h\in \sym{n}$. As $M=M^t$, by the definition of $t_{M,e}$ and Theorem \ref{T;JamesKerber}, for any  $1\leq i,j\leq \ell(\lambda)$,
\begin{align}
|\mathbf{n}_i^\lambda\cap g\mathbf{n}_j^\lambda|=a_{i,j}=a_{j,i}=|\mathbf{n}_j^\lambda\cap g\mathbf{n}_i^\lambda|=|g^{-1}\mathbf{n}_j^\lambda\cap \mathbf{n}_i^\lambda|=|\mathbf{n}_i^\lambda\cap g^{-1}\mathbf{n}_j^\lambda|.
\end{align}
By $(9.2)$ and Theorem \ref{T;JamesKerber}, note that $g^{-1}\in \sym{\lambda}g\sym{\lambda}$. So $g^{-1}=ugv$ for some $u$, $v\in\sym{\lambda}$. Set $h=(gu)^{-1}$. As $R(\{t^\lambda\})=\sym{\lambda}$, we have
\begin{align*}
h\{t^\lambda\}=u^{-1}g^{-1}\{t^\lambda\}=u^{-1}ugv\{t^\lambda\}=g\{t^\lambda\}\ \text{and}\
hg\{t^\lambda\}=u^{-1}g^{-1}g\{t^\lambda\}=\{t^\lambda\}.
\end{align*}
Therefore, $h$ swaps $\{t^\lambda\}$ and $g\{t^\lambda\}$ and the claim is shown. By this claim and Lemma \ref{L;subgroup4}, we deduce that $L(t_{M,e})=R\rtimes \langle z\rangle$ for some involution $z\in\sym{n}$.
\end{proof}
\begin{lem}\label{L;symmetricpowercase}
Let $\lambda\vdash n$, $\lambda\neq (n)$ and $M\in \mathcal{M}_\lambda$.
\begin{enumerate}[(i)]
\item [\em (i)] Let $N\in \mathcal{M}_\lambda$. Then $V(t_{M,s})=V(t_{N,s})$ if and only if $N=M$ or $N=M^t$.
\item [\em (ii)]Let $g\in (\sym{\lambda}\backslash \sym{n}/\sym{\lambda})\setminus \{1\}$ and $t_{M,s}=\{t^\lambda\}\odot g\{t^\lambda\}$. Then we have $M=M^t$ if and only if $K(t_{M,s})=(R(\{t^\lambda\})\cap R(g\{t^\lambda\}))\rtimes\langle z\rangle$ for some involution $z\in\sym{n}$.
\end{enumerate}
\end{lem}
\begin{proof}
By mimicking the proof of Lemma \ref{L;Orbits}, we prove that $\O^s(t_{M,s})=\O^s(t_{M^t,s})$ and $V(t_{M,s})=V(t_{M^t,s})$. We now can show (i) and (ii) by mimicking the proofs given in Lemmas \ref{L;Orbits1} and \ref{L;Orbits2} respectively.
\end{proof}
Let $\lambda\vdash n$ and $\lambda\neq (n)$. For our following results, let $\mathcal{S}_\lambda=\{M\in \mathcal{M}_\lambda: M=M^t\}$. If $\mathcal{S}_\lambda\varsubsetneqq\mathcal{M}_\lambda$, note that $2\mid |\mathcal{M}_\lambda|-|\mathcal{S}_\lambda|$ and let
$2x_\lambda+|\mathcal{S}_\lambda|=|\mathcal{M}_\lambda|$. Pick $M_i\in \mathcal{M}_\lambda\setminus \mathcal{S}_\lambda$ for each $1\leq i\leq x_\lambda$ and require that $\{M_i,M_i^t\}\cap\{M_j,M_j^t\}=\varnothing$ for all $1\leq i\neq j\leq x_\lambda$.
Set $\mathcal{N}_{\lambda}=\{M_i: 1\leq i\leq x_\lambda\}$. If $\mathcal{S}_\lambda=\mathcal{M}_\lambda$, put $\mathcal{N}_\lambda=\varnothing$.
\begin{prop}\label{P;Orbits3}
Let $\lambda\vdash n$ and $\lambda\neq (n)$.
\begin{enumerate}[(i)]
\item [\em (i)] We have
$\Lambda^2M^\lambda=\bigoplus_{M\in{\mathcal{S}_\lambda\cup\mathcal{N}_\lambda}}\!\!W(t_{M,e})$.
\item [\em (ii)] If $M\in\mathcal{S}_\lambda\cup\mathcal{N}_\lambda$, then $P_M$ is $\sym{n}$-conjugate to a Sylow $p$-subgroup of $L(t_{M,e})$.
\end{enumerate}
\end{prop}
\begin{proof}
As $\lambda\neq (n)$, by the definitions of $\mathcal{S}_\lambda$ and $\mathcal{N}_\lambda$, note that $\mathcal{S}_\lambda\cup \mathcal{N}_\lambda\neq \varnothing$. Recall that $\mathcal{T}(\lambda)_2^e$ is an $\F$-basis of $\Lambda^2M^\lambda$. Set $U=\mathcal{S}_\lambda\cup \mathcal{N}_\lambda$. For any distinct $M_1$, $M_2\in U$, by Lemma \ref{L;Orbits1} and the definitions of $\mathcal{S}_\lambda$, $\mathcal{N}_\lambda$, $\O^e(t_{M_1,e})$ and $\O^e(t_{M_2,e})$, note that $\O^e(t_{M_1,e})\cap\O^e(t_{M_2,e})=\varnothing$. Therefore, as $W(t_{M,e})=\langle \O^e(t_{M,e})\rangle_\F$ for all $M\in U$, to get the desired equality, it suffices to show that $\mathcal{T}(\lambda)_2^e=\bigcup _{M\in U}\O^e(t_{M,e})$. It is clear that $\bigcup _{M\in U}\O^e(t_{M,e})\subseteq \mathcal{T}(\lambda)_2^e$. For any $t\in \mathcal{T}(\lambda)_2^e$, there exists some $g\in\sym{n}$ such that $x=\{t^\lambda\}\wedge g\{t^\lambda\}\in \mathcal{T}(\lambda)_2^e$ and $t\in \O^e(x)$. Note that $g=uhv$, where $h\in (\sym{\lambda}\backslash\sym{n}/\sym{\lambda})\setminus\{1\}$ and $u$, $v\in\sym{\lambda}$. By Theorem \ref{T;JamesKerber}, there exists $N\in \mathcal{M}_\lambda$ such that $t_{N,e}=\{t^\lambda\}\wedge h\{t^\lambda\}$. As $u$, $v\in \sym{\lambda}=R(\{t^\lambda\})$, notice that $ut_{N,e}=x$, which implies that $\O^e(x)=\O^e(t_{N,e})$ by the definitions of $\O^e(x)$ and $\O^e(t_{N,e})$. Therefore, we have $t\in \O^e(x)=\O^e(t_{N,e})\subseteq \bigcup _{M\in U}\O^e(t_{M,e})$ by Lemma \ref{L;Orbits1} and the definitions of $\mathcal{S}_\lambda$ and $\mathcal{N}_\lambda$. So $\mathcal{T}(\lambda)_2^e\subseteq\bigcup _{M\in U}\O^e(t_{M,e})$ and the desired equality thus follows.
(i) is shown.

Let $M=(a_{i,j})\in\mathcal{S}_\lambda\cup\mathcal{N}_\lambda$. By Theorem \ref{T;JamesKerber}, there exists $g\in(\sym{\lambda}\backslash\sym{n}/\sym{\lambda})\setminus\{1\}$ such that $a_{i,j}=|\mathbf{n}_i^\lambda\cap g\mathbf{n}_j^\lambda|$ for any $1\leq i,j\leq \ell(\lambda)$. So $t_{M,e}=\{t^\lambda\}\wedge g\{t^\lambda\}$. If $p>2$ or $M\neq M^t$, by Lemmas \ref{L;Orbits2}, \ref{L;subgroup3} and \ref{L;subgroup4}, note that a Sylow $p$-subgroup $P$ of $L(t_{M,e})$ is also a Sylow $p$-subgroup of $R(\{t^\lambda\})\cap R(g\{t^\lambda\})$. By $(9.1)$ and $(8.1)$, $P_M$ is $\sym{n}$-conjugate to $P$. If $p=2$ and $M=M^t$,
let $z$ be an involution of $\sym{n}$ that swaps $\mathbf{n}_i^\lambda\cap g\mathbf{n}_j^\lambda$ and $\mathbf{n}_j^\lambda\cap g\mathbf{n}_i^\lambda$ for all $1\leq i<j\leq \ell(\lambda)$ and fixes the remaining numbers of $\mathbf{n}$. Note that
\begin{align}
z\mathbf{n}_i^\lambda=z(\mathbf{n}_i^\lambda\cap g\mathbf{n})=z(\bigcup_{j=1}^{\ell(\lambda)}(\mathbf{n}_i^\lambda\cap g\mathbf{n}_j^\lambda))=\bigcup_{j=1}^{\ell(\lambda)}(\mathbf{n}_j^\lambda\cap g\mathbf{n}_i^\lambda)=\mathbf{n}\cap g\mathbf{n}_i^\lambda=g\mathbf{n}_i^\lambda.
\end{align}
As $z^2=1$, by $(9.3)$, we have $z\{t^\lambda\}=g\{t^\lambda\}$ and $zg\{t^\lambda\}=\{t^\lambda\}$. By Lemma \ref{L;subgroup4}, we thus have $L(t_{M,e})=(R(\{t^\lambda\})\cap R(g\{t^\lambda\}))\rtimes \langle z\rangle$. Therefore, by $(9.1)$, $(8.1)$ and the definition of $z$, note that $P_M$ is $\sym{n}$-conjugate to a Sylow $2$-subgroup of $L(t_{M,e})$. (ii) is also proved. The proof is now complete.
\end{proof}
\begin{prop}\label{P;Orbits4}
Let $\lambda\vdash n$.
\begin{enumerate}[(i)]
\item [\em (i)] We have
$S^2M^\lambda=V(t_\lambda)\oplus\!\bigoplus_{M\in{\mathcal{S}_\lambda\cup\mathcal{N}_\lambda}}\!\!V(t_{M,s}),$
where $t_\lambda=\{t^\lambda\}\odot\{t^\lambda\}$ and $K(t_\lambda)=\sym{\lambda}$.
\item [\em (ii)]If $\mathcal{S}_\lambda\cup \mathcal{N}_\lambda\neq \varnothing$ and $M\in\mathcal{S}_\lambda\cup\mathcal{N}_\lambda$, then $P_M$ is $\sym{n}$-conjugate to a Sylow $p$-subgroup of $K(t_{M,s})$.
\end{enumerate}
\end{prop}
\begin{proof}
By the definitions of $\mathcal{S}_\lambda$ and $\mathcal{N}_\lambda$, note that $\mathcal{S}_\lambda\cup \mathcal{N}_\lambda=\varnothing$ if and only if $\lambda=(n)$. As $S^2M^{(n)}\cong \F$, the proposition holds for $S^2M^{(n)}$. We thus assume that $\lambda\neq (n)$. Therefore, $\mathcal{S}_\lambda\cup \mathcal{N}_\lambda\neq \varnothing$. Observe that $\O^s(t_\lambda)\cap \O^s(t_{M,s})=\varnothing$ for any $M\in\mathcal{S}_\lambda\cup \mathcal{N}_\lambda$ and $K(t_\lambda)=\sym{\lambda}$. By Lemma \ref{L;symmetricpowercase} (i) and mimicking the proof of Proposition \ref{P;Orbits3} (i), (i) is shown. By Lemma \ref{L;symmetricpowercase} (ii) and mimicking the proof of Proposition \ref{P;Orbits3} (ii), (ii) is also proved. The proof is now complete.
\end{proof}
\begin{lem}\label{L;Homlemma}
Let $\lambda\vdash n$ and $\lambda\neq (n)$. If $M\in \mathcal{S}_\lambda\cup\mathcal{N}_\lambda$, then
\begin{align*}
\dim_\F\mathrm{Hom}_{\F\sym{n}}(\F, W(t_{M,e}))=\begin{cases}
0, & \text{if}\ p>2\ \text{and}\ M\in \mathcal{S}_\lambda,\\
1, &\text{otherwise}.
\end{cases}
\end{align*}
\end{lem}
\begin{proof}
As $M\in \mathcal{S}_\lambda\cup \mathcal{N}_\lambda\subseteq \mathcal{M}_\lambda$, by Theorem \ref{T;JamesKerber} and the definition of $t_{M,e}$, there exists $g\in (\sym{\lambda}\backslash \sym{n}/\sym{\lambda})\setminus\{1\}$ such that $t_{M,e}=\{t^\lambda\}\wedge g\{t^\lambda\}$. Let $N=\langle \{t_{M,e}\}\rangle_\F$ and note that $N$ is a one-dimensional $\F L(t_{M,e})$-module. As $W(t_{M,e})=\langle \O^e(t_{M,e})\rangle_\F$, also notice that $W(t_{M,e})\cong N{\uparrow^{\sym{n}}}$. If $p>2$ and $M\in \mathcal{S}_\lambda$, by Lemma \ref{L;Orbits2}, there exists some $z\in\sym{n}$ such that $z\in L(t_{M,e})$, $z^2=1$, $z\{t^\lambda\}=g\{t^\lambda\}$ and $zg\{t^\lambda\}=\{t^\lambda\}$. We thus have $zt_{M,e}=-t_{M,e}$ and $N\not\cong \F$ as $\F L(t_{M,e})$-modules. By the Frobenius reciprocity, we get $\dim_\F\mathrm{Hom}_{\F\sym{n}}(\F, W(t_{M,e}))=0$.

If $M\in \mathcal{N}_\lambda$, note that $M\neq M^t$. By Lemmas \ref{L;Orbits2}, \ref{L;subgroup3} and \ref{L;subgroup4}, also note that $L(t_{M,e})=R(\{t^\lambda\})\cap R(g\{t^\lambda\})$. So $N\cong \F$ as $\F L(t_{M,e})$-modules. By the Frobenius reciprocity, notice that $\dim_\F\mathrm{Hom}_{\F\sym{n}}(\F, W(t_{M,e}))=1$. If $p=2$, then it is clear that $W(t_{M,e})\cong (\F_{L(t_{M,e})}){\uparrow^{\sym{n}}}$. We thus have $\dim_\F\mathrm{Hom}_{\F\sym{n}}(\F, W(t_{M,e}))=1$ by the Frobenius reciprocity. The lemma follows by the three cases.
\end{proof}
\begin{cor}\label{C;hom}
Let $\lambda\vdash n$.
\begin{enumerate}[(i)]
\item [\em (i)] We have $\dim_\F\mathrm{Hom}_{\F\sym{n}}(\F, S^2M^\lambda)=1+|\mathcal{S}_\lambda|+|\mathcal{N}_\lambda|$.
\item [\em(iii)] If $\lambda\neq (n)$, then we have \[\dim_\F\mathrm{Hom}_{\F\sym{n}}(\F, \Lambda^2M^\lambda)=\begin{cases}
|\mathcal{N}_\lambda|, & \text{if}\ p>2,\\
|\mathcal{S}_\lambda|+|\mathcal{N}_\lambda|, &\text{if}\ p=2.
\end{cases}\]
\end{enumerate}
\end{cor}
\begin{proof}
By the definitions of $\mathcal{S}_\lambda$ and $\mathcal{N}_\lambda$, we can observe that $|\mathcal{S}_\lambda\cup \mathcal{N}_\lambda|=|\mathcal{S}_\lambda|+|\mathcal{N}_\lambda|$. For (i), for any $t\in\mathcal{T}(\lambda)_2^s$, recall that $V(t)\cong (\F_{K(t)}){\uparrow^{\sym{n}}}$. Therefore, (i) follows by Proposition \ref{P;Orbits4} (i) and the Frobenius reciprocity. (ii) is also shown by Proposition \ref{P;Orbits3} (i) and Lemma \ref{L;Homlemma}.
\end{proof}
Let $\lambda=(\lambda_1,\ldots,\lambda_\ell)\vdash n$ and $\mathcal{R}_\lambda=\mathcal{S}_\lambda\cup\mathcal{N}_\lambda\cup\{D_\lambda\}$,  where $D_\lambda=\mathrm{diag}(\lambda_1,\ldots,\lambda_\ell)$. Let $P_{D_\lambda}$ be a Sylow $p$-subgroup of $\sym{\lambda}$ and $T'(D_\lambda)=T(D_\lambda)$. We shall define an equivalence relation $\sim$ for $\mathcal{R}_{\lambda}$. For any $M$, $N\in \mathcal{R}_\lambda$, write $M\sim N$ if and only if $P_M$ is $\sym{n}$-conjugate to $P_N$. Let $[M]_{R_\lambda}$ be the equivalence class of $M$ with respect to $\sim$. Set $\mathcal{X}_\lambda=\mathcal{S}_\lambda\cup\mathcal{N}_\lambda$ if $p=2$. Otherwise, put $\mathcal{X}_\lambda=\mathcal{N}_\lambda$. If $\mathcal{X}_\lambda\neq\varnothing$, as $\mathcal{X}_\lambda\subseteq \mathcal{R}_\lambda$, the restriction of $\sim$ is also an equivalence relation for $\mathcal{X}_\lambda$. For any $R\in \mathcal{X}_\lambda$, let $[R]_{\mathcal{X}_\lambda}=\{X\in \mathcal{X}_\lambda: X\sim R\}$. Note that $[R]_{\mathcal{X}_\lambda}$ is the equivalence class of $R$ with respect to the restriction of $\sim$. By $(8.1)$, Lemma \ref{L;p=2Youngsubgroups} and Theorem \ref{T;conjugacy}, it is easy to compute both $|[M]_{\mathcal{R}_\lambda}|$ and $|[R]_{\mathcal{X}_\lambda}|$. The main result of this section is as follows.
\begin{thm}\label{T;multiplicity}
Let $\lambda\vdash n$ and $P$ be a $p$-subgroup of $\sym{n}$.
\begin{enumerate}[(i)]
\item [\em (i)] We have $Sc_{\sym{n}}(P)\mid S^2M^\lambda$ if and only if $P$ is $\sym{n}$-conjugate to $P_M$ for some $M\in \mathcal{R}_\lambda$. Moreover, for any $M\in \mathcal{R}_\lambda$, $[S^2M^\lambda:Sc_{\sym{n}}(P_M)]=|[M]_\mathcal{R_\lambda}|$.
\item [\em (ii)]If $\lambda\neq (n)$ and $\mathcal{X}_\lambda\neq \varnothing$, $Sc_{\sym{n}}(P)\mid \Lambda^2M^\lambda$ if and only if $P$ is $\sym{n}$-conjugate to $P_M$ for some $M\in \mathcal{X}_\lambda$. Moreover, for any $M\in \mathcal{X}_\lambda$, $[\Lambda^2M^\lambda:Sc_{\sym{n}}(P_M)]=|[M]_{\mathcal{X}_\lambda}|$.
\item [\em (iii)] If $\lambda\neq (n)$, then $\mathcal{X}_\lambda=\varnothing$ if and only if $p>2$ and $\ell(\lambda)=2$. If $p>2$ and $\ell(\lambda)=2$, then $[\Lambda^2M^\lambda:Sc_{\sym{n}}(Q)]=0$ for any $p$-subgroup $Q$ of $\sym{n}$.
\end{enumerate}
\end{thm}
\begin{proof}
By Proposition \ref{P;Orbits4} (i) and the Krull-Schmidt Theorem, we observe that $Sc_{\sym{n}}(P)\mid S^2M^\lambda$ if and only if $Sc_{\sym{n}}(P)\mid V(t_\lambda)$ or $Sc_{\sym{n}}(P)\mid V(t_{N,s})$ for some $N\in \mathcal{S}_\lambda\cup\mathcal{N}_\lambda$. As $K(t_\lambda)=\sym{\lambda}$ by Proposition \ref{P;Orbits4} (i) and $V(t_\lambda)\cong (\F_{K(t_\lambda)}){\uparrow^{\sym{n}}}$, note that $Sc_{\sym{n}}(P)\mid V(t_\lambda)$ if and only if $P$ is $\sym{n}$-conjugate to $P_{D_\lambda}$. For any $U\in \mathcal{S}_\lambda\cup\mathcal{N}_\lambda$, as $V(t_{U,s})\cong (\F_{K(t_{U,s})}){\uparrow^{\sym{n}}}$, by Proposition \ref{P;Orbits4} (ii), also note that $Sc_{\sym{n}}(P)\mid V(t_{U,s})$ if and only if $P$ is $\sym{n}$-conjugate to $P_U$. As $\mathcal{R}_\lambda=\mathcal{S}_\lambda\cup \mathcal{N}_\lambda\cup\{D_\lambda\}$, the first statement of (i) is shown. The second statement of (i) follows by Proposition \ref{P;Orbits4} (i), (ii) and the definition of $[M]_{\mathcal{R}_\lambda}$.

For (ii), if $p>2$ and $U\in \mathcal{S}_\lambda$, by Lemma \ref{L;Homlemma}, note that $W(t_{U,e})$ does not have a trivial $\F\sym{n}$-submodule. As $Sc_{\sym{n}}(P)$ contains a trivial $\F\sym{n}$-submodule, we thus have $Sc_{\sym{n}}(P)\nmid W(t_{U,e})$. If $U\in \mathcal{N}_\lambda$, by Lemma \ref{L;Homlemma}, note that $W(t_{U,e})\cong (\F_{L(t_{U,e})}){\uparrow^{\sym{n}}}$. If $p=2$ and $U\in \mathcal{X}_\lambda$, also note that $W(t_{U,e})\cong (\F_{L(t_{U,e})}){\uparrow^{\sym{n}}}$. By Proposition \ref{P;Orbits3} (i), the Krull-Schmidt Theorem and the three facts, notice that $Sc_{\sym{n}}(P)\mid \Lambda^2M^\lambda$ if and only if $Sc_{\sym{n}}(P)\mid W(t_{N,e})$ for some $N\in \mathcal{X}_\lambda$. For any $U\in \mathcal{X}_\lambda$, the last two facts also tell us that $W(t_{U,e})\cong (\F_{L(t_{U,e})}){\uparrow^{\sym{n}}}$. Therefore, by Proposition \ref{P;Orbits3} (ii), note that $Sc_{\sym{n}}(P)\mid W(t_{U,e})$ if and only if $P$ is $\sym{n}$-conjugate to $P_U$. The first statement of (ii) is shown. The second statement of (ii) follows by Proposition \ref{P;Orbits3} (i), (ii), the first statement of (ii) and the definition of $[M]_{\mathcal{X}_\lambda}$.

For (iii), if $p>2$ and $\ell(\lambda)=2$, by the definitions of $\mathcal{N}_\lambda$, $\mathcal{M}_\lambda$ and $\mathcal{S}_\lambda$, note that $\mathcal{X}_\lambda=\mathcal{N}_\lambda\subseteq\mathcal{M}_\lambda=\mathcal{S}_\lambda$ and $\mathcal{S}_\lambda\cap \mathcal{N}_\lambda=\varnothing$. So $\mathcal{X}_\lambda=\varnothing$. Conversely, for the case $p=2$, we have $\mathcal{X}_\lambda=\mathcal{S}_\lambda\cup \mathcal{N}_\lambda$. As $\lambda\neq (n)$, Proposition \ref{P;Orbits3} (i) forces that $\mathcal{X}_\lambda\neq \varnothing$. For the case $p>2$, we have $\mathcal{X}_\lambda=\mathcal{N}_\lambda$. By the definitions of $\mathcal{M}_\lambda$, $\mathcal{N}_\lambda$ and $\mathcal{S}_\lambda$, the condition $\mathcal{X}_\lambda=\varnothing$ forces that $\mathcal{M}_\lambda=\mathcal{S}_\lambda$. It is not too difficult to see that $\mathcal{M}_\lambda=\mathcal{S}_\lambda$ if and only if $\ell(\lambda)=2$. The first statement of (iii) follows. If $p>2$ and $\ell(\lambda)=2$, by the first statement of (iii), we have   $\mathcal{X}_\lambda=\mathcal{N}_\lambda=\varnothing$. By Corollary \ref{C;hom} (ii), this fact implies that $\Lambda^2M^\lambda$ has no trivial $\F\sym{n}$-submodule. As any $\F\sym{n}$-Scott module contains a trivial $\F\sym{n}$-submodule, the second statement of (iii) follows.
\end{proof}
Theorem \ref{T;D} is established by Theorem \ref{T;multiplicity}. We end this paper with an example.

Let $p=2$, $n=4$ and $\lambda=(2,1^2)\vdash 4$. Note that $\mathcal{M}_\lambda\cup\{D_\lambda\}$ contains precisely
\begin{align*}
& M_1=\begin{pmatrix}
2& 0& 0\\
0 & 0 & 1\\
0 & 1 & 0\\
\end{pmatrix},\
M_2=\begin{pmatrix}
1& 1 & 0\\
1 & 0 & 0\\
0 & 0 & 1\\
\end{pmatrix},\
M_3=\begin{pmatrix}
1& 1 & 0\\
0 & 0 & 1\\
1 & 0 & 0\\
\end{pmatrix},\
M_4=\begin{pmatrix}
1& 0 & 1\\
1 & 0 & 0\\
0 & 1 & 0\\
\end{pmatrix},\\
& M_5=\begin{pmatrix}
1& 0 & 1\\
0 & 1 & 0\\
1 & 0 & 0\\
\end{pmatrix},\
M_6=\begin{pmatrix}
0& 1 & 1\\
1 & 0 & 0\\
1 & 0 & 0\\
\end{pmatrix},\
M_7=\begin{pmatrix}
2& 0 & 0\\
0 & 1 & 0\\
0 & 0 & 1\\
\end{pmatrix}.\
\end{align*}
Also note that $\mathcal{S}_\lambda=\{M_1,M_2,M_5,M_6\}$ and let $\mathcal{N}_\lambda=\{M_3\}$. By $(8.1)$, Lemma \ref{L;p=2Youngsubgroups} and Theorem \ref{T;conjugacy}, $\{M_1\}$, $\{M_2,M_5,M_7\}$, $\{M_3\}$, $\{M_6\}$ are all the equivalence classes of $\mathcal{R}_\lambda$ with respect to $\sim$. Similarly, $\{M_1\}$, $\{M_2,M_5\}$, $\{M_3\}$, $\{M_6\}$ are all the equivalence classes of $\mathcal{X}_\lambda$ with respect to $\sim$. Let $Sc_i=Sc_{\sym{4}}(P_{M_i})$ for all $1\leq i\leq 7$ and note that $Sc_2\cong Sc_5\cong Sc_7$. Let $m_{i,s}=[S^2M^{(2,1^2)}:Sc_i]$ and $m_{i,e}=[\Lambda^2M^{(2,1^2)}:Sc_i]$. By Theorem \ref{T;multiplicity}, we have
$$m_{2,s}=3,\ m_{2,e}=2\ \text{and}\ m_{1,s}=m_{3,s}=m_{6,s}=m_{1,e}=m_{3,e}=m_{6,e}=1.$$
\subsection*{Acknowledgement}
The author thanks his supervisor Dr. Kay Jin Lim for some suggestions of refining this paper. He also thanks Professor Ping Jin and Professor Gang Chen for their constant encouragement.


\begin{thebibliography}{99}
\bibitem{JAlperin1}J. L. Alperin, Local Representation Theory, Cambridge Studies in Advanced Mathematics, $\mathbf{11}$, Cambridge University Press, 1986.
\bibitem{JAlperin}J. L. Alperin, L. Evens, Representations, resolutions and Quillen's dimension theorem,
J. Pure Appl. Algebra $\mathbf{22}$ (1981), 1-9.
\bibitem{Benson} D. J. Benson, Representations of elementary abelian p-groups and vector bundles,
Cambridge Tracts in Mathematics, $\mathbf{208}$, Cambridge University Press, 2017.
\bibitem{BensonLim} D. J. Benson, K. J. Lim, Generic Jordan type of the symmetric and exterior powers, J. Algebra Appl. $\mathbf{13}$ (2014), Article ID 1350163.
\bibitem{Broue} M. Brou\'{e}, On Scott modules and p-permutation modules: an approach through the Brauer morphism, Proc. Amer. Math. Soc. $\mathbf{93}$ (1985), 401-408.
\bibitem{Burry} D. W. Burry, Scott modules and lower defect groups, Comm. Algebra $\mathbf{10}$ (1982), 1855-1872.
\bibitem{CHN} F. R. Cohen, D. J. Hemmer, D. K. Nakano, On the cohomology of Young modules for the symmetric group, Adv. Math. $\mathbf{224}$ (2010), 1419-1461.
\bibitem{Donkin} S. Donkin, On tilting modules for algebraic groups, Math. Z. $\mathbf{212}$ (1993), 39-60.
\bibitem{Donkin2} S. Donkin, Symmetric and exterior powers, linear source modules and representations of Schur superalgebras, Proc. London Math. Soc. $\mathbf{83}$ (2001), 647-680.
\bibitem{GLDM} E. Giannelli, K. J. Lim, W. O'Donovan, M. Wildon, On signed Young permutation modules and signed p-Kostka numbers, J. Group Theory $\mathbf{20}$ (2017), 637-679.
\bibitem{Gill} C. C. Gill, Young module multiplicities, decomposition numbers and the indecomposable Young permutation modules, J. Algebra Appl. $\mathbf{13}$ (2014), Artical ID 1350147.
\bibitem{GJ} J. Grabmeier, Unzerlegbare Moduln mit trivialer Youngquelle und Darstellungstheorie der Schuralgebra, Bayreuth. Math. Schr. $\mathbf{20}$ (1985), 9-152.
\bibitem{JGreen} J. A. Green, On the indecomposable representations of a finite group, Math. Z. $\mathbf{70}$ (1959), 430-445.
\bibitem{DHDN} D. J. Hemmer, D. K. Nakano, Support varieties for modules over symmetric groups, J. Algebra $\mathbf{254}$ (2002), 422-440.
\bibitem{GJ1} G. D. James, The Representation Theory of the Symmetric Groups, Lecture Notes in Mathematics, $\mathbf{682}$, Springer, Berlin, 1978.
\bibitem{GJ2} G. D. James, Trivial source modules for symmetric groups, Arch. Math. (Basel) $\mathbf{41}$ (1983), 294-300.
\bibitem{GJ3} G. D. James, A. Kerber, The Representation Theory of the Symmetric Group,  Encyclopedia of Mathematics and its Applications, $\mathbf{16}$, Addison-Wesley Publishing Co., 1981.
\bibitem{JLW} Y. Jiang, K. J. Lim, J. L. Wang, On the Brauer constructions and generic Jordan types of Young modules, arXiv: 1707.04075 [math. RT], 2017.
\bibitem{Jiang} Y. Jiang, On the complexities of some simple modules of symmetric groups, Beitr. Algebra Geom. $\mathbf{60}$ (2019), 599-625.
\bibitem{Lim} K. J. Lim, Straightening rule for an $m'$-truncated polynomial ring, J. Algebra $\mathbf{522}$ (2019), 11-30.
\bibitem{HNYT} H. Nagao, Y.  Tsushima, Representations of Finite Groups, Academic Press, San Diego, 1989.
\end{thebibliography}
\end{document}